\documentclass[reqno,12pt]{amsart}
\usepackage{amsmath}
\usepackage{amsxtra}
\usepackage{amscd}
\usepackage{amsthm}
\usepackage{amsfonts}
\usepackage{amssymb}
\usepackage{eucal}
\usepackage[matrix,arrow,curve]{xy}
\usepackage{arydshln}
\usepackage{fancybox}
\usepackage{ascmac}
\usepackage{color}
\usepackage[dvipdfmx]{graphicx,xcolor}

\numberwithin{equation}{section}
\allowdisplaybreaks

 \newtheorem{dfn}{Definition}[section]
 \newtheorem{thm}[dfn]{Theorem}
 \newtheorem{prp}[dfn]{Proposition}
 \newtheorem{lem}[dfn]{Lemma}
 \newtheorem{cor}[dfn]{Corollary}
 
 \newtheorem{rmk}[dfn]{Remark}
  \newtheorem{con}[dfn]{Conjecture}

 \newtheorem*{ack}{Acknowledgements}

\newcommand{\Tr}{\operatorname{Tr}}

\newcommand{\seteq}{\mathbin{:=}}

\newcommand{\G}{{\mathsf G}}
\newcommand{\Lv}{{\mathsf L}}

\newcommand{\I}{{\mathsf I}}

\newcommand{\Svir}{{ {\mathcal{SV}ir\!}_{q,k} }}

\newcommand{\ignore}[1]{}

\newcommand{\no}{
{
{\!\!\begin{array}{c}
\scriptstyle \circ\\[-2mm]
\scriptstyle\circ\end{array}\!\!}
}}

\newcommand{\beq}{\begin{equation}}
\newcommand{\eeq}{\end{equation}}
\newcommand{\beqa}{\begin{eqnarray}}
\newcommand{\eeqa}{\end{eqnarray}}
\newcommand{\CR}{\nonumber \\}

\newcommand{\gq}{\mathfrak{q}}
\newcommand{\gd}{\mathfrak{d}}

\title{
A quantum deformation of the ${\mathcal N}=2$ superconformal algebra}

\author{Hidetoshi~Awata}
\author{Koichi~Harada}
\author{Hiroaki~Kanno}
\author{Jun'ichi~Shiraishi}
\address{H.Awata:~Graduate School of Mathematics, Nagoya University, Nagoya 464-8602, Japan}
\email{awata@math.nagoya-u.ac.jp}
\address{K.Harada: Department of Physics, The University of Tokyo, Hongo, Tokyo 113-0033 Japan}
\email{kkhharada@gmail.com}
\address{H.Kanno:~Graduate School of Mathematics and Kobayashi-Maskawa Institute, Nagoya University, Nagoya 464-8602, Japan}
\email{kanno@math.nagoya-u.ac.jp}
\address{J.Shiraishi: Graduate School of Mathematical Sciences, The University of Tokyo, Komaba, Tokyo 153-8914, Japan}
\email{shiraish@ms.u-tokyo.ac.jp}

\begin{document}
\begin{flushright}
Revised Version
\end{flushright}

\setlength\dashlinedash{1.4pt}
\setlength\dashlinegap{1.3pt}

\bigskip

\begin{abstract}
We introduce a unital associative algebra $\Svir$, 
having $q$ and $k$ as complex parameters, generated by the elements $K^\pm_m$ ($\pm m\geq 0$),
$T_m$ ($m\in  \mathbb{Z}$), and $G^\pm_m$ ($m\in  \mathbb{Z}+{1\over 2}$ in the Neveu-Schwarz sector, 
$m\in  \mathbb{Z}$ in the Ramond sector), satisfying relations which are at most quartic. 
Calculations of some low-lying Kac determinants are made, 
providing us with a conjecture for the factorization property of the Kac determinants. 
The analysis of the screening operators gives a supporting 
evidence for our conjecture.
It is shown that by taking the limit $q\rightarrow 1$ of $\Svir$ we recover the 
ordinary ${\mathcal N}=2$ superconformal algebra. 
We also give a nontrivial Heisenberg representation of the algebra $\Svir$, 
making a twist of the $U(1)$ boson in the  Wakimoto representation of the quantum affine algebra 
$U_q(\widehat{\mathfrak{sl}}_2)$, which naturally follows 
from the construction of $\Svir$ by gluing the deformed $Y$-algebras 
of Gaiotto and Rap$\check{\mathrm{c}}$\'{a}k.
\end{abstract}

\bigskip

\maketitle

\setcounter{tocdepth}{1}
\tableofcontents


\section{Introduction}

The ${\mathcal N}=2$ superconformal algebra (SCA) is a supersymmetric extension 
of the Virasoro algebra with a pair of supercurrents $G^{\pm}(z)$. 
We already have a huge list of literature on various aspects of the ${\mathcal N}=2$ SCA.
In 1980's the ${\mathcal N}=2$ SCA was intensively explored, since a compactification of the type II superstring theory on Calabi-Yau 3-folds
was expected to provide a unified theory of elementary particles including gravity (\cite{Polchinsky} and references therein).
The representation theory of the ${\mathcal N}=2$ SCA provided useful tools
in the algebraic approach to the compactification on Calabi-Yau 3-folds.
In mathematical physics, through the idea of the chiral ring \cite{Lerche:1989uy} and the topological twist \cite{Eguchi:1990vz},
the ${\mathcal N}=2$ SCA is also closely related to the mirror symmetry 
of Calabi-Yau 3-folds \cite{Witten:1991zz}.

On the other hand, there is a quantum deformation of the Virasoro algebra, which is a unital associative algebra
with deformation parameters $q$ and $t$. We also have quantum deformations 
of the $\mathcal{W}$-algebras \cite{Awata:1995zk}, \cite{FR1995}, \cite{Feigin:1995sf},
as its generalizations with higher spin currents. 
In \cite{Shiraishi:1995rp} the deformed Virasoro algebra was discovered by looking for an associative algebra 
whose singular vectors in the Verma module are related to the Macdonald symmetric polynomials with parameters $(q,t)$. 
In \cite{FR1995} it was constructed as a Poisson algebra by using the Wakimoto realization of the quantum affine algebra 
at the critical level. 
Physically it controls an off-critical deformation of two dimensional lattice model. 
For example the deformed Virasoro algebra appears as a symmetry of the Andrews-Baxter-Forestor model \cite{Lukyanov:1996qs}.
It also plays an important role in the investigation
of quantum integrable system arising from massive deformations of conformal field theories.
Namely the deformed Virasoro algebra ensures the integrability, or the existence of
infinitely many conserved quantities. 
From the viewpoint of symmetry, we can regard it as a kind of elliptic algebras \cite{Odake:1999un},
in the sense that the ratio of the structure functions of the algebra is an elliptic function. 
One can also see that the screening currents satisfy an elliptic analogue of the Drinfeld relations for the 
quantum affine algebras \cite{Feigin:1995sf}. 
Later these deformed algebras appeared in the proposal of a five dimensional uplift of AGT correspondence \cite{Awata:2009ur}.
The relation to the quantum toroidal algebra of type $\mathfrak{gl}_1$ (Ding-Iohara-Miki algebra) 
was also clarified \cite{FHSSY}.

It is natural to combine these two generalizations of the Virasoro algebra, namely to try to find a quantum deformation of
the ${\mathcal N}=2$ SCA. Surprisingly enough, up to now there are no literatures 
on such an algebra as far as we know. 
In this paper we propose a quantum deformation of the ${\mathcal N}=2$ SCA, which we denote $\Svir$ for short,
as a unital associative algebra with two parameters $q$ and $k$, where the second parameter $k$ comes from
the level of the quantum affine algebra $U_q(\widehat{\mathfrak{sl}}_2)$. 
In this paper, we assume that $q$ and $k$ are generic. The level is not critical $k \neq -2$ and $q$ is not a root of unity.
We would like to emphasize that the defining relations of $\Svir$ are at most quartic in the generators,
which has not been appreciated for a long time. We also show that the limit $q \to 1$ of $\Svir$ correctly 
reproduces the original ${\mathcal N}=2$ SCA with the central charge $c=\frac{3k}{k+2}$. 

\smallskip


We denote by $\mathbb{Z}^\mathrm{R}\seteq \mathbb{Z}$ the set of integers, and 
by $\mathbb{Z}^\mathrm{NS}\seteq \mathbb{Z}+{1\over 2}$ the
set of half-integers 
$\left\{ n+{1\over 2}|n\in \mathbb{Z}\right\}$, where $\mathrm{R}$ and $\mathrm{NS}$ stand 
for the Ramond and the Neveu-Schwarz sector, respectively.
Let $q$ and $k$ be generic complex numbers satisfying the conditions
$|q|\leq 1$ and $|q^k|\leq 1$.
We use the standard notation $[u]$ for the $q$-number 
$[u]=(q^u-q^{-u})/(q-q^{-1})$.

\subsection{Definition of the algebra $\Svir$}

\begin{dfn}\label{DefSvir}
The quantum deformation of the ${\mathcal N}=2$ superconformal algebra, 
which we call $\Svir$,  
in the Neveu-Schwarz {\rm (NS)} sector (or in the Ramond {\rm (R)} sector) is 
defined to be 
the unital associative algebra generated by  the elements
\begin{align*}
&K^\pm_m \qquad (\pm m\in \mathbb{Z}_{\geq 0}),
\qquad T_m \qquad (m \in \mathbb{Z}),\\
&G^\pm_m \qquad (m \in \mathbb{Z}+{1\over 2}\mbox{ for {\rm NS} sector}, \,\,
m \in \mathbb{Z}\mbox{ for {\rm R} sector}),
\end{align*}
satisfying the set of relations (\ref{r-1})-(\ref{r-11}) below for the sector {\rm A} {\rm (}$=$~{\rm NS} or {\rm R}{\rm )}:
\begin{align}
&K^+_0=K^-_0,\mbox{ and\,\,  } K^\pm_0 \mbox{ are invertible},\label{r-1}\\
& K^\pm_m K^\pm_n=K^\pm_m K^\pm_n,\label{r-2}\\
&K^-_m K^+_n=K^+_n K^-_m-(q-q^{-1})^2
\sum_{\ell= 1}^{-m}{[k][k+2][2\ell]\over [2]}K^+_{n-\ell}K^-_{m+\ell},\label{r-3}\\
&K^-_mT_n  =T_n K^-_m+(q-q^{-1})^2
\sum_{\ell=1}^{-m}{[k+2][(k+3)\ell]\over [k+3]}T_{n-\ell}K^-_{m+\ell},\label{r-4}\\
& T_mK^+_n=K^+_nT_m +(q-q^{-1})^2
\sum_{\ell= 1}^n {[k+2][(k+3)\ell]\over [k+3]}K^+_{n-\ell}T_{m+\ell},\label{r-5}\\
&K^-_{m} G^\pm_n =q^{\pm (k+2)} G^\pm_n K^-_{m}
\pm (q-q^{-1}) [k+2]
\sum_{\ell= 1}^{-m}q^{\pm 2(k+2)\ell }G^\pm_{n-\ell}K^-_{m+\ell},\label{r-6}\\
&G^\pm_m K^+_{n} =q^{\mp (k+2)}K^+_{n}G^\pm_m 
\mp (q-q^{-1}) [k+2]
\sum_{\ell= 1}^n q^{\mp 2(k+2)\ell } K^+_{n-\ell}G^\pm_{m+\ell},\label{r-7}\\
&G^\pm_mG^\pm_n+G^\pm_mG^\pm_n=0,\label{r-8}\\
&
G^+_mG^-_n+G^-_nG^+_m = 
{1\over (q-q^{-1})^2}\Biggl(
{1\over [k+1]}q^{(2k+2)(m-n)}
\sum_{\alpha,\beta\geq 0}
\delta_{m+n+\alpha-\beta,0}K^-_{-\alpha}K^+_{\beta} \nonumber\\
&~~~\hspace{6cm}
-q^{(k+2)(m-n)}\delta_{m+n,0}
+q^{(k+1)(m-n)} T_{m+n} \Biggr),\label{r-9} \\
&G^\pm_mT_n-T_nG^\pm_m
=\pm (q-q^{-1}){[k+2]\over[k+1]}q^{\pm (k+1)(2m-n)}
\sum_{\alpha,\beta\geq 0}
K^-_{-\alpha}G^\pm_{m+n+\alpha-\beta}K^+_{\beta},\label{r-10}\\
&T_mT_n-T_nT_m
=(q-q^{-1})^2{[k+2]^2\over [k+1]}[(m-n)(k+1)]
\sum_{\alpha,\beta\geq 0}
K^-_{-\alpha}W_{m+n+\alpha-\beta}K^+_{\beta}. \label{r-11}
\end{align}
In the last relation (\ref{r-11}), we have used the shorthand notation
\begin{align}\label{W_m}
&W_m=\sum_{\mu,\nu\geq 0}
c^{\rm A}_{m}(2k+2) \,\delta_{m+\mu-\nu,0}{1\over [k+1]}
K^-_{-\mu} K^+_{\nu} -c^{\rm A}_{m}(k+2)\,\delta_{m,0} \nonumber \\
&\qquad \qquad \qquad + c^{\rm A}_{m}(k+1)\,T_{m}
-(q-q^{-1})^3 
\sum_{\gamma\in \mathbb{Z}^\mathrm{A}} 
\no 
G^+_{-\gamma}G^-_{m+\gamma}
\no, 
\end{align}
where we have the numerical coefficients
\begin{align}\label{cAdef}
c^{\rm NS}_\alpha(u)=
\begin{cases}
\displaystyle  {1\over[u]} & (\alpha :{\rm even}), \\
\displaystyle {1\over[u]}{q^{u}+q^{-u}\over 2}& (\alpha :{\rm odd}),
\end{cases}
\quad 
c^{\rm R}_\alpha(u)=
\begin{cases}
\displaystyle  {1\over[u]}{q^{u}+q^{-u}\over 2} & (\alpha :{\rm even}), \\
\displaystyle {1\over[u]}& (\alpha :{\rm odd}),
\end{cases}
\end{align}
and the symbol $\no \bullet\no $ for the normal ordered product
\begin{align}\label{noGmodes}
 \no G^+_{m}G^-_{n}\no ~=
\begin{cases}
G^+_{m}G^-_{n}& (m<n), \\
\displaystyle {1\over 2}\left(G^+_{m}G^-_{m}-G^-_{m}G^+_{m}\right) & (m=n), \\
\displaystyle -G^-_{n}G^+_{m}& (m>n).
\end{cases}
\end{align}

\end{dfn}

\begin{rmk}\label{Inv}
The algebra $\Svir$ has the following involutive symmetry.
\begin{equation}
q \rightarrow q^{-1},\qquad k\rightarrow k,\qquad
K^{\pm}_m \rightarrow  K^{\pm}_m,\qquad 
 G^{\pm}_m\rightarrow G^{\mp}_m,\qquad 
 T_m \rightarrow T_m.
\end{equation}
\end{rmk}

\subsection{Generating functions}
\begin{dfn}

Introduce the generating functions $K^\pm(z),T(z)$ and $G^\pm(z)$ 
for the generators $K^\pm_m,T_m$ and $G^\pm_m$ as
\begin{align}
&K^\pm(z)=\sum_{\pm m\geq 0}K^\pm_m z^{-m},\qquad T(z)=\sum_{ m\in \mathbb{Z}}T_m z^{-m},\\
&G^\pm(z)=\sum_{ m\in \mathbb{Z}^\mathrm{A}}G^\pm_m z^{-m},
\end{align}
in the sector {\rm A} {\rm (}$=$~{\rm NS} or {\rm R}{\rm )}.
Set for simplicity that 
\begin{align}
&K(z)=K^-(z)K^+(z).
\end{align}
\begin{rmk}\label{NSvsR}
It should be emphasized that our convention of the mode expansion of $G^\pm(z)$ is different from the standard two 
dimensional superconformal field theory. We expand $G^\pm(z)$ in the integral powers of $z$ in the R sector
and in the half-integral powers of $z$ in the NS sector.
\end{rmk}
We also introduce the generating function $W(z)$ for the modes
$W_m$ in (\ref{W_m}) as
\begin{align}\label{W-current}
&W(z)=\sum_{ m\in \mathbb{Z}}W_m z^{-m} \CR
&={1\over [k+1]}  \widetilde{K}^{\rm A}(z)
-c^{\rm A}_0(k+2)+\widetilde{T}^{\rm A}(z)
-(q-q^{-1})^3 \no G^+(z)G^-(z)\no, 
\end{align}
where
\begin{align*}
&\widetilde{K}^{\rm A}(z)
= c^{\rm A}_{0}(2k+2) {1\over 2} \Bigl(K(z)+K(-z)\Bigr)+ c^{\rm A}_{1}(2k+2) {1\over 2} \Bigl(K(z)-K(-z)\Bigr),\\
&\widetilde{T}^{\rm A}(z)
= c^{\rm A}_{0}(k+1) {1\over 2} \Bigl(T(z)+T(-z)\Bigr)+ c^{\rm A}_{1}(k+1) {1\over 2} \Bigl(T(z)-T(-z)\Bigr).
\end{align*}
\end{dfn}

In Section \ref{Proof-1}
we will provide the defining relations of $\Svir$ in terms of the generating functions.

\begin{rmk}
In terms of the generating functions the involutive symmetry \eqref{Inv} is
\begin{equation*}
q \rightarrow q^{-1},\qquad k\rightarrow k,\qquad
K^{\pm}(z)\rightarrow  K^{\pm}(z),\qquad 
 G^{\pm}(z) \rightarrow G^{\mp}(z),\qquad 
 T(z)\rightarrow T(z).
\end{equation*}
\end{rmk}

\subsection{Heisenberg subalgebra}\label{Heisenberg}

Recall that the generators $K^\pm_0 $ (satisfying the condition $K^+_0=K^-_0$) are invertible. 
\begin{prp}
Defining $K^\pm_0 =q^{H_0}$, we have 
\begin{align*}
&q^{H_0} K^\pm(z) q^{-H_0}=K^\pm(z) ,\qquad q^{H_0}T(z) q^{-H_0}=T(z),\\
&q^{H_0} G^\pm(z) q^{-H_0}=q^{\pm (k+2)}G^\pm(z),
\end{align*}
namely
\begin{align}
[H_0,  K^\pm(z)]=0,\qquad [H_0, T(z)]=0,\qquad [H_0,  G^\pm(z)]=(k+2)G^\pm(z).\label{H_0}
\end{align}
\end{prp}
\begin{proof}
These relations follow from (\ref{r-2})-(\ref{r-7}).
\end{proof}

\begin{dfn}
Define the elements $H_m$ $(m\neq 0)$ by 
\begin{align}
&
\exp\left((q-q^{-1}) \sum_{m> 0}H_{\pm m}z^{\mp m}\right)=(K^\pm_0)^{-1}K^\pm(z)=
1+
\sum_{m> 0}(K^\pm_0)^{-1}K^\pm_{\pm m}z^{\mp m}.\label{Kpm(z)}
\end{align}

\end{dfn}

\begin{prp}\label{Hei}
We have the Heisenberg commutation relations
\begin{align}
&
[H_m,H_n]={[(k+2)m][km]\over m} \delta_{m+n,0}.\label{H-H}
\end{align}
\end{prp}
\begin{proof}
This follows from (\ref{r-3}).
\end{proof}


\subsection{Scope of the present article}
We have introduced the algebra $\Svir$ presented by the generators and relations (Definition \ref{DefSvir}, Proposition \ref{generating_functions}),
having the Heisenberg subalgebra (Proposition \ref{Hei}) and the involutive symmetry (Remark \ref{Inv}).
Some explanations are in order concerning the authors' motivation for
considering this quartic algebra $\Svir$, and what are planned to be investigated in the present article.

It is known that the $\mathcal{N}=2$ superconformal (or super Virasoro) algebra can be
realized by twisting the Wakimoto representation of the 
affine Lie algebra $\widehat{\mathfrak{sl}}_2$, 
while the $\mathbb{Z}_k$ parafermion structure kept intact. 
More explicitly, in terms of the parafermion currents $\psi^{\pm}(z)$ defined by \cite{Fateev:1985mm}, \cite{Fateev:1985ig}, \cite{Zamolodchikov:1986gh};
\begin{equation}
\psi^{\pm}(z) = 
\left( \sqrt{\frac{k+2}{2}} \partial \phi_1(z) \pm i \sqrt{\frac{k}{2}} \partial \phi_2(z) \right)
e^{\pm i \sqrt{\frac{2}{k}} \phi_2(z)},
\end{equation}
the affine Lie algebra $\widehat{\mathfrak{sl}}_2$ with level $k$ is realized as 
\begin{equation}
J^{\pm}(z) = \psi^{\pm}(z) e^{\pm \sqrt{\frac{2}{k}} \phi_0(z)}, \qquad J_3(z) = \sqrt{\frac{k}{2}} \partial \phi_0(z).
\end{equation}
Here 
$\phi_i(z)\phi_j(w) \sim \delta_{i,j}\log(z-w)$ are free bosons. The parafermion currents are characterized as 
the kernel of the two fermionic screening charges;
\begin{equation}
S^{\pm} = \oint dz~S^{\pm}(z), \qquad S^{\pm}(z) = e^{\sqrt{\frac{k+2}{2}} \phi_1(z)
 \pm i \sqrt{\frac{k}{2}}\phi_2(z)}.
\end{equation}
On the other hand the supercurrents $G^{\pm}(z)$ of the $\mathcal{N}=2$ superconformal algebra are
realized as 
\begin{equation}
G^{\pm}(z) = \psi^{\pm}(z) e^{\pm \sqrt{\frac{k+2}{k}}\phi_0(z)},
\end{equation}
and the remaining currents $K(z)$ and $T(z)$ are generated by the fusion (OPE) of $G^{\pm}(z)$.
Thus, we see that the currents $J^{\pm}(z)$ and $G^{\pm}(z)$ share the common parafermion currents $\psi^{\pm}(z)$
and their difference is the unit length of $\mathbb{Z}$ lattice (or \lq\lq compactification radius\rq\rq) of the $U(1)$ boson 
$\phi_0(z)$ in the vertex operator. 
This change also implies that the conformal weight of $G^{\pm}(z)$ is $\frac{3}{2}$, while $J^{\pm}(z)$ are spin $1$ currents.

Then naturally 
the present authors have lead to guess what should be a reasonable $q$-analogue of the 
$\mathcal{N}=2$ superconformal algebra, 
playing similar with the Wakimoto representation of the quantum affine algebra $U_q(\widehat{\mathfrak{sl}}_2)$. 
Note that, a priori, this kind of guess work could be problematic. Even if it has a good meaning, it
may not have a unique solution, because a different choice of generators
leads to a different set of relations. 
After some experience, fortunately, the authors found the algebra $\Svir$, among many other equivalent but somewhat 
more complicated definitions. Since $\Svir$ seems better behaving than any others, it 
has been chosen as the main object of the present paper. 

The questions we address in this paper are the following.
\begin{itemize}
\item
What kind of structures do we have for the Verma modules of $\Svir$? 
Can we find and prove, or guess at least, a factorization formula for the corresponding Kac determinants?
\item
Can we confirm by taking the $q\rightarrow 1$ limit of $\Svir$ that we recover the 
ordinary $\mathcal{N}=2$ superconformal (or super Virasoro) algebra?
\item Is there any nontrivial representation of $\Svir$, such as 
a Heisenberg  (or a Wakimoto-type)  representation?
\end{itemize}

For lack of space, in the present article we do not study 
explicit formulas for the singular vectors of $\Svir$ in the Wakimoto representation.
It seems an interesting problem to find 
a relation between these singular vectors and  
the supersymmetric version of the Jack or Macdonald polynomials 
(see \cite{Alarie-Vezina:2013qfa}, \cite{Alarie-Vezina:2018obz}, \cite{Alarie-Vezina:2019ohz}, \cite{Blondeau-Fournier:2016jth} 
and \cite{Desrosiers:2012zt} and references therein).

The paper is organized as follows;
In the next section we prove the defining relations of $\Svir$ in terms of the generating functions.
A conjecture on the Kac determinants is proposed in section 3. To support the conjecture, some examples of 
lower level singular vectors in the Verma module are worked out explicitly.
In section 4, we set $q=e^\hbar$ and make the $\hbar$ expansion of the generating functions and the defining relations of $\Svir$.
It is confirmed that the limit $q \to 1$ of $\Svir$ correctly reproduces the ordinary $\mathcal{N}=2$ SCA. 
After reviewing the Wakimoto representation of the quantum affine algebra $U_q(\widehat{\mathfrak{sl}}_2)$ and 
introducing basic vertex operators for the deformed parafermion sector in section 5, 
we argue the construction of $\Svir$ from the deformed $Y$-algebras (a.k.a. corner vertex operator algebras) in section 6.
Namely, $\Svir$ is obtained as a result of gluing two deformed $Y$-algebras, one of which is identified with the deformed parafermion
and the other provides the additional deformed Heisenberg algebra. 
It implies that we should twist the $U(1)$ boson in the Wakimoto representation of 
$U_q(\widehat{\mathfrak{sl}}_2)$, while keeping the parafermions intact. 
It also reveals a connection to the Fock representation of the quantum toroidal algebras.
Finally in section 7, we work out a Heisenberg representation of $\Svir$ by twisting the Wakimoto 
representation of $U_q(\widehat{\mathfrak{sl}}_2)$. 
Some of technical details are provided in Appendices. In Appendix A we summarize computations of operator product expansion (OPE) 
among vertex operators appearing in the Wakimoto representation of $U_q(\widehat{\mathfrak{sl}}_2)$. In particular 
the proof of the free field representation of $\Svir$ in section 7 relies on the OPE relations 
in the deformed parafermion sector shown in Appendix A. Appendix B supplements the proof in section 7. 
We conclude the paper with Appendix C, where as a first step to a proof of the conjecture on the Kac determinants,
we compute the vanishing lines arising from the screening operators among the Fock modules defined in section 7.


\section{$\Svir$ in terms of the generating functions}\label{Proof-1}

\begin{dfn}
Define the delta functions $\delta^{\rm NS}(z)$, $\delta^{\rm R}(z)$  and $\delta(z)$ by 
\begin{align}
\delta^{\rm NS}(z)=\sum_{ m\in \mathbb{Z}+{1\over 2}} z^m,\qquad 
\delta^{\rm R}(z)=\delta(z)=\sum_{ m\in \mathbb{Z}} z^m.
\end{align}
\end{dfn}

\begin{prp}\label{generating_functions}
The relations (\ref{r-2})-(\ref{r-11})
in the sector {\rm A} {\rm (}$=${\rm NS} or {\rm R)} are written in terms of the generating functions as
\begin{align}
&K^\pm(z)K^\pm(w)=K^\pm(w)K^\pm(z),\label{rr-2}\\
&
K^-(z)K^+(w)= {(1-q^{2k+2}z/w)(1-q^{-2k-2}z/w)\over (1-q^{2}z/w)(1-q^{-2}z/w)}
K^+(w)K^-(z),\label{rr-3}\\
&K^-(z)T(w)={(1-q^{k+1} z/w)(1-q^{-k-1}z/w)\over (1-q^{k+3}z/w)(1-q^{-k-3}z/w)} T(w)K^-(z),\label{rr-4}\\
&T(z)K^+(w)={(1-q^{k+1}z/w)(1-q^{-k-1}z/w)\over (1-q^{k+3}z/w)(1-q^{-k-3} z/w)} K^+(w)T(z),\label{rr-5}\\
&K^-(z)G^\pm (w)=q^{\pm (k+2)}{1-z/w\over 1-q^{\pm 2(k+2)}z/w}G^\pm(w) K^-(z),\label{rr-6}\\
&G^\pm (z)K^+(w)= q^{\mp (k+2)}{1-z/w\over 1-q^{\mp2(k+2)}z/w} K^+(w) G^\pm (z) ,\label{rr-7}\\
&G^\pm(z)G^\pm(w)+G^\pm(w)G^\pm (z)=0,\label{rr-8}\\
&
G^+(z)G^-(w)+G^-(w)G^+(z) 
= {1\over (q-q^{-1})^2}\Biggl(
\delta^{\rm A}\left( q^{4(k+1)} { w\over z}\right){1\over [k+1]}K(q^{2(k+1)}w)
\nonumber \\
&~~~\hspace{5cm}
-\delta^{\rm A}\left(q^{2(k+2)}{ w\over  z}\right)
+\delta^{\rm A}\left(q^{2(k+1)}{ w\over z}\right)  T(q^{k+1}w) \Biggr),\label{rr-9} \\
&G^\pm(z)T(w)-T(w)G^\pm(z) \CR
&=\pm (q-q^{-1}){[k+2]\over[k+1]}
\delta^{\rm A}\left(q^{\pm3(k+1)}{w\over z} \right)K^-(q^{\pm (k+1)}w)G^\pm(q^{\pm (k+1)}w)K^+(q^{\pm (k+1)}w), \label{rr-10}\\
&T(z)T(w)-T(w)T(z)
=(q-q^{-1}){[k+2]^2\over [k+1]}\CR
&\qquad \times \Biggl( \delta\left(q^{+2(k+1)}{w\over z} \right) K^- (q^{+k+1}w)  W(q^{+k+1}w) K^+(q^{+k+1}w)\nonumber \\
&\qquad \quad- \delta\left(q^{-2(k+1)}{w\over z} \right) K^- (q^{-k-1}w)  W(q^{-k-1}w) K^+(q^{-k-1}w)\Biggr).\label{rr-11}
\end{align}
Here the rational factors in (\ref{rr-3}), (\ref{rr-4}), (\ref{rr-5}), (\ref{rr-6}) and (\ref{rr-7}) 
should be read as the Taylor series of those in the domain $|z|\ll|w|$. 
\end{prp}

We need the following lemma. 
\begin{lem}\label{delta-op}
We have
\begin{align*}
\sum_{m}(q^{\alpha} w/z)^m \sum_{l}(q^\beta w)^{-l} \Phi_l=\sum_{m} \sum_{n} z^{-m} w^{-n} q^{(\alpha-\beta)m-\beta n} \Phi_{m+n},
\end{align*}
where the indices $l$ and $m$ run over either $ \mathbb{Z}$ or $ \mathbb{Z}+{1\over 2}$ and the range of $n=l -m$ is fixed accordingly. 
\end{lem}

Then Proposition \ref{generating_functions} is proved in a straightforward manner as follows.

\begin{proof}
It is clear that (\ref{r-2}) and (\ref{rr-2}) are the same. The equivalence between 
(\ref{r-3})-(\ref{rr-3}), (\ref{r-4})-(\ref{rr-4}), (\ref{r-5})-(\ref{rr-5}), (\ref{r-6})-(\ref{rr-6}), and (\ref{r-7})-(\ref{rr-7}) follow from
the Taylor series
\begin{align*}
&{(1-q^{2k+2}z/w)(1-q^{-2k-2}z/w)\over (1-q^{2}z/w)(1-q^{-2}z/w)}=
1-(q-q^{-1})^2\sum_{\ell \geq 1}{[k][k+2][2\ell]\over [2]} (z/w)^\ell,\\
&{(1-q^{k+1}z/w)(1-q^{-k-1}z/w)\over (1-q^{k+3}z/w)(1-q^{-k-3}z/w)}=
1+(q-q^{-1})^2\sum_{\ell \geq 1}{[k+2][(k+3)\ell]\over [k+3]} (z/w)^\ell,\\
&q^{\pm(k+2)}{1-z/w\over 1-q^{\pm 2(k+2)}z/w}=
q^{\pm(k+2)}\pm(q-q^{-1})[k+2]\sum_{\ell \geq 1} (q^{\pm 2(k+2)}z/w)^\ell.
\end{align*}
We immediately  see that (\ref{r-8}) and (\ref{rr-8}) are the same. 

Noting $K(w)=\sum_{\alpha,\beta\geq 0} w^{\alpha-\beta}K^{-}_{-\alpha} K^{+}_{\beta} $, and
using Lemma \ref{delta-op}, we have
\begin{align*}
\delta^{\rm A}\left( q^{4(k+1)} { w\over z}\right)K(q^{2(k+1)}w)
&=\sum_{m,n \in \mathbb{Z}^\mathrm{A} } z^{-m} w^{-n}  q^{(2k+2)(m-n)} \sum_{\alpha,\beta\geq 0}\delta_{m+n+\alpha-\beta ,0} K^{-}_{-\alpha} K^{+}_{\beta},\\
\delta^{\rm A}\left( q^{2(k+2)} { w\over z}\right)
&=\sum_{m,n \in \mathbb{Z}^\mathrm{A}}  z^{-m} w^{-n}   q^{(k+2)(m-n)} \delta_{m+n,0},\\
\delta^{\rm A}\left( q^{2(k+1)} { w\over z}\right)T(q^{k+1}w)
&=\sum_{m,n \in \mathbb{Z}^\mathrm{A}} z^{-m}w^{-n}    q^{(k+1)(m-n)} T_{m+n},
\end{align*}
showing the equivalence between (\ref{r-9})-(\ref{rr-9}).

Noting that
\begin{align*}
K^-(w) G^\pm(w) K^+(w)
&= \sum_{\alpha,\beta\geq 0}\sum_{m \in \mathbb{Z}^\mathrm{A}}  
w^{\alpha-m-\beta}K^-_{-\alpha} G^\pm_m K^+_\beta \CR
&= \sum_{m \in \mathbb{Z}^\mathrm{A}}  w^{-m} \sum_{\alpha,\beta\geq 0} K^-_{-\alpha} G^\pm_{m+\alpha-\beta} K^+_\beta,
\end{align*}
and using Lemma \ref{delta-op} with $m, l \in \mathbb{Z}^\mathrm{A}$, we have
\begin{align*}
&\delta^{\rm A}\left(q^{\pm3(k+1)}{w\over z} \right)K^-(q^{\pm (k+1)}w)G^\pm(q^{\pm (k+1)}w)K^+(q^{\pm (k+1)}w)\\
&= \sum_{m \in \mathbb{Z}^\mathrm{A}}  \sum_{n \in \mathbb{Z}} z^{-m} w^{-n} q^{\pm(k+1)(2m- n)}  
\sum_{\alpha,\beta\geq 0} K^-_{-\alpha} G^\pm_{m+n+\alpha-\beta} K^+_\beta,
\end{align*}
showing the equivalence between (\ref{r-10})-(\ref{rr-10}).

In the same way we have
\begin{align*}
&\delta\left(q^{+2(k+1)}{w\over z} \right) K^- (q^{+k+1}w)  W(q^{+k+1}w) K^+(q^{+k+1}w)\\
&- \delta\left(q^{-2(k+1)}{w\over z} \right) K^- (q^{-k-1}w)  W(q^{-k-1}w) K^+(q^{-k-1}w)\\
&=\sum_{m} \sum_{n} z^{-m} w^{-n}  \Bigl( q^{(k+1)(m-n)}-q^{-(k+1)(m-n)} \Bigr) 
\sum_{\alpha,\beta\geq 0} K^-_{-\alpha} W_{m+n+\alpha-\beta} K^+_\beta,
\end{align*}
giving the equivalence between (\ref{r-11})-(\ref{rr-11}).
\end{proof}

For later use, here we summarize the basic properties of the delta-functions $\delta^{\rm NS}(z)$
and $\delta^{\rm R}(z)=\delta(z)$.

\begin{lem}\label{delta-flip}
We have in the sector {\rm A}$=${\rm NS} or {\rm R} that
\begin{align}
&G^\pm(w) \delta \left( {w\over z}\right)=G^\pm(z) \delta^{\rm A}\left( {w\over z}\right),\label{delta-2}\\
&K^\pm(w) \delta^{\rm A}\left( {w\over z}\right)=K^\pm(z) \delta^{\rm A}\left( {w\over z}\right),
\qquad 
T(w) \delta^{\rm A}\left( {w\over z}\right)=T(z) \delta^{\rm A}\left( {w\over z}\right).\label{delta-3}
\end{align}
Note that the exchange $\delta \leftrightarrow \delta^{\rm NS}$ occurs, when we change
the argument of $G^\pm$ in the NS sector by using the delta-function. 
\end{lem}
\begin{proof}
We demonstrate (\ref{delta-2}) and the second equality in (\ref{delta-3}).
We have
\begin{align*}
G^\pm(w) \delta\left( {w\over z}\right) &=
\sum_{ m \in\mathbb{Z}^\mathrm{A}} w^{-m}G^\pm_{m}\sum_{ n \in\mathbb{Z}}w^{n}z^{-n}
= \sum_{ m \in\mathbb{Z}^\mathrm{A}} \sum_{ n \in\mathbb{Z}} w^{-m+n}z^{-n+m}z^{-m} G^\pm_{m} \\
& = G^\pm(z) \delta^{\rm A}\left( {w\over z}\right).
\end{align*}
Similarly 
\begin{align*}
&T(w) \delta^{\rm A}\left( {w\over z}\right)=
\sum_{ m \in \mathbb{Z}} w^{-m}T_{m}\sum_{ n \in \mathbb{Z}^\mathrm{A}}w^{n}z^{-n} \\
&=
\sum_{ m \in \mathbb{Z}} \sum_{ n \in \mathbb{Z}^\mathrm{A}} w^{-m+n}z^{-n+m}z^{-m} T_{m}
=T(z) \delta^{\rm A}\left( {w\over z}\right).
\end{align*}

\end{proof}

\begin{rmk}
The commutation relation \eqref{rr-11} is equivalent to
\begin{align*}
&T(z)T(w)-T(w)T(z)\\
=&-(q-q^{-1})^4{[k+2]^2\over [k+1]}
\left( \delta\left(q^{2k+2}{w\over z} \right)
K^{-}(q^{k+1}w) G^+(q^{k+1}w) G^-(q^{-k-1}z) K^{+}(q^{-k-1}z)
\right.\\
& \left.- \delta\left(q^{-2k-2}{w\over z} \right)
K^{-}(q^{k+1}z) G^+(q^{k+1}z) G^-(q^{-k-1}w) K^{+}(q^{-k-1}w)\right).
\end{align*}
Note that in the relation \eqref{rr-11} $G^{+}(z)$ and $G^{-}(z)$ are normal ordered (see the definition of $W(z)$).
\end{rmk}
To show the equivalence we prove the following lemma.
\begin{lem}\label{GG=W}
\begin{align*}
G^+(z) G^-(z) = - (q-q^{-1})^3 W(z).
\end{align*}
\end{lem}

\begin{proof}
The issue is the normal ordering of $G^{\pm}(z)$ and
we have to work with the mode expansion. 
Hence, we consider the R sector and the NS sector separately. 

\smallskip

(1)~R sector:
\begin{align*}
&G^+(w) G^-(w)
=\sum_{m,n\in \mathbb{Z}}w^{-m}w^{-n}G^+_mG^-_n
= \sum_{\alpha\in \mathbb{Z}} w^{-\alpha}\sum_{m\in \mathbb{Z}} 
G^+_mG^-_{\alpha-m}\\
=&
\sum_{\alpha:{\rm even}} w^{-\alpha} \Biggl[G^+_{{\alpha\over2}}G^-_{{\alpha\over2}}+
\sum_{m>0} 
G^+_{{\alpha\over2}-m}G^-_{{\alpha\over2}+m}-
\sum_{m> 0} 
G^-_{{\alpha\over2}-m}G^+_{{\alpha\over2}+m}\\
&+{1\over (q-q^{-1})^2}\sum_{m> 0} \left( q^{2(k+1)2m}{1\over [k+1]}K_{\alpha}-
q^{(k+2)2m} \delta_{\alpha,0}+q^{(k+1)2m} T_{\alpha}\right)\Biggr] \\
&
+\sum_{\alpha:{\rm odd}} w^{-\alpha} \Biggl[ \sum_{m\geq0} 
 G^+_{{\alpha\over2}-{1\over 2}-m}G^-_{{\alpha\over2}+{1\over 2}+m}-
\sum_{m\geq0} 
G^-_{{\alpha\over2}-{1\over 2}-m}G^+_{{\alpha\over2}+{1\over 2}+m}\\
&+{1\over (q-q^{-1})^2}\sum_{m\geq0} \left( q^{2(k+1)(2m+1)}{1\over [k+1]}K_{\alpha}-
q^{(k+2)(2m+1)} \delta_{\alpha,0}+q^{(k+1)(2m+1)} T_{\alpha}\right) \Biggr] \\
=&
\sum_{\alpha:{\rm even}} w^{-\alpha} \Biggr[G^+_{{\alpha\over2}}G^-_{{\alpha\over2}}+
\sum_{m>0} 
G^+_{{\alpha\over2}-m}G^-_{{\alpha\over2}+m}-
\sum_{m> 0} 
G^-_{{\alpha\over2}-m}G^+_{{\alpha\over2}+m}\\
&+{1\over (q-q^{-1})^2}\left( {q^{4k+4}\over 1-q^{4k+4}}{1\over [k+1]}K_{\alpha}-
{q^{2k+4}\over 1-q^{2k+4}}\delta_{\alpha,0}+
{q^{2k+2}\over 1-q^{2k+2}} T_{\alpha}\right)\Biggr] \\
&
+\sum_{\alpha:{\rm odd}} w^{-\alpha} \Biggl[ \sum_{m\geq0} 
 G^+_{{\alpha\over2}-{1\over 2}-m}G^-_{{\alpha\over2}+{1\over 2}+m}-
\sum_{m\geq0} 
G^-_{{\alpha\over2}-{1\over 2}-m}G^+_{{\alpha\over2}+{1\over 2}+m}\\
&+{1\over (q-q^{-1})^2}\left( {q^{2k+2}\over 1-q^{4k+4}}{1\over [k+1]}K_{\alpha}-
{q^{k+2}\over 1-q^{2k+4}}\delta_{\alpha,0}+
{q^{k+1}\over 1-q^{2k+2}}T_{\alpha}\right)\Biggr] \\
=&
\sum_{\alpha\in \mathbb{Z}}w^{-\alpha} 
\Biggl[ \sum_{m\in \mathbb{Z}} \no G^+_{-m}G^-_{\alpha+m} \no \CR
& \qquad
- {1\over (q-q^{-1})^3}\Biggl( c_\alpha^{\rm R}(2k+2){1\over [k+1]}K_{\alpha}-
c_\alpha^{\rm R}(k+2)\delta_{\alpha,0}+
c_\alpha^{\rm R}(k+1)T_{\alpha}\Biggr)\Biggr],
\end{align*}
where the normal ordered product $\no G^+_{m}G^-_{n}\no$  is defined  by \eqref{noGmodes}
and the numerical coefficient $c_\alpha^{\rm R}(u)$ is defined by \eqref{cAdef}.
Comparing with the definition of $W_m$ given by \eqref{W_m}, we obtain the relation in {\rm R} sector.

\smallskip

(2)~NS sector:
\begin{align*}
&G^+(w) G^-(w)
= \sum_{m,n\in \mathbb{Z}^\mathrm{NS}}w^{-m}w^{-n}G^+_mG^-_n
= \sum_{\alpha\in \mathbb{Z}} w^{-\alpha}\sum_{m\in \mathbb{Z}^\mathrm{NS}} 
G^+_mG^-_{\alpha-m}\\
=&
\sum_{\alpha:{\rm even}} w^{-\alpha} \Biggl[
\sum_{m\geq 0} 
G^+_{{\alpha\over2}-m-{1\over 2}}G^-_{{\alpha\over2}+m+{1\over 2}}-
\sum_{m\geq 0} 
G^-_{{\alpha\over2}-m-{1\over 2}}G^+_{{\alpha\over2}+m+{1\over 2}}\\
&+{1\over (q-q^{-1})^2}\sum_{m\geq  0} \Biggl( q^{2(k+1)(2m+1)}{1\over [k+1]}K_{\alpha}-
q^{(k+2)(2m+1)} \delta_{\alpha,0}+q^{(k+1)(2m+1)} T_{\alpha}\Biggr)\Biggr]\\
&
+\sum_{\alpha:{\rm odd}} w^{-\alpha} \Biggl[
G^+_{{\alpha\over2}}G^-_{{\alpha\over2}}+
\sum_{m>0} 
 G^+_{{\alpha\over2}-m}G^-_{{\alpha\over2}+m}-
\sum_{m>0} 
G^-_{{\alpha\over2}-m}G^+_{{\alpha\over2}+m}\\
&+{1\over (q-q^{-1})^2}\sum_{m>0} \Biggl( q^{2(k+1)2m}{1\over [k+1]}K_{\alpha}-
q^{(k+2)2m} \delta_{\alpha,0}+q^{(k+1)2m} T_{\alpha}\Biggr)\Biggr]\\
=&
\sum_{\alpha\in \mathbb{Z}}w^{-\alpha} 
\Biggl[ \sum_{m\in \mathbb{Z}^\mathrm{NS}} \no G^+_{-m}G^-_{\alpha+m}\no\\
&\qquad-
{1\over (q-q^{-1})^3}\Biggl( c^{\rm NS}_\alpha(2k+2){1\over [k+1]}K_{\alpha}-
c^{\rm NS}_\alpha(k+2)\delta_{\alpha,0}+
c^{\rm NS}_\alpha(k+1)T_{\alpha}\Biggr)\Biggr],
\end{align*}
where $\no G^+_{m}G^-_{n}\no$ is defined as before
and the numerical coefficient in the NS sector is defined by \eqref{cAdef}.
We see the relation in {\rm NS} sector is also valid.

Substituting the definition  \eqref{W_m} of the modes $W_m$, we have the following relations
in each sector $\mathrm{A}(=\mathrm{R}$ or $\mathrm{NS})$.
\begin{align}
&\sum_{\alpha,\beta\geq 0}K^-_{-\alpha}W_{m+\alpha-\beta}K^+_{\beta}
= \sum_{\alpha,\beta\geq 0}\sum_{\gamma\in \mathbb{Z}^\mathrm{A}} 
K^-_{-\alpha}
\no G^+_{-\gamma}G^-_{m+\alpha-\beta+\gamma} \no K^+_{\beta} \CR
&\qquad-
{1\over (q-q^{-1})^3}\sum_{\alpha,\beta\geq 0}\Biggl( c_{m+\alpha-\beta}^{\mathrm{A}}(2k+2){1\over [k+1]}
K^-_{-\alpha}K_{m+\alpha-\beta}K^+_{\beta} \CR
&\qquad-
c_{m+\alpha-\beta}^{\rm A}(k+2)\delta_{m+\alpha-\beta,0}K^-_{-\alpha}K^+_{\beta}+
c_{m+\alpha-\beta}^{\rm R}(k+1)K^-_{-\alpha}T_{m+\alpha-\beta}K^+_{\beta}\Biggr),
\end{align}
and 
\begin{align}\label{TTmodes-com}
&T_mT_n-T_nT_m 
= (q-q^{-1})^2{[k+2]^2\over [k+1]}[(m-n)(k+1)] \CR
&\times\sum_{\alpha,\beta\geq 0}\Biggl( 
\sum_{\mu,\nu\geq 0}c_{-\mu+\nu}^{\rm A}(2k+2) \delta_{m+n+\alpha-\beta+\mu-\nu,0}{1\over [k+1]}
K^-_{-\alpha}K^-_{-\mu} K^+_{\nu}K^+_{\beta} \CR
&\qquad-
c_{0}^{\rm A}(k+2)\delta_{m+n+\alpha-\beta,0}K^-_{-\alpha}K^+_{\beta}+
c_{m+n+\alpha-\beta}^{\rm A}(k+1)K^-_{-\alpha}T_{m+n+\alpha-\beta}K^+_{\beta}\Biggr) \CR
&-(q-q^{-1})^5{[k+2]^2\over [k+1]}[(m-n)(k+1)]
\sum_{\alpha,\beta\geq 0}\sum_{\gamma\in \mathbb{Z}^\mathrm{A}} 
K^-_{-\alpha}
\no G^+_{-\gamma}G^-_{m+n+\alpha-\beta+\gamma}\no K^+_{\beta}.
\end{align}

Now to obtain the commutation relation of the generating function $T(z)$, 
let us multiply both sides of \eqref{TTmodes-com} with $z^{-m}w^{-n}$ and 
take the summation over $m,n \in \mathbb{Z}$. 
The term with $c_{0}^{\rm A}(k+2)$ gives
\begin{align}
\sum_{m,n\in \mathbb{Z}}z^{-m}w^{-n}
\sum_{\alpha,\beta\geq 0}\delta_{m+n+\alpha-\beta,0}K^-_{-\alpha}K^+_{\beta}
=& \sum_{m\in \mathbb{Z}}z^{-m}w^{m}\sum_{\alpha,\beta\geq 0}
w^{\alpha}K^-_{-\alpha}w^{-\beta}K^+_{\beta} \CR
=& \delta\left({w\over z} \right)K^-(w)K^+(w).
\end{align}
Similarly we have
\begin{align}
&\sum_{m,n\in \mathbb{Z}}z^{-m}w^{-n}
\sum_{\alpha,\beta\geq 0}c_{m+n+\alpha-\beta}^{\mathrm A} (k+1)K^-_{-\alpha}T_{m+n+\alpha-\beta}K^+_{\beta} \CR
=&\sum_{m\in \mathbb{Z}}z^{-m}w^{m}
\sum_{\alpha,\beta\geq 0}\sum_{n\in \mathbb{Z}}
w^{\alpha}K^-_{-\alpha}
w^{-m-n-\alpha+\beta}
c_{m+n+\alpha-\beta}^{\mathrm A}(k+1)T_{m+n+\alpha-\beta}
w^{-\beta}K^+_{\beta} \CR
=& \delta\left({w\over z} \right)K^-(w)\widetilde{T}^{\mathrm A}(w)K^+(w),
\end{align}
where we have defined 
\begin{align}
\widetilde{T}^{\mathrm A}(z)=&\sum_{m\in \mathbb{Z}}c_{m}^{\mathrm A}(k+1)z^{-m}T_{m} \CR
=&~c_{0}^{\mathrm A}(k+1) {1\over 2} (T(z)+T(-z))+ c_{1}^{\mathrm A}(k+1) {1\over 2} (T(z)-T(-z)).
\end{align}
Finally the quartic term in the modes $K_m$ is
\begin{align}
&\sum_{m,n\in \mathbb{Z}}z^{-m}w^{-n}
\sum_{\alpha,\beta\geq 0}\sum_{\mu,\nu\geq 0}c_{-\mu+\nu}^{\mathrm A}(2k+2) \delta_{m+n+\alpha-\beta+\mu-\nu,0}{1\over [k+1]}
K^-_{-\alpha}K^-_{-\mu} K^+_{\nu}K^+_{\beta} \CR
&=\sum_{m\in \mathbb{Z}}z^{-m}w^{m}
\sum_{\alpha,\beta\geq 0}\sum_{\mu,\nu\geq 0}c_{-\mu+\nu}^{\mathrm A}(2k+2) 
{1\over [k+1]}
w^{\alpha}K^-_{-\alpha}w^{\mu}K^-_{-\mu} w^{-\nu}K^+_{\nu}w^{-\beta}K^+_{\beta} \CR
&={1\over [k+1]} \delta\left({w\over z} \right)K^- (w)\widetilde{K}^{\mathrm A}(w)K^+(w),
\end{align}
where we have defined 
\begin{align}
&\widetilde{K}^{\mathrm A}(z)=\sum_{m\in \mathbb{Z}}c_{m}^{\mathrm A}(2k+2)z^{-m}K_{m} \CR
=&~c_{0}^{\mathrm A}(2k+2) {1\over 2} (K(z)+K(-z))+ c_{1}^{\mathrm A}(2k+2) {1\over 2} (K(z)-K(-z)).
\end{align}
In summary we obtain
\begin{align}
T(z)T(w)-T(w)T(z) =&~(q-q^{-1}){[k+2]^2\over [k+1]}
\sum_{\epsilon=\pm 1}\epsilon \,\,
\delta\left(q^{(2k+2)\epsilon }{w\over z} \right)
K^- (q^{(k+1)\epsilon}w) \CR
&\times\Biggl({1\over [k+1]}  \widetilde{K}^{\mathrm A}(q^{(k+1)\epsilon}w)
-c_0^{\mathrm A}(k+2)+\widetilde{T}^{\mathrm A}(q^{(k+1)\epsilon}w) \CR
&\qquad - (q-q^{-1})^3 :G^+(q^{(k+1)\epsilon}w)G^-(q^{(k+1)\epsilon}w):\Biggr)K^+(q^{(k+1)\epsilon}w).
\end{align}

\end{proof}

\begin{rmk}
In the case of the deformed Virasoro algebra and the deformed $W_3$ algebra,
the commutation relation of $T(z)$ are
\begin{equation}\label{qVir}
 f^{(2)}(w/z)T(z)T(w) - f^{(2)}(z/w)T(w)T(z) 
=\frac{(1-q_1)(1-q_2)}{1- q_3^{-1}} \left( \delta(q_3w/z) - \delta(w/q_3z) \right),
\end{equation}
and 
\beqa
&&f^{(3)}(w/z) T(z) T(w) - f^{(3)}(z/w) T(w) T(z) \CR
&=& \frac{(1-q_1)(1-q_2)}{1- q_3^{-1}} \left( \delta(q_3w/z)W(q_3^{1/2}w) - \delta(w/q_3z) W(q_3^{1/2}z)\right),
\eeqa
where $q_1=q,q_2=t^{-1}$ and $q_3=p^{-1}=(q_1q_2)^{-1}$. $W(z)$ is a higher current in the $W_3$ algebra.\footnote{
$W(z)$ has nothing to do with the $W$-current defined by \eqref{W-current} and appears in \eqref{rr-11}.}
Thus, the structure function (the generating function of the structure constants) 
\beq\label{structuregeneral}
f^{(N)}(z) = \exp \left( \sum_{n=1}^\infty \frac{z^n}{n}(1-q_1^n)(1-q_2^n) 
\frac{1-q_3^{-(N-1)n}}{1- q_3^{-Nn}}
\right)
\eeq
is required for writing down the commutation relation of $T(z)$. 
Contrary to these cases, for the deformed $\mathcal{N}=2$ superconformal algebra, 
we do not have such structure functions in $G$-$G$, $G$-$T$ and $T$-$T$ commutation relations.
In the free field representation to be discussed in section \ref{Twist-Wakimoto}, we will see that
this is a result of the cancellation of OPE coefficients coming from the parafermion sector and 
the $U(1)$ boson sector.
\end{rmk}

\begin{rmk}
The associative algebra $\Svir$ does not contain the $q$-deformed algebra as a subalgebra.
This is a common feature to the $q$-deformed $W$-algebras \cite{Awata:1995zk}.
\end{rmk}

\section{Verma Modules and the Kac determinants}

The formulae for the Kac determinants of the $\mathcal{N}=2$ superconformal algebra were worked out by several groups;
\cite{Boucher:1986bh}, \cite{DiVecchia:1985ief}, \cite{DiVecchia:1986fwg}, \cite{Nam:1985qe}, \cite{Kato:1986td}, \cite{KM-adv}.
For the deformed Virasoro algebra, an explicit formula of the Kac determinant was conjectured in 
\cite{Shiraishi:1995rp}. See also \cite{BP}. 
In this section we explore the singular vectors in the Verma modules of $\Svir$
and give a conjecture on the factorization property of the Kac determinants. 
Useful methods of obtaining explicit forms of the Virasoro singular vectors are given, for example,
in \cite{Benoit:1988aw}, \cite{Bauer:1991qm}, \cite{MillionschikovA}, \cite{MillionschikovB}.
For explicit forms of the low-lying singular vectors in the case of the deformed Virasoro algebra,
see \cite{Shiraishi-SGC}.

\subsection{Verma modules in the {\rm NS} sector}
\begin{dfn}
Let $h$ and $u$ be complex parameters. Consider the {\rm NS} sector of $\Svir$.
Let $|h,u\rangle$ be the
highest weight vector satisfying the conditions
\begin{align}\label{hwv}
&K^\pm_0 |h,u\rangle=u |h,u\rangle,\qquad 
T_0|h,u\rangle=h |h,u\rangle,\nonumber\\
&K^+_m |h,u\rangle=T_m |h,u\rangle=0\qquad (m>0), \\
&G^\pm_m |h,u\rangle=0\qquad (m\geq 1/2).\nonumber
\end{align}
The Verma module $M_{h,u}$ in the {\rm NS} sector is defined to be the left 
$\Svir$ module $M_{h,u}=\Svir|h,u\rangle$.
\end{dfn}

\begin{dfn}\label{hwv-dual}
The dual Verma module $M^*_{h,u}$ in the {\rm NS} sector is defined to be the right $\Svir$ module
$M^*_{h,u}=\langle h,u|\Svir$, with the vector $\langle h,u|$ 
satisfying the conditions
\begin{align}
&\langle h,u|h ,u\rangle=1,\nonumber\\
&\langle h,u|K^\pm_0 = u \langle h,u|,\qquad 
\langle h,u|T_0= h \langle h,u|,\nonumber\\
&\langle h,u|K^-_m = \langle h,u|T_m =0\qquad (m<0), \\
&\langle h,u|G^\pm_m =0\qquad (m\leq -1/2).\nonumber
\end{align}
\end{dfn}

Recall that we have the character formula \cite{Boucher:1986bh};
\begin{align}
{\rm ch}_{\rm NS}(p,x)= \Tr_{\mathcal{V}_{\mathrm{NS}}} \left( p^{\mathsf{L}_0 - h} x^{\mathsf{I}_0-u} \right) 
= \prod_{i=0}^\infty {(1+p^{i+1/2}x)(1+p^{i+1/2}x^{-1})\over (1-p^{i+1}) (1-p^{i+1})},
\qquad |p|, |x| <1 \label{ch-NS}
\end{align}
for the Verma module $\mathcal{V}_{\mathrm{NS}}$
of the ordinary $\mathcal{N}=2$ superconformal algebra in the {\rm NS} sector. 
$\mathsf{L}_0$ and $\mathsf{I}_0$ are mutually commuting zero modes of
the $\mathcal{N}=2$ superconformal algebra (see section \ref{sec:CFT-limit}) and 
they have eigenvalues $h$ and $u$ on the ground state of $\mathcal{V}_{\mathrm{NS}}$.
Since we lack a $q$-analogue of the Poincar\'e-Birkoff-Witt (PBW) theorem for $\Svir$, 
we simply assume that we have the same  character for $M_{h,u}$ by the rule
\begin{align*}
&K^\pm_{-m},T_{-m}, G^\pm_{-m} \mbox{ have the $p$-degree } m,\\
&K^\pm_{-m},T_{-m}\mbox{ have the $x$-degree } 0,\\
&G^\pm_{-m} \mbox{ have the $x$-degree } \pm 1.
\end{align*}

\begin{dfn}
A finite non-increasing sequence of positive integers 
$\lambda=(\lambda_1,\lambda_2,\ldots,\lambda_l)$
($\lambda_i\in \mathbb{Z}, \lambda_1\geq \lambda_2\geq \cdots \geq \lambda_l>0$)
is called a partition. We denote by $\ell(\lambda)=l$ the length of $\lambda$. 
The set of partitions is denoted by $\mathcal{P}$.
A finite strictly decreasing sequence of positive half-integers 
$\alpha=(\alpha_1,\alpha_2,\ldots,\alpha_l)$
($\alpha_i\in \mathbb{Z}+1/2, \alpha_1> \alpha_2> \cdots > \alpha_l>0$)
is called a fermionic partition in the {\rm NS} sector. We denote by $\ell(\alpha)=l$ the length of $\alpha$. 
The set of fermionic partitions in the {\rm NS} sector is denoted by $\mathcal{P}^{\rm NS}$.
\end{dfn}

Let $\lambda=(\lambda_1,\ldots,\lambda_l)$ be a partition and 
$\alpha=(\alpha_1,\ldots,\alpha_a)$ be a fermionic partition in the {\rm NS} sector.
We introduce the following notations for the ordered products of generators
\begin{align*}
&K^-_{-\lambda}=K^-_{-\lambda_1}K^-_{-\lambda_2}\cdots K^-_{-\lambda_l},\qquad 
K^+_{\lambda}=K^+_{\lambda_l}\cdots K^+_{\lambda_2}K^+_{\lambda_1}
\\
&T_{-\lambda}=T_{-\lambda_1}T_{-\lambda_2}\cdots T_{-\lambda_l},\qquad \,\,\,\,\,\,
T_{\lambda}=T_{\lambda_l}\cdots T_{\lambda_2}T_{\lambda_1},\\
&G^\pm_{-\alpha}=G^\pm_{-\alpha_1}G^\pm_{-\alpha_2}\cdots G^\pm_{-\alpha_a},\qquad
G^\pm_{\alpha}=G^\pm_{\alpha_a}\cdots G^\pm_{\alpha_2} G^\pm_{\alpha_1}.
\end{align*}
Then, for a pair of partitions
$\lambda=(\lambda_1,\ldots,\lambda_l),\mu=(\mu_1,\ldots,\mu_m)$, and 
a pair of fermionic partitions in the {\rm NS} sector
$\alpha=(\alpha_1,\ldots,\alpha_a),\beta=(\beta_1,\ldots,\beta_b)$, 
set 
\begin{align*}
&|\lambda,\mu,\alpha,\beta\rangle=|\lambda,\mu,\alpha,\beta;h,u\rangle=
K^-_{-\lambda}T_{-\mu}G^+_{-\alpha}G^-_{-\beta}|h,u\rangle,\\
&\langle \lambda,\mu,\alpha,\beta|=\langle\lambda,\mu,\alpha,\beta;h,u|=
\langle h,u|G^+_{\beta}G^-_{\alpha}T_{\mu}K^+_{\lambda}.
\end{align*}
Note that we have the lexicographical orderings on the sets $\mathcal{P}$ and $\mathcal{P}^{\rm NS}$. 
Hence one may introduce the associated total ordering on the set
$\mathcal{P}\times \mathcal{P}\times \mathcal{P}^{\rm NS}\times \mathcal{P}^{\rm NS}$.

\bigskip

Example: for the subspace with $p$-degree $2$ and $x$-degree $0$ we have
\begin{align*}
&((2),\emptyset,\emptyset,\emptyset)>((1^2),\emptyset,\emptyset,\emptyset)>((1),(1),\emptyset,\emptyset)>
((1),\emptyset,(1/2),(1/2))
>(\emptyset,(2),\emptyset,\emptyset)\\
&>(\emptyset,(1^2),\emptyset,\emptyset)>
(\emptyset,(1),(1/2),(1/2))
>(\emptyset,\emptyset,(3/2),(1/2))>(\emptyset,\emptyset,(1/2),(3/2)).
\end{align*}

We define the $p$- and the $x$-degrees for the states in $M_{h,u}$ and $M^*_{h,u}$ as follows.
\begin{center}
\begin{tabular}{c|c|c}
state& $p$-degree& $x$-degree\\[2mm]\hline
$ |\lambda,\mu,\alpha,\beta;h,u\rangle$&$|\lambda|+|\mu|+|\alpha|+|\beta|$& $\ell(\alpha)-\ell(\beta)$\\[2mm]\hline
$ \langle \lambda,\mu,\alpha,\beta;h,u|$&$|\lambda|+|\mu|+|\alpha|+|\beta|$& $\ell(\alpha)-\ell(\beta)$
\end{tabular}
\end{center}

\begin{prp}
The ordered collection $(|\lambda,\mu,\alpha,\beta\rangle)_{\lambda,\mu \in \mathcal{P}, \alpha,\beta\in \mathcal{P}^{\rm NS}}$ 
forms a basis of  $M_{h,u}$. Similarly, $(\langle\lambda,\mu,\alpha,\beta|)_{\lambda,\mu \in \mathcal{P}, 
\alpha,\beta\in \mathcal{P}^{\rm NS}}$ forms a basis of $M^*_{h,u}$.
\end{prp}

Sketch of proof~:~
First, observe that all the defining relations in Definition \ref{DefSvir} 
play the role of the  normal ordering rules.  
Hence the collection $(|\lambda,\mu,\alpha,\beta\rangle)_{\lambda,\mu \in \mathcal{P}, \alpha,\beta\in \mathcal{P}^{\rm NS}}$ 
spans  $M_{h,u}$.

Next, we show that the collection is linearly independent. 
We use a deformation argument with respect to the parameter $q=e^\hbar$.
From Propositions \ref{prp:K-ser}, \ref{Exp-1}, Definition \ref{CFTcurrents} below, we have
the $\hbar$ expansions of the form
\begin{align*}
&G^\pm(z)=\sqrt{k+2\over 2} \mathsf{G}^\pm(z)+ \mathcal{O}(\hbar),\\
&K'(z):={1\over \hbar} (K(z)-1)=2(k+2) \mathsf{I}(z)+ \mathcal{O}(\hbar),\\
&T'(z):={1\over \hbar^2}\left(T(z)-{k\over k+1} K(z)-1\right)=
\left(  4(k+2) \mathsf{L}(z)-{k(2k+1)\over 6(k+1)}\right)+ \mathcal{O}(\hbar),
\end{align*}
where $ \mathsf{G}^\pm(z), \mathsf{I}(z),  \mathsf{L}(z)$ satisfy the $\mathcal{N}=2$ superconformal algebra 
as in Theorem \ref{CFT-Theorem}.
Hence in the limit $\hbar \rightarrow 0$ the ordered collection 
$({K'}_{-\lambda}{T'}_{-\mu}G^+_{-\alpha}G^-_{-\beta}|h,u\rangle)$
tends to a PBW basis of the $\mathcal{N}=2$ superconformal Lie superalgebra, 
proving the linear independence of the collection for $\hbar=0$. 
Then the deformation argument shows that we have the linear independence 
also for $\hbar\neq 0$.\footnote{
Since the determinant is a holomorphic function of $\hbar$, if it is non-vanishing at $\hbar=0$,
the same is valid for $\hbar \neq 0$ with sufficiently small $|\hbar|$.}
Recalling that $K(z)=K^-(z)K^+(z)$, 
we know that the transition matrix from 
$({K'}_{-\lambda}{T'}_{-\mu}G^+_{-\alpha}G^-_{-\beta}|h,u\rangle)$ to 
$( |\lambda,\mu,\alpha,\beta;h,u\rangle)$ is an upper triangular matrix with 
non vanishing diagonal entries. \hfill $\square$

Recall the definition of a singular vector.
A singular vector $|\chi\rangle \in M_{h,u}$ is a non-zero vector satisfying
\begin{align}\label{singular}
&K^\pm_0 |\chi\rangle=u_\chi|\chi\rangle,\qquad
T_0|\chi\rangle=h_\chi |\chi\rangle\qquad ((h_\chi,u_\chi)\neq (h,u)
\,\, \mbox{i.e.}\,\, |\chi\rangle \not \propto |h,u\rangle ),\nonumber\\
&K^+_m |\chi\rangle=T_m |\chi\rangle=0\qquad (m>0), \\
&G^\pm_m |\chi\rangle=0\qquad (m\geq 1/2).\nonumber
\end{align}
Hence, the information about the zero's of the Kac determinant plays an essential role for
finding the singular vectors.
When there exists a singular vector $|\chi\rangle \in M_{h,u}$, we have a proper submodule 
$\Svir|\chi\rangle \subsetneq M_{h,u}$.

\subsection{Notations}
We prepare some notations necessary for our description of the Kac determinants.

\begin{dfn}
Let $P^{\rm NS} (n,j)$ denotes the multiplicity of the {\rm NS}-character:
\begin{align}
\sum_{n\in {1\over 2}\mathbb{Z}_{\geq 0}}\sum_{j\in \mathbb Z} P^{\rm NS} (n,j)
p^n x^j=
\prod_{i=0}^\infty {(1+p^{i+1/2}x)(1+p^{i+1/2}x^{-1})\over (1-p^{i+1}) (1-p^{i+1})}.
\end{align}
For $\ell \in \mathbb{Z}+{1\over 2}$, let
$\widetilde{P}^{\rm NS} (n,j;\ell)$ denotes the multiplicity of the character
with one fermion with the weight $p^{|\ell|} x^{{\rm sgn}(\ell) }$ missing:
\begin{align}
\sum_{n\in {1\over 2}\mathbb{Z}_{\geq 0}}\sum_{j\in \mathbb Z} 
\widetilde{P}^{\rm NS} (n,j;\ell)
p^n x^j={1\over 1+p^{|\ell|} x^{{\rm sgn}(\ell) }}
\prod_{i=0}^\infty {(1+p^{i+1/2}x)(1+p^{i+1/2}x^{-1})\over (1-p^{i+1}) (1-p^{i+1})},
\end{align}
where ${\rm sgn}(\ell)=\ell/|\ell|$.
\end{dfn}

\begin{dfn}
Set
\begin{align}
f(r,s;u,v) =&~u^4 (q-q^{-1})^4
\Bigl( q^{1-r+(k+2)s} v - q^{-1+r-(k+2)s} v^{-1}\Bigr) \CR
& \qquad \times \Bigl( q^{-1-r+(k+2)s} v^{-1} - q^{1+r-(k+2)s} v\Bigr),
\end{align}
and
\begin{align}
g(\ell;u,v) =&~{-u \over (q-q^{-1})^2} 
\Bigl( q^{\ell+1\over 2} u^{1\over 2} v^{1\over 2}-
q^{-{\ell+1\over 2}} u^{-{1\over 2}} v^{-{1\over 2}} \Bigr) \CR
& \quad \times  \Bigl( q^{\ell-1\over 2} u^{1\over 2} v^{-{1\over 2}}-
q^{-{\ell-1\over 2}} u^{-{1\over 2}} v^{{1\over 2}} \Bigr).
\end{align}
\end{dfn}

\subsection{Kac determinants in the {\rm NS} sector}
We study the Kac determinants associated with the Verma module $M_{h,u}$ in the NS sector.
We denote by ${\rm det}_{n,j}={\rm det}_{n,j}^{\rm NS}(h,u)$
the Kac determinant in the NS sector associated with the subspace in $M_{h,u}$ 
having the $p$-degree $n$ and the $x$-degree $j$.

\begin{dfn}
Let $v \in \mathbb{C}^{\times}$ be a generic parameter. 
Introduce the parametrization of $h$ in terms of the parameter $v$ as
\begin{align}\label{huv}
h (u,v) :=q u v- {[k+2]\over [k+1]}u^2+q^{-1} u v^{-1}.
\end{align}
\end{dfn}

Note that we have
\begin{equation*}
f(r,s;u,v)
=u^4 (q-q^{-1})^4
\biggl((q^{r-(k+2)s}+q^{-r+(k+2)s})^2-h^2 u^{-2}
- 2 h {[k+2]\over [k+1]} -u^2 {[k+2]^2\over [k+1]^2}\Biggr),
\end{equation*}
and 
\begin{equation*}
g(\ell;u,v)
={1\over (q-q^{-1})^2}
\biggl(-q^\ell u+ {[k+2]\over [k+1]} u^2+h-q^{-\ell}\biggr).
\end{equation*}

\begin{con}\label{Kac-det-NS}
We have
\begin{align}
{\rm det}_{n,j}^{\rm NS}(h,u) =
{\rm cst.}
\prod_{\scriptstyle r,s \in \mathbb{Z}_{> 0}
\atop 
\scriptstyle 1\leq rs\leq n} \Bigl(f(r,s;u,v)\Bigr)^{P^{\rm NS}(n-rs,j)}
\prod_{\ell\in \mathbb{Z}+{1\over 2}}
\Bigl(g(2\ell;u,v)\Bigr)^{\widetilde{P}^{\rm NS}(n-|\ell|,j-{\rm sgn}(\ell);\ell)},
\end{align}
where {\rm cst.} is a certain non zero constant not depending on $u$ or $v$.
\end{con}

It is possible to take the limit $q \to 1$ of Conjecture \ref{Kac-det-NS}, 
which reproduces the formula of the Kac determinant of the $\mathcal{N}=2$ superconformal 
algebra (\it{e.g.} in \cite{Kato:1986td}).
Explicit examples of the low-lying singular vectors are given in the next subsection. 
In Appendix \ref{app:screening} we work out the condition for the screening operators
between Fock modules to produce singular vectors, which provides a supporting
evidence for the Conjecture \ref{Kac-det-NS}.

\subsection{Examples of the Kac determinants in the {\rm NS} sector}
\subsubsection{Case $n=0,j=0$}
For the subspace with the $p$-degree zero in $M_{h,u}$, we 
only have the highest weight vector $|h,u\rangle\in M_{h,u}$ with the $x$-degree zero.
Hence we have
\begin{align*}
&{\rm det}^{\rm NS}_{0,0}=\langle h,u|h,u\rangle=1.
\end{align*}

\subsubsection{Case $n=1/2$, $j=\pm 1$}
For the subspace with the $p$-degree $1/2$, we 
have two vectors $G^+_{-1/2}|h,u\rangle,G^-_{-1/2}|h,u\rangle\in M_{h,u}$ with the $x$-degrees $ +1$ and $-1$. 
We have
\begin{align*}
&{\rm det}^{\rm NS}_{1/2,1}=\langle h,u|G^-_{1/2}G^+_{-1/2}|h,u\rangle\\
&=
{1\over (q-q^{-1})^2}\Biggl({1\over [k+1]}q^{-2k-2}u^2-q^{-k-2}+q^{-k-1} h\Biggr)=q^{-k-1} g(+1;u,v),\\
&{\rm det}^{\rm NS}_{1/2,-1}=\langle h,u|G^+_{1/2}G^-_{-1/2}|h,u\rangle\\
&=
{1\over (q-q^{-1})^2}\Biggl({1\over [k+1]}q^{+2k+2}u^2-q^{+k+2}+q^{+k+1} h\Biggr)=q^{+k+1} g(-1;u,v).
\end{align*}

\begin{prp}
If $g(\pm 1;u,v)=0$, then $G^\pm_{-1/2}|h,u\rangle$ is a singular vector.
\end{prp}

\begin{rmk}
Write 
$$
\vert \chi(1/2,\pm 1)\rangle=G^\pm_{-1/2}|h,u\rangle,
$$
for arbitrary $u$ and $v$.
Note that we have
\begin{align*}
G^\pm_{-1/2}\vert \chi(1/2,\pm 1)\rangle=0.
\end{align*}
This explains the factor ${1/(1+p^{1/2}x^{\pm 1}})$ in 
the character of the descendants of $G^\pm_{-1/2}|h,u\rangle$  given by 
$$
{1\over 1+p^{1/2}x^{\pm 1}}
\prod_{i=0}^\infty {(1+p^{i+1/2}x)(1+p^{i+1/2}x^{-1})\over (1-p^{i+1}) (1-p^{i+1})}
=
\sum_{n\in {1\over 2}\mathbb{Z}_{\geq 0} }\sum_{j\in \mathbb{Z}}
\widetilde{P}^{\rm NS}(n,j,\pm 1/2) p^n x^j.
$$
\end{rmk}

\subsubsection{Case $n=1$, $j=0$}
For the subspace with the $p$-degree $1$, we 
have three vectors $K^-_{-1}|h,u\rangle,T^-_{-1}|h,u\rangle,
G^+_{-1/2}G^-_{-1/2}|h,u\rangle\in M_{h,u}$ with the $x$-degree $0$. 
We have
\begin{align*}
&{\rm det}^{\rm NS}_{1,0}\\
&=\left|
\begin{array}{lll}
\langle h,u|K^+_{1}K^-_{-1}|h,u\rangle&
\langle h,u|K^+_{1}T_{-1}|h,u\rangle&
\langle h,u|K^+_{1}G^+_{-1/2}G^-_{-1/2}|h,u\rangle\\
\langle h,u|T_{1}K^-_{-1}|h,u\rangle&
\langle h,u|T_{1}T_{-1}|h,u\rangle&
\langle h,u|T_{1}G^+_{-1/2}G^-_{-1/2}|h,u\rangle\\
\langle h,u|G^+_{1/2}G^-_{1/2}K^-_{-1}|h,u\rangle&
\langle h,u|G^+_{1/2}G^-_{1/2}T_{-1}|h,u\rangle&
\langle h,u|G^+_{1/2}G^-_{1/2}G^+_{-1/2}G^-_{-1/2}|h,u\rangle
\end{array}
\right|\\
&=[k+2]^2 g(+1;u,v)g(-1;u,v)f(1,1;u,v).
\end{align*}

\begin{prp}
If $f(1,1;u,v)=0$, then 
\begin{align}
&\biggl(-1+{[k+2]\over[k+1]} q^{k+1}u\biggr)K^-_{-1}|h,u\rangle+T_{-1}|h,u\rangle\\
&+
(q-q^{-1})^3[k+2] \biggl({q^{k+2}u\over u-q^{k+2}}\biggr)
G^+_{-1/2}G^-_{-1/2}|h,u\rangle,\nonumber
\end{align}
 is a singular vector.
 \end{prp}

\subsubsection{Case $n=3/2$, $j=\pm 1$}
For the subspace with the $p$-degree $1$, we 
have three vectors $K^-_{-1}G^\pm_{-1/2}|h,u\rangle,T^-_{-1}G^\pm_{-1/2}|h,u\rangle,
G^\pm_{-3/2}|h,u\rangle\in M_{h,u}$ with the $x$-degrees $\pm 1$. 
We have
\begin{align*}
&{\rm det}^{\rm NS}_{3/2,+1}=q^{-3k-1}[k+2]^2 g(+1;u,v)^2g(+3;u,v)f(1,1;u,v),\\
&{\rm det}^{\rm NS}_{3/2,-1}=q^{+3k+1}[k+2]^2 g(-1;u,v)^2g(-3;u,v)f(1,1;u,v).
\end{align*}

\begin{prp}
If $g(3;u,v)=0$, then 
\begin{align}
&{ q^{-k}u\over[k+1]} \,K^-_{-1}G^+_{-1/2}|h,u\rangle+T_{-1}G^+_{-1/2}|h,u\rangle
+
q^{k}(1-q^{-2})(1-q^{4}u^2)
G^+_{-3/2}|h,u\rangle,\nonumber
\end{align}
 is a singular vector.
If $g(-3;u,v)=0$, then 
\begin{align}
&{ q^{k}u\over[k+1]} \,K^-_{-1}G^-_{-1/2}|h,u\rangle+T_{-1}G^-_{-1/2}|h,u\rangle
+
q^{-k}(1-q^2)(1-q^{-4}u^2)
G^-_{-3/2}|h,u\rangle,\nonumber
\end{align}
 is a singular vector.
 \end{prp}

\begin{rmk}
Write 
\begin{equation}\label{level2/3-singular}
\vert \chi(3/2,\pm 1)\rangle=
\Bigl({ q^{\mp k}u\over[k+1]} \,K^-_{-1}G^\pm_{-1/2}+T_{-1}G^\pm_{-1/2}
+
q^{\pm k}(1-q^{\mp 2})(1-q^{\pm 4}u^2)
G^\pm_{-3/2}\Bigr)|h,u\rangle,
\end{equation}
for {\it arbitrary} $u$ and $v$.
We can check by using the Mathematica that
\begin{align}\label{level2/3-vanishing}
&
\Bigl({ q^{\mp k}(q^{\pm (k+2)}u)\over[k+1]} \,K^-_{-1}G^\pm_{-1/2}+T_{-1}G^\pm_{-1/2} \CR
&\qquad +
q^{\pm k}(1-q^{\mp 2})(1-q^{\pm 4}(q^{\pm (k+2)}u)^2)
G^\pm_{-3/2}\Bigr)\vert \chi(3/2,\pm 1)\rangle=0.
\end{align}
This explains the factor ${1/(1+p^{3/2}x^{\pm 1}})$ in 
the character of the descendants of $\vert \chi(3/2,\pm 1)\rangle$  given by 
$$
{1\over 1+p^{3/2}x^{\pm 1}}
\prod_{i=0}^\infty {(1+p^{i+1/2}x)(1+p^{i+1/2}x^{-1})\over (1-p^{i+1}) (1-p^{i+1})}
=
\sum_{n\in {1\over 2}\mathbb{Z}_{\geq 0} }\sum_{j\in \mathbb{Z}}
\widetilde{P}^{\rm NS}(n,j,\pm 3/2) p^n x^j.
$$
\end{rmk}

It should be emphasized that the vanishing property \eqref{level2/3-vanishing} is valid,
even if \eqref{level2/3-singular} is not a singular vector.
In fact when the parameters $u$ and $v$ meet the condition for \eqref{level2/3-singular} to be a singular vector,
which is an image of the fermionic screening operator (see Appendix \ref{app:screening}),
\eqref{level2/3-vanishing} is expected by the fermionic nature of the screening.
We observe that when we write the product of operators in \eqref{level2/3-singular} and \eqref{level2/3-vanishing}
in terms of the PBW basis, the coefficients are independent of $h$, which may explain the validity of \eqref{level2/3-vanishing} 
for generic parameters.

\subsubsection{Case $n=2$, $j=\pm 2,0$}
For the subspace with the $p$-degree $2$, 
we have two vectors $G^+_{-3/2}G^+_{-1/2}|h,u\rangle$, 
$G^-_{-3/2}G^-_{-1/2}|h,u\rangle$ with  the $x$-degrees $+2$ and $-2$. 
We have nine vectors for the subspace with the $p$-degree $2$ and the $x$-degree $0$:
\begin{align*}
&K^-_{-2}|h,u\rangle,K^-_{-1}K^-_{-1}|h,u\rangle, K^-_{-1}T_{-1}|h,u\rangle, 
K^-_{-1}G^+_{-1/2}G^-_{-1/2}|h,u\rangle, T_{-2}|h,u\rangle, T_{-1}T_{-1}|h,u\rangle, \\
&T_{-1}G^+_{-1/2}G^-_{-1/2}|h,u\rangle,
G^+_{-3/2}G^-_{-1/2}|h,u\rangle,G^+_{-1/2}G^-_{-3/2}|h,u\rangle.
\end{align*}
We have
\begin{align*}
&{\rm det}^{\rm NS}_{2,+2}=q^{-4k-4} g(+1;u,v)g(+3;u,v),\\
&{\rm det}^{\rm NS}_{2,-2}=q^{+4k+4} g(-1;u,v)g(-3;u,v).
\end{align*}
A remark is in order
as to the calculation of the nine by nine determinant ${\rm det}^{\rm NS}_{2,0}$ by Mathematica. 
It seems not easy to have ${\rm det}^{\rm NS}_{2,0}$ factorized within a 
reasonably practical time duration for arbitrary $q,k,u,v$. 
However, if we substitute various prime numbers to some of the variable $u,v$ and $k$
we can easily factorize the reduced determinant, providing us with a possibility to guess 
the exact formula. Hence we conjecture that
\begin{align*}
&{\rm det}^{\rm NS}_{2,0}=[2]^2[k+2]^8[2k+4]^2\\
&\times g(+3;u,v) g(+1;u,v)^3g(-1;u,v)^3g(-3;u,v)\nonumber\\
&\times f(1,1;u,v)^3 f(2,1;u,v)f(1,2;u,v) .
\end{align*}

\begin{prp}
If we have $f(2,1;u,v)=0$, then 
\begin{align*}
&c_1K^-_{-2}|h,u\rangle+c_2K^-_{-1}K^-_{-1}|h,u\rangle+c_3 K^-_{-1}T_{-1}|h,u\rangle+c_4K^-_{-1}G^+_{-1/2}G^-_{-1/2}|h,u\rangle
\\
&+c_5 T_{-2}|h,u\rangle+c_6T_{-1}T_{-1}|h,u\rangle+c_7 T_{-1}G^+_{-1/2}G^-_{-1/2}|h,u\rangle\\
&+
c_8G^+_{-3/2}G^-_{-1/2}|h,u\rangle+c_9G^+_{-1/2}G^-_{-3/2}|h,u\rangle,
\end{align*}
is a singular vector, where
\begin{align*}
&c_1=-q^{3k+4}u^4 {(q-q^{-1})^2[k+2]^2\over (u-q^{k+3})[k+1]}
\biggl(1-{[k+1]\over q^{k+2}u [k+2]} -{1\over q^{k-1} u}\\
&\qquad +
{[k+1]^2\over q^{2k+2}u^2 [k+2]^2}+
{[k+1]\over q^{2k+1}u^3 [k+2]^2}
\biggr),\\
&c_2=q^{2k+2}u^2{[k+2]^2\over [k+1]^2}
\biggl( 1-{[2][k+1]\over q^{k+1}u [k+2]}+
{[k+1]^2\over q^{2k+2}u^2 [k+2]^2}
\biggr),\\
&c_3=q^{k+1}u {[k+2]\over [k+1]}\biggl( 2-{[2][k+1]\over q^{k+1}u [k+2]}\biggr),\\
&c_4=
q^{2k+4}u^3 {(q-q^{-1})^3[2][k+2]^2\over  (u-q^{k+1})(u-q^{k+3})[k+1]}
\biggl(1-{[2k+3]\over u[k+2]}+
{[4][k+1]\over u^2[2]^2[k+2]}\biggr),\\
&c_5=
-q^{k+1}u^2 (q-q^{-1})^2{[k+2]^2\over [k+1]}
\biggl(1- {[k+1]^2\over q^{k+1}u[k+1]^2} \biggr),\\
&c_6=1,\\
&c_7=
q^{2k+4}u^2 {(q-q^{-1})^3[2k+4]\over  (u-q^{k+1})(u-q^{k+3})}
\biggl( 1-{[2][k+2]\over u[2k+4]}\biggr),\\
&c_8=
-q^{2k+5}u^3 {(q-q^{-1})^4(q^{k+1}u-1)[k+2]^2\over  (u-q^{k+1})(u-q^{k+3})}
\biggl( 1-{[k+3]\over q^{-k}u[k+2]}+{1\over u[k+2]}\biggr),\\
&c_9=
q^{k+2}u^3 {(q-q^{-1})^4[k+2]^2\over  u-q^{k+3}}
\biggl( 1-{[k+3]\over q^k u[k+2]}+{1\over u[k+2]}\biggr).
\end{align*}
\end{prp}
We omit writing the singular vector for the case $f(1,2;u,v)=0$.

\subsubsection{Case $n=5/2$, $j=\pm 1$}
We have for the subspace with $p$-degree $5/2$, nine vectors 
\begin{align*}
&K^-_{-2}G^+_{-1/2}|h,u\rangle,K^-_{-1}K^-_{-1}G^+_{-1/2}|h,u\rangle, K^-_{-1}T_{-1}G^+_{-1/2}|h,u\rangle,\\
&K^-_{-1}G^+_{-3/2}|h,u\rangle, T_{-2}G^+_{-1/2}|h,u\rangle, T_{-1}T_{-1}G^+_{-1/2}|h,u\rangle, \\
&T_{-1}G^+_{-3/2}|h,u\rangle,
G^+_{-5/2}|h,u\rangle,G^+_{-3/2}G^+_{-1/2}G^-_{-1/2}|h,u\rangle,
\end{align*}
with $x$-degree $+1$, and nine vectors
\begin{align*}
&K^-_{-2}G^-_{-1/2}|h,u\rangle,K^-_{-1}K^-_{-1}G^-_{-1/2}|h,u\rangle, K^-_{-1}T_{-1}G^-_{-1/2}|h,u\rangle,\\
&K^-_{-1}G^-_{-3/2}|h,u\rangle, T_{-2}G^-_{-1/2}|h,u\rangle, T_{-1}T_{-1}G^-_{-1/2}|h,u\rangle, \\
&T_{-1}G^-_{-3/2}|h,u\rangle,
G^+_{-1/2}G^-_{-3/2}G^-_{-1/2}|h,u\rangle,G^-_{-5/2}|h,u\rangle,
\end{align*}
with $x$-degree $-1$.

We have the conjecture:
\begin{align*}
&{\rm det}^{\rm NS}_{5/2,+1}=q^{-9k+1} 
[2]^2[k+2]^8[2k+4]^2\\
&\times
g(-1;u,v)g(+1;u,v)^6 g(+3;u,v)^3 g(+5;u,v)\\
&\times  f(1,1;u,v)^3 f(2,1;u,v) f(1,2;u,v),\\
&{\rm det}^{\rm NS}_{5/2,-1}=q^{+9k-1} 
[2]^2[k+2]^8[2k+4]^2\\
&\times
g(+1;u,v)g(-1;u,v)^6 g(-3;u,v)^3 g(-5;u,v)\\
&\times  f(1,1;u,v)^3 f(2,1;u,v) f(1,2;u,v).
\end{align*}

We omit writing the singular vectors for the cases $g(\pm5;u,v)=0$.


\subsection{Verma modules in the {\rm R} sector}
\begin{dfn}
Let $h$ and $u$ be complex parameters. Consider the {\rm R} sector of $\Svir$.
Let $|h,u\rangle$ be the
highest weight vector satisfying the conditions
\begin{align}\label{hwvR}
&K^\pm_0 |h,u\rangle=q^{k/2}u |h,u\rangle,\qquad 
T_0|h,u\rangle=q^{-1}h |h,u\rangle,\nonumber\\
&K^+_m |h,u\rangle=T_m |h,u\rangle=0\qquad (m>0), \\
&G^+_m |h,u\rangle=0\qquad (m\geq 0),\nonumber \\
&G^-_m |h,u\rangle=0\qquad (m>0).\nonumber
\end{align}
The Verma module $M_{h,u}$ in the {\rm R} sector is defined to be 
the left $\Svir$ module $M_{h,u}=\Svir|h,u\rangle$.
\end{dfn}

\begin{dfn}
The dual Verma module $M^*_{h,u}$ in the {\rm R} sector is defined to be the right $\Svir$ module
$M^*_{h,u}=\langle h,u|\Svir$, with the vector $\langle h,u|$ 
satisfying the conditions
\begin{align}
&\langle h,u|h ,u\rangle=1,\nonumber\\
&\langle h,u|K^\pm_0 = q^{k/2}u \langle h,u|,\qquad 
\langle h,u|T_0= q^{-1}h \langle h,u|,\nonumber\\
&\langle h,u|K^-_m = \langle h,u|T_m =0\qquad (m<0), \\
&\langle h,u|G^-_m =0\qquad (m\leq 0),\nonumber\\
&\langle h,u|G^+_m =0\qquad (m<0).\nonumber
\end{align}
\end{dfn}

Recall that we have the character formula\footnote{In \cite{Boucher:1986bh} the factor $1+x^{-1}$ in the numerator 
is replaced with $x^{\frac{1}{2}} + x^{-\frac{1}{2}}$ so that the character is symmetric in $x$ and $x^{-1}$.}
\begin{align}
{\rm ch}_{\rm R}(p,x) = \Tr_{\mathcal{V}_{\mathrm{R}}} \left( p^{\mathsf{L}_0 - h}x^{\mathsf{I}_0 - u}\right) 
= \prod_{i=0}^\infty {(1+p^{i+1}x)(1+p^{i}x^{-1})\over (1-p^{i+1}) (1-p^{i+1})}, 
\qquad |p|, |x| <1 \label{ch-R}
\end{align}
for the Verma module $\mathcal{V}_{\mathrm{R}}$ of the ordinary $\mathcal{N}=2$ superconformal algebra in the {\rm R} sector,
where the ground states are doubly degenerate and they are connected by the zero modes of the super currents. 
$u$ and $u-1$ are the $U(1)$ charges the ground states, which leads to the factor $1+x^{-1}$ in the numerator.
As is the case of {\rm NS} sector, we assume that we have the same  character for $M_{h,u}$ by the rule
\begin{align*}
&K^\pm_{-m},T_{-m}, G^\pm_{-m} \mbox{ have the $p$-degree } m,\\
&K^\pm_{-m},T_{-m}\mbox{ have the $x$-degree } 0,\\
&G^\pm_{-m} \mbox{ have the $x$-degree } \pm 1.
\end{align*}

\begin{dfn}
A finite strictly decreasing sequence of positive integers 
$\alpha=(\alpha_1,\alpha_2,\ldots,\alpha_l)$
($\alpha_i\in \mathbb{Z}, \alpha_1> \alpha_2> \cdots > \alpha_l>0$)
is called a fermionic partition in the {\rm R} sector. We denote by $\ell(\alpha)=l$ the length of $\alpha$. 
The set of  fermionic partitions in the {\rm R} sector is denoted by $\mathcal{P}^{\rm R}$.
A finite strictly decreasing sequence of non negative integers 
$\beta=(\beta_1,\beta_2,\ldots,\beta_l)$
($\beta_i\in \mathbb{Z}, \beta_1> \beta_2> \cdots > \beta_l\geq 0$)
is called a fermionic partition  with zero mode in the {\rm R} sector. We denote by $\ell(\beta)=l$ the length of $\beta$. 
The set of fermionic partitions with zero mode in the {\rm R} sector is denoted by $\mathcal{P}^{\rm R}_0$.
\end{dfn}

Let $\lambda=(\lambda_1,\ldots,\lambda_l)$
be a partition, 
$\alpha=(\alpha_1,\ldots,\alpha_a)$
be a fermionic partition in the {\rm R} sector, and 
$\beta=(\beta_1,\beta_2,\ldots,\beta_l)$ be a fermionic partition with zero mode in the {\rm R} sector.
We introduce the following notations for the ordered products of generators
\begin{align*}
&K^-_{-\lambda}=K^-_{-\lambda_1}K^-_{-\lambda_2}\cdots K^-_{-\lambda_l},\qquad 
K^+_{\lambda}=K^+_{\lambda_l}\cdots K^+_{\lambda_2}K^+_{\lambda_1}
\\
&T_{-\lambda}=T_{-\lambda_1}T_{-\lambda_2}\cdots T_{-\lambda_l},\qquad \,\,\,\,\,\,
T_{\lambda}=T_{\lambda_l}\cdots T_{\lambda_2}T_{\lambda_1},\\
&G^+_{-\alpha}=G^+_{-\alpha_1}G^+_{-\alpha_2}\cdots G^+_{-\alpha_a},\qquad
G^-_{\alpha}=G^-_{\alpha_a}\cdots G^-_{\alpha_2} G^-_{\alpha_1},\\
&G^-_{-\beta}=G^-_{-\beta_1}G^-_{-\beta_2}\cdots G^-_{-\beta_a},\qquad
G^+_{\beta}=G^+_{\beta_a}\cdots G^+_{\beta_2} G^+_{\beta_1}.
\end{align*}
Then, for a pair of partitions
$\lambda=(\lambda_1,\ldots,\lambda_l),\mu=(\mu_1,\ldots,\mu_m)$, and 
a pair of fermionic partitions in the {\rm NS} sector
$\alpha=(\alpha_1,\ldots,\alpha_a),\beta=(\beta_1,\ldots,\beta_b)$, 
set 
\begin{align*}
&|\lambda,\mu,\alpha,\beta\rangle=|\lambda,\mu,\alpha,\beta;h,u\rangle=
K^-_{-\lambda}T_{-\mu}G^+_{-\alpha}G^-_{-\beta}|h,u\rangle,\\
&\langle \lambda,\mu,\alpha,\beta|=\langle\lambda,\mu,\alpha,\beta;h,u|=
\langle h,u|G^+_{\beta}G^-_{\alpha}T_{\mu}K^+_{\lambda}.
\end{align*}
Note that we have the lexicographical orderings on the sets $\mathcal{P}$, $\mathcal{P}^{\rm R}$ and $\mathcal{P}^{\rm R}_0$. 
Hence one may introduce the associated total ordering on the set
$\mathcal{P}\times \mathcal{P}\times \mathcal{P}^{\rm R}\times \mathcal{P}^{\rm R}_0$.

\begin{prp}
The ordered collection $(|\lambda,\mu,\alpha,\beta\rangle)_{\lambda,\mu \in \mathcal{P}, 
\alpha\in \mathcal{P}^{\rm R},\beta \in \mathcal{P}^{\rm R}_0}$ forms a basis of  $M_{h,u}$.
Similarly, $(\langle\lambda,\mu,\alpha,\beta|)_{\lambda,\mu \in \mathcal{P}, 
\alpha\in \mathcal{P}^{\rm R},\beta \in \mathcal{P}^{\rm R}_0}$
forms a basis of $M^*_{h,u}$.
\end{prp}

\subsection{Kac determinants in the {\rm R} sector}
We study the Kac determinants associated with the Verma module $M_{h,u}$ in the {\rm R} sector.
We denote by ${\rm det}_{n,j}={\rm det}_{n,j}^{\rm R}(h,u)$
the Kac determinant in the {\rm R} sector associated with the subspace in $M_{h,u}$ 
having the $p$-degree $n$ and the $x$-degree $j$.

\begin{dfn}
Set
\begin{align}
\sum_{n\in \mathbb{Z}_{\geq 0}}\sum_{j\in \mathbb Z} P^{\rm R} (n,j)
p^n x^j=
\prod_{i=0}^\infty {(1+p^{i+1}x)(1+p^{i}x^{-1})\over (1-p^{i+1}) (1-p^{i+1})},
\end{align}
and set for $\ell \in \mathbb{Z}$ 
\begin{align}
\sum_{n\in \mathbb{Z}_{\geq 0}}\sum_{j\in \mathbb Z} 
\widetilde{P}^{\rm R} (n,j;\ell)
p^n x^j={1\over 1+p^{|\ell|} x^{{\rm sgn}(\ell) }}
\prod_{i=0}^\infty {(1+p^{i+1}x)(1+p^{i}x^{-1})\over (1-p^{i+1}) (1-p^{i+1})},
\end{align}
where ${\rm sgn}(\ell)=1$ for $\ell>0$ and ${\rm sgn}(\ell)=-1$ for $\ell\leq 0$.
Note that ${\rm sgn}(0)=-1$ in our convention. 
\end{dfn}

Recall the parametrization \eqref{huv} of $h$ in terms of the parameter $v$
%

\begin{con}\label{Kac-det-R}
We have
\begin{align}
{\rm det}_{n,j}^{\rm R}(h,u) =
{\rm cst.}
\prod_{\scriptstyle r,s \in \mathbb{Z}_{> 0}
\atop 
\scriptstyle 1\leq rs\leq n} \Bigl(f(r,s;u,v)\Bigr)^{P^{\rm R}(n-rs,j)}
\prod_{\ell\in \mathbb{Z}}
\Bigl(g(2\ell-1;u,v)\Bigr)^{\widetilde{P}^{\rm R}(n-|\ell|,j-{\rm sgn}(\ell);\ell)},
\end{align}
where {\rm cst.} is a certain non zero constant not depending on $u$ or $v$.
\end{con}

As in the case of \rm{NS} sector we can take the limit $q \to 1$ of Conjecture \ref{Kac-det-R}, 
which reproduces the formula in \cite{Kato:1986td}.
For an evidence for the Conjecture \ref{Kac-det-R}, see Appendix \ref{app:screening}, where we investigate 
the screening operators which are intertwiners among Fock representations.

\section{Conformal field theory limit of $\Svir$}
\label{sec:CFT-limit}

In this section, we set $q=e^\hbar$ and investigate the $\hbar$-expansions of generators and relations of $\Svir$, 
extracting the ordinary ${\mathcal N}=2$ superconformal Lie superalgebra 
defined by the generators
\begin{align*}
&c,\qquad \I_m, \qquad \Lv_m \qquad (m \in \mathbb{Z}),\\
&\G^\pm_m \qquad (m \in \mathbb{Z}+{1\over 2}\mbox{ for {\rm NS} sector}, \,\,
m \in \mathbb{Z}\mbox{ for {\rm R} sector}),
\end{align*}
and the relations \cite{Ademollo:1975an}, \cite{Ademollo:1976pp},
\begin{align*}
&c:\mbox{  central},\\
&\{\G^\pm_m,\G^\pm_n\}=0,\\
&\{\G^+_m,\G^-_n\}=2 \Lv_{m+n} +(m-n) \I_{m+n} +{c\over 3} \left( m^2-{1\over 4}\right) \delta_{m+n,0}\\
&[\Lv_m,\G^\pm_n]=\left({m\over 2}-n \right)\G^\pm_{m+n},\\
&[\I_m,\G^\pm_n]=\pm \G^\pm_{m+n},\\
&[\I_m,\I_n]= {c\over 3} m \delta_{m+n,0},\\ 
&[\I_m,\Lv_n]= m \I_{m+n},\\
&[\Lv_m,\Lv_n]=(m-n)\Lv_{m+n}+ {c\over 12} m(m^2-1) \delta_{m+n,0}.
\end{align*}

\subsection{$\hbar$-expansions of the generators}
For simplicity set $K^\pm_m=0$ for $\mp m>0$.
We assume that the generators $K^\pm_m,G^\pm_m$ and $T_m$ of $\Svir$ are expanded in
positive powers in $\hbar$ as
\begin{align*}
&K^\pm_m=\sum_{i\geq 0} \hbar^i K^{\pm(i)}_m,\qquad 
G^\pm_m=\sum_{i\geq 0} \hbar^i G^{\pm(i)}_m,\qquad 
T_m=\sum_{i\geq 0} \hbar^i T^{(i)}_m.
\end{align*}
Introduce the generating functions $K^{\pm(i)}(z)$, $G^{\pm(i)}(z)$ and $T^{(i)}(z)$ $(i\geq 0)$ as
\begin{align*}
&K^{\pm(i)}(z)=\sum_{m \in \mathbb{Z}} K^{\pm(i)}_m z^{-m},\qquad G^{\pm(i)}(z)=\sum_{m \in \mathbb{Z}^\mathrm{A}} G^{\pm(i)}_m z^{-m},\qquad
T^{(i)}(z)=\sum_{m \in \mathbb{Z}} T^{(i)}_m z^{-m}.
\end{align*}
Then we have
\begin{align*}
&K^\pm(z)=\sum_{i\geq 0} \hbar^i K^{\pm(i)}(z),\qquad G^\pm(z)=\sum_{i\geq 0} \hbar^i G^{\pm(i)}(z),\qquad 
T(z)=\sum_{i\geq 0} \hbar^i T^{(i)}(z),
\end{align*}
We also need the $\hbar$-expansions of $K(z)=K^-(z)K^+(z)$ as
\begin{align*}
&K(z)=\sum_{i\geq 0} \hbar^i K^{(i)}(z),\qquad
K^{(i)}(z)=\sum_{m \in \mathbb{Z}} K^{(i)}_m z^{-m}.
\end{align*}

First, in view of the relations for the Heisenberg subalgebra (\ref{H-H}) and (\ref{H_0}), 
we assume that the element $H_0$ does not depend on $\hbar$, and 
the Heisenberg generators $H_m$ ($m\neq 0$) have the
expansions as $H_m=\sum_{i\geq 0} \hbar^i H^{(i)}_m$. Then we have the following 
description of the leading terms of the $\hbar$ expansions for $K^{\pm}(z)$, and several low lying 
terms of the $\hbar$ expansions for $K(z)$.

\begin{prp}\label{prp:K-ser}
We have
\begin{align}
&K^{\pm (0) }_m=\delta_{m,0},\qquad K^{\pm (0) }(z)=1,\\
&K^\pm(z)=1+\sum_{i\geq 1} \hbar^i K^{\pm(i)}(z),\\
&K(z)=1+\sum_{i\geq 1} \hbar^i K^{(i)}(z), \label{K=1+}\\
&K^{(0)}(z)=1,\CR
&K^{(1)}(z)=K^{-(1)}(z)+K^{+(1)}(z),\CR
&K^{(2)}(z)=K^{-(2)}(z)+K^{-(1)}(z)K^{+(1)}(z)+K^{+(2)}(z). \nonumber
\end{align}
\end{prp}

Combined with the $\hbar$-expansions of $K(z)=K^-(z)K^+(z)$, 
the relation (\ref{rr-9}) plays a crucial role 
to find out the basic structures and the roles of the several
leading terms of the $\hbar$-expansion of the generators of $\Svir$.

\begin{prp}\label{Exp-1}
By using the $\hbar$ expansions of  $K^\pm(z), K(z)$ in Proposition \ref{prp:K-ser}, and the $\hbar$ expansion of 
the relation (\ref{rr-9}),
we have that 
\begin{enumerate}
\item $G^{\pm(0)}(z)$ is the first nontrivial element in the $\hbar$-expansion of  $G^{\pm}(z)$,
\item $T^{(0)}(z)= {k\over k+1}$, 
$T^{(1)}(z)= -{1\over k+1}K^{(1)}(z)$, and $T^{(2)}(z)$ is the first newly appearing nontrivial element in the $\hbar$-expansion of $T(z)$. 
\end{enumerate}
\end{prp}

For a proof of the second statement see the next subsection (\eqref{T0} and \eqref{T1}). 
Hence, with these particular elements explicitly written, we have 
\begin{align*}
&G^\pm(z)= G^{\pm(0)}(z)+{\mathcal O}(\hbar^1),\\
&K^\pm(z)= 1+\hbar\,\, K^{\pm(1)}(z) +\hbar^2\,\, K^{\pm(2)}(z) +{\mathcal O}(\hbar^3),\\
&T(z)= {k\over k+1}+\hbar\,\, \Biggl(-{1\over k+1}K^{(1)}(z)\Biggr) +\hbar^2\,\, T^{(2)}(z) +{\mathcal O}(\hbar^3).
\end{align*}

We decided that (1) the odd generators $G^\pm(z)$ should be expanded up to the order of $\hbar^0$,
and (2) even generators $K^\pm(z),K(z),T(z)$ should be expanded up to the order of $\hbar^2$. 
Then we know to what extent we should perform the exact $\hbar$-expansions of the relations as follows;
\begin{table}[h]
\begin{tabular}{c|c|c|c}
types & relations & required orders & worked out in \\  [2mm] \hline 
$G^\pm$ vs. $G^\pm$&(\ref{rr-8}) and (\ref{rr-9})&up to  $\hbar^0$ & \S\ref{sub42}\\[2mm]
$K^\pm$ vs. $G^\pm$& (\ref{rr-6}) and (\ref{rr-7})& up to  $\hbar^2$ & \S\ref{sub43} \\[2mm]
$G^\pm$ vs. $T$ &(\ref{rr-10})&up to  $\hbar^2$ &  \S\ref{sub44}  \\[2mm]
$K^\pm$ vs. $T$&(\ref{rr-4}) and (\ref{rr-5})&up to  $\hbar^4$ & \S\ref{sub45} \\[2mm]
$K^\pm$ vs. $K^\pm$& (\ref{rr-2}) and (\ref{rr-3})& up to  $\hbar^4$ & \S\ref{sub46} \\[2mm]
$T$ vs. $T$ & (\ref{rr-11})& up to  $\hbar^4$ & \S\ref{sub47}
\end{tabular}
\end{table}

\begin{dfn}\label{CFTcurrents}
Rescaling or combining the above leading elements $G^{\pm(0)}(z)$, $K^{(1)}(z)$, $K^{(2)}(z)$ and $T^{(2)}(z)$,
set\footnote{In the $q$-deformed case it is convenient to define the generating currents without 
the degree shift by the conformal weight. When we compute the $\hbar$ expansion, the commutation relation is more tractable 
than the operator product expansion.}
\begin{align*}
&\G^\pm(z)=\sum_{m \in \mathbb{Z}^\mathrm{A}}\G^\pm_m z^{-m}=\sqrt{{2\over k+2}}G^{\pm(0)}(z),\\
&\I(z)=\sum_{m \in \mathbb{Z}}\I_m z^{-m}={1\over 2(k+2)}K^{(1)}(z),\\
&\Lv(z)=\sum_{m \in \mathbb{Z}}\Lv_m z^{-m}= {1\over 4(k+2)} T^{(2)}(z)+{1\over 4(k+1)(k+2)}K^{(2)}(z) +{k(2k+1)\over 24(k+1)(k+2)}.
\end{align*}
\end{dfn}

\begin{dfn}
Denote by $D$ the Euler differential $D=z{\partial \over \partial z}$.
Set
\begin{align*}
&
(D\delta^{\rm A})(z)=\sum_{m \in \mathbb{Z}^\mathrm{A}} m z^m ,\qquad (D^2\delta^{\rm A})(z)=\sum_{m \in \mathbb{Z}^\mathrm{A}} m^2 z^m,
\end{align*}
where $\mathrm{A}=\mathrm{NS}$ or $\mathrm{R}$. Recall that 
$\mathbb{Z}^\mathrm{NS} = \mathbb{Z} + \frac{1}{2}$ and $\mathbb{Z}^\mathrm{R} = \mathbb{Z}$.
The definition for $\delta(z)$ is the same as $\delta^{\rm R}(z)$.
\end{dfn}

\begin{thm}\label{CFT-Theorem}
The elements $\G^\pm(z)$, $\I(z)$ and $\Lv(z)$ satisfy the relations for the 
${\mathcal N}=2$ superconformal algebra with the central charge\footnote{
The central charge agrees with the sum of the central charges of the level $k$ parafermion and a $\mathfrak{u}_1$ free boson.}
\begin{equation}
c=\frac{3k}{k+2}.
\end{equation}
Namely we have;
\begin{align}
\label{N=2comrel-1}
&\{\G^\pm(z),\G^\pm(w)\}=0,\\
\label{N=2comrel-2}
&\{\G^+(z),\G^-(w)\}=2 \delta^{\rm A} \left({w\over z}\right)\Lv(w)+
(D\delta^{\rm A}) \left({w\over z}\right)\Bigl(\I(w)+\I(z)\Bigr) \CR
&\qquad\qquad  +
{c\over 3} \left(  (D^2\delta^{\rm A}) \left({w\over z}\right) -
{1\over 4}\delta^{\rm A} \left({w\over z}\right)
\right) ,\\
\label{N=2comrel-3}
&[\Lv(z),\G^\pm(w)]={1\over 2}(D\delta) \left({w\over z}\right)\G^\pm(w)+
(D\delta^{\rm A}) \left({w\over z}\right)\G^\pm(z),\\
\label{N=2comrel-4}
&[\I(z),\G^\pm(w)]=\pm~\delta\left({w\over z}\right)\G^\pm(w),\\
\label{N=2comrel-5}
&[\I(z),\I(w)]= {c\over 3} (D\delta) \left({w\over z}\right),\\ 
\label{N=2comrel-6}
&[\I(z),\Lv(w)]=  (D\delta) \left({w\over z}\right) \I(w),\\
\label{N=2comrel-7}
&[\Lv(z),\Lv(w)]=
 (D\delta) \left({w\over z}\right)\Bigl( \Lv(z)+\Lv(w)\Bigr) 
+ {c\over 12} \Bigl((D^3\delta) \left({w\over z}\right)
- \delta\left({w\over z}\right) \Bigr).
\end{align}
\end{thm}

In \eqref{N=2comrel-3} the first term should involve $D\delta$, while we have 
$D\delta^{\rm A}$ in the second term. We can see this as follows;
\begin{align*}
[\Lv(z),\G^\pm(w)] &= \sum_{m \in \mathbb{Z}} \sum_{n \in \mathbb{Z}^\mathrm{A}} z^{-m} w^{-n} [L_m, G^\pm_n] 
= \sum_{m \in \mathbb{Z}} \sum_{n \in \mathbb{Z}^\mathrm{A}} z^{-m} w^{-n} \left( \frac{m}{2} -n \right)G^\pm_{m+n} \\
&= \sum_{m \in \mathbb{Z}} \sum_{n \in \mathbb{Z}^\mathrm{A}} \left[ \frac{m}{2} \left(\frac{w}{z}\right)^m w^{-n-m}
-n \left(\frac{z}{w}\right)^n z^{-n-m} \right]G^\pm_{m+n}  \\
&= \frac{1}{2} D\delta \left(\frac{w}{z} \right) G^\pm(w) +  D\delta^\mathrm{A} \left(\frac{w}{z} \right) G^\pm(z).
\end{align*}

The following table shows where each relation is proved.\footnote{Some of the relations are proved in the form where 
$z$ and $w$ are exchanged.}

\begin{table}[h]
\begin{tabular}{l|c|c|c|c|c|c|c}
Relation & \eqref{N=2comrel-1} & \eqref{N=2comrel-2} & \eqref{N=2comrel-3} & \eqref{N=2comrel-4} 
& \eqref{N=2comrel-5} & \eqref{N=2comrel-6} & \eqref{N=2comrel-7} \\
\hline
Proved as & Prop.\ref{P1} & Prop.\ref{P2}  & Prop.\ref{P3}  & Prop.\ref{P4} & Prop.\ref{P5} & Prop.\ref{P6} & Prop.\ref{P7}
\end{tabular}
\end{table}

\subsection{$\hbar$-expansions of the relations for $G^{\pm}$ vs. $G^{\pm}$ }\label{sub42}
We start our investigation of the relations from the $\hbar$-expansions of (\ref{rr-8}) and  (\ref{rr-9}) up to 
the order of $\hbar^0$.

\begin{prp}\label{P1}
{}Taking the terms of the order of $\hbar^0$ in the $\hbar$-expansions of the relation (\ref{rr-8}),
we have 
\begin{align*}
G^{\pm(0)}(z) G^{\pm(0)}(w)+G^{\pm(0)}(w)G^{\pm(0)}(z)=0.
\end{align*}
Namely we have
\begin{align*}
\G^{\pm}(z) \G^{\pm}(w)+\G^{\pm}(w)\G^{\pm}(z)=0.
\end{align*}
\end{prp}

Next, we turn to the $\hbar$ expansion of (\ref{rr-9}). We show that
three nontrivial relations will appear by considering the 
terms of the orders of $\hbar^{-2}$, $\hbar^{-1}$ and $\hbar^{0}$. 

\begin{prp}
Using (\ref{K=1+}), we have
\begin{align*}
\mbox{{\rm RHS of (\ref{rr-9})}}= {1\over 4\hbar^{2}}\left( T^{(0)}(w)-{k\over k+1}\right)+{\mathcal O}(\hbar^{-1}).
\end{align*}
On the other hand  we have $\mbox{\rm LHS of (\ref{rr-9})}={\mathcal O}(\hbar^{0})$. 
Hence we have the relation
\begin{align}
&T^{(0)}(z)={k\over k+1}.\label{T0}
\end{align}
\end{prp}

\begin{prp}
Using (\ref{K=1+}) and (\ref{T0}), we have
\begin{align*}
\mbox{{\rm RHS of (\ref{rr-9})}}= {1\over 4\hbar}\left( T^{(1)}(w)+{1\over k+1}K^{(1)}(w)\right)+{\mathcal O}(\hbar^{0}).
\end{align*}
Hence by the same reason as above, we have the relation 
\begin{align}
&T^{(1)}(z)=-{1\over k+1}K^{(1)}(z).\label{T1}
\end{align}
\end{prp}

%

\begin{prp}\label{P2}
Using (\ref{K=1+}), (\ref{T0}) and (\ref{T1}), 
taking the terms of the order of $\hbar^0$ in the $\hbar$ expansion of (\ref{rr-9}), we have
\begin{align}\label{G0G0}
&G^{+(0)}(z)G^{-(0)}(w)+G^{-(0)}(w)G^{+(0)}(z) \CR
&=\left({1\over 4} T^{(2)}(w)+{1\over 4} (DK^{(1)})(w) +{1\over 4(k+1)}K^{(2)}(w) -{k(k+2)\over 24(k+1)}\right)
\delta^{\rm A}\left({w\over z}\right)\nonumber\\
&\qquad+{1\over 2}K^{(1)}(w) (D\delta^{\rm A})\left({w\over z}\right)+{k\over 2}(D^2\delta^{\rm A})\left({w\over z}\right),
\end{align}
which is recast as
\begin{align*}
&\G^{+}(z)\G^{-}(w)+\G^{-}(w)\G^{+}(z)\\
&=2 
\delta^{\rm A}\left({w\over z}\right)\Lv(w)
+(D\delta^{\rm A})\left({w\over z}\right) \Bigl(\I(z)+\I(w) \Bigr)+
{k\over k+2 }\Bigl( (D^2\delta^{\rm A})\left({w\over z}\right) -
{1\over 4} \delta^{\rm A}\left({w\over z}\right) \Bigr).
\end{align*}
\end{prp}

\begin{proof}
Rearranging  (\ref{G0G0}) slightly, we have
\begin{align*}
&{2\over k+2}\Bigl(G^{+(0)}(z)G^{-(0)}(w)+G^{-(0)}(w)G^{+(0)}(z)\Bigr)\\
&=2 \left({1\over 4(k+2)} T^{(2)}(w)+{1\over 4(k+1)(k+2)}K^{(2)}(w) +{k(2k+1)\over 24(k+1)(k+2)}\right)
\delta^{\rm A}\left({w\over z}\right)\\
&\qquad+(D\delta^{\rm A})\left({w\over z}\right) \Bigl({1\over 2(k+2)}K^{(1)}(w) +{1\over 2(k+2)}K^{(1)}(z) \Bigr)\\
&\qquad+
{k\over k+2 }\Bigl( (D^2\delta^{\rm A})\left({w\over z}\right) -
{1\over 4} \delta^{\rm A}\left({w\over z}\right) \Bigr).
\end{align*}
\end{proof}

\subsection{$\hbar$ expansions of the relations for $K^{\pm}$ vs. $G^{\pm}$ }\label{sub43}
We study the $\hbar$ expansions of the relations  (\ref{rr-6}) and (\ref{rr-7}), 
up to the order of $\hbar^2$.

\begin{dfn}
Set
\begin{align*}
&s(z)={1\over 2}+\sum_{\ell>0} z^\ell,\qquad
(Ds)(z)=\sum_{\ell>0} \ell z^\ell,\qquad (D^2s)(z)=\sum_{\ell>0} \ell^2 z^\ell.
\end{align*}
\end{dfn}
\begin{lem}
Note that we have
\begin{align*}
&\delta(z)=s(z)+s(z^{-1}),\qquad s(z)^2={1\over 4}+(Ds)(z),\\
&(D\delta)(z)=(Ds)(z)-(Ds)(z^{-1}),\qquad (D^2\delta)(z)=(D^2s)(z)+(D^2s)(z^{-1}).
\end{align*}
\end{lem}

\begin{lem}
We have
\begin{align*}
&q^{\pm (k+2)}{1-w/z\over 1-q^{\pm 2(k+2)}w/z}=1\pm \hbar\,\, 2(k+2) s\left({w\over z} \right)\\
&\qquad +
\hbar^2 \,\,\Bigl( {(k+2)^2\over 2}+4(k+2)^2 (Ds)\left({w\over z} \right)\Bigr)+{\mathcal O}(\hbar^3).
\end{align*}
\end{lem}

\begin{prp}\label{P4}
We have no nontrivial relation by 
taking the terms of the order of $\hbar^0$ in the $\hbar$ expansions of the relations  (\ref{rr-6}) and (\ref{rr-7}).
Taking the terms of the order of $\hbar^1$ in them, we have
\begin{align}
&[K^{-(1)}(w) ,G^{\pm (0)}(z)]=\pm 2(k+2) s\left({w\over z}\right)G^{\pm (0)}(z),\label{K-1G0}\\
&[ G^{\pm (0)}(z),K^{+(1)}(w)]=\mp 2(k+2) s\left({z\over w}\right)G^{\pm (0)}(z),\label{G0K+1}
\end{align}
which are equivalent to the single equation
\begin{align}
&[K^{(1)}(w) ,G^{\pm (0)}(z)]=\pm 2(k+2) \delta\left({w\over z}\right)G^{\pm (0)}(z).\label{G0K}
\end{align}
Namely we have
\begin{align*}
&[\I(w) ,\G^{\pm }(z)]=\pm \delta\left({w\over z}\right)\G^{\pm }(z),
\end{align*}
which is \eqref{N=2comrel-5} with $z$ and $w$ being exchanged. 
\end{prp}

\begin{lem}
{}From (\ref{K-1G0}) and  (\ref{G0K+1}), we have
\begin{align}\label{K+K-G0}
&[K^{-(1)}(w) K^{+(1)}(w) ,G^{\pm (0)}(z)] \CR
&=
\pm 2(k+2) s\left({w\over z}\right)G^{\pm (0)}(z)K^{+(1)}(w)
\pm 2(k+2) s\left({z\over w}\right)K^{-(1)}(w)G^{\pm (0)}(z).
\end{align}
\end{lem}

\begin{proof}
We have
\begin{align*}
&[K^{-(1)}(w) K^{+(1)}(w) ,G^{\pm (0)}(z)]\\
=&~[K^{-(1)}(w)  ,G^{\pm (0)}(z)]K^{+(1)}(w)+K^{-(1)}(w) [K^{+(1)}(w) ,G^{\pm (0)}(z)]\\
=&
\pm 2(k+2) s\left({w\over z}\right)G^{\pm (0)}(z)K^{+(1)}(w)
\pm 2(k+2) s\left({z\over w}\right)K^{-(1)}(w)G^{\pm (0)}(z).
\end{align*}
\end{proof}

\begin{prp}
Taking the terms of the order of $\hbar^2$ in the $\hbar$ expansions of the relations  (\ref{rr-6}) and (\ref{rr-7}), 
we have
\begin{align}
&[K^{-(1)}(w) ,G^{\pm (1)}(z)]+[K^{-(2)}(w) ,G^{\pm (0)}(z)] \CR
&=\pm 2(k+2) s\left({w\over z}\right)\Bigl(G^{\pm (1)}(z)+G^{\pm (0)}(z)K^{-(1)}(w) \Bigr)\CR
&\qquad
+{(k+2)^2\over 2}G^{\pm (0)}(z)+4(k+2)^2 (Ds)\left({w\over z}\right) G^{\pm (0)}(z),\label{K-1G1}\\
&[G^{\pm (1)}(z),K^{+(1)}(w) ]+[G^{\pm (0)}(z),K^{+(2)}(w) ] \CR
&=\mp 2(k+2) s\left({z\over w}\right)\Bigl(G^{\pm (1)}(z)+K^{+(1)}(w)G^{\pm (0)}(z) \Bigr)\CR
&\qquad
+{(k+2)^2\over 2}G^{\pm (0)}(z)+4(k+2)^2 (Ds)\left({z\over w}\right) G^{\pm (0)}(z), \label{G1K+1}
\end{align}
which are equivalent to the single equation 
\begin{align}\label{KG-2}
&[K^{(1)}(w) ,G^{\pm (1)}(z)]+[K^{(2)}(w) ,G^{\pm (0)}(z)] \CR
=&\pm 2(k+2) \delta\left({w\over z}\right)\Bigl(G^{\pm (1)}(z)+K^{-(1)}(w)G^{\pm (0)}(z)
+G^{\pm (0)}(z)K^{+(1)}(w) \Bigr) \CR
=&\pm 2(k+2) \delta^{\rm A}\left({w\over z}\right)\Bigl(G^{\pm (1)}(w)+K^{-(1)}(w)G^{\pm (0)}(w)
+G^{\pm (0)}(w)K^{+(1)}(w) \Bigr),
\end{align}
under the condition (\ref{G0K}).
\end{prp}

\begin{proof}
Using (\ref{G0K}), we have from (\ref{K-1G1}) and  (\ref{G1K+1}) that
\begin{align*}
&[K^{-(1)}(w) ,G^{\pm (1)}(z)]+[K^{-(2)}(w) ,G^{\pm (0)}(z)]\\
=&\pm 2(k+2) s\left({w\over z}\right)\Bigl(G^{\pm (1)}(z)+K^{-(1)}(w)G^{\pm (0)}(z) \Bigr)
-{(k+2)^2\over 2}G^{\pm (0)}(z),\\
&[G^{\pm (1)}(z),K^{+(1)}(w) ]+[G^{\pm (0)}(z),K^{+(2)}(w) ]\\
=&\mp 2(k+2) s\left({z\over w}\right)\Bigl(G^{\pm (1)}(z)+G^{\pm (0)}(z) K^{+(1)}(w)\Bigr)
-{(k+2)^2\over 2}G^{\pm (0)}(z).
\end{align*}
Then by using (\ref{K+K-G0}), one finds that these are equivalent to the single equation
\begin{align*}
&[K^{(1)}(w) ,G^{\pm (1)}(z)]+[K^{(2)}(w) ,G^{\pm (0)}(z)]\\
=&\pm 2(k+2) \delta\left({w\over z}\right)\Bigl(G^{\pm (1)}(z)+K^{-(1)}(w)G^{\pm (0)}(z)
+G^{\pm (0)}(z)K^{+(1)}(w)
 \Bigr).
\end{align*}

\end{proof}

\subsection{$\hbar$ expansions of the relations for $G^{\pm}$ vs. $T$ }\label{sub44}

We study the $\hbar$ expansions
of the relations (\ref{rr-10}),
considering the terms of the orders up to $\hbar^2$.

\begin{prp}
We have no nontrivial relation by 
taking the terms of the order of $\hbar^0$ in the $\hbar$ expansions of the relations  (\ref{rr-10}).
Taking the terms of the order of $\hbar^1$ in them, we have
\begin{align*}
&-{1\over k+1} [G^{\pm(0)}(z), K^{(1)}(w)]=\pm {2(k+2)\over k+1} G^{\pm(0)}(w) \delta^{\rm A}\left( {w\over z}\right),
\end{align*}
which are the same as (\ref{G0K}),\footnote{In \eqref{G0K} the argument of $G^\pm$ is the same for both side, while they are
different here. This is the reason why we have $\delta$ in \eqref{G0K}, but $\delta^{\rm A}$ here. 
See Lemma \ref{delta-flip}} 
obtaining no new relations.
\end{prp}

\begin{prp}
Taking the terms of the order of $\hbar^2$ in the $\hbar$ expansions of the relations (\ref{rr-10}), we have
\begin{align}\label{KT-2}
&-{1\over k+1}[G^{\pm(1)}(z), K^{(1)}(w)]+ [G^{\pm(0)}(z), T^{(2)}(w)] \CR
&=\pm {2(k+2)\over k+1} 
\Bigl(G^{\pm(1)}(w)+ K^{-(1)}(w)G^{\pm(0)}(w)+G^{\pm(0)}(w)K^{+(1)}(w)\Bigr) \delta^{\rm A}\left( {w\over z}\right) \CR
&\qquad+2(k+2) (D G^{\pm(0)}) (w)  \delta^{\rm A}\left( {w\over z}\right)+6(k+2) G^{\pm(0)}(w) (D\delta^{\rm A})\left( {w\over z}\right).
\end{align}
\end{prp}

\begin{prp}\label{P3}
Combining (\ref{KG-2}) and (\ref{KT-2}), we have 
\begin{align}
& [G^{\pm(0)}(z), {1\over 4(k+2)}T^{(2)}(w)+{1\over 4(k+1)(k+2)} K^{(2)}(w)] \CR
&={1\over 2} (D G^{\pm(0)}) (w)  \delta^{\rm A}\left( {w\over z}\right)+{3\over 2} G^{\pm(0)}(w) (D\delta^{\rm A})\left( {w\over z}\right),
\end{align}
namely 
\begin{align*}
 [\G^{\pm}(z),\Lv(w)]
&={1\over 2} (D \G^{\pm}) (w)  \delta^{\rm A}\left( {w\over z}\right)+{3\over 2} \G^{\pm}(w) (D\delta^{\rm A})\left( {w\over z}\right) \CR
&={1\over 2} \G^{\pm}(z) (D\delta) \left( {w\over z}\right) + \G^{\pm}(w) (D\delta^{\rm A})\left( {w\over z}\right),
\end{align*}
where we used the lemma below for the last equality. Since  $(D\delta^{\rm A})\left(z\right)= -  (D\delta^{\rm A})\left(z^{-1}\right)$,
this is the desired relation with $w \leftrightarrow z$. 
\end{prp}

\begin{lem}\label{DdeltaA}
We have
\begin{align}
  (D\delta^{\rm A})\left( {w\over z}\right)g(w)= (D\delta) \left( {w\over z}\right)g(z)-\delta^{\rm A}\left( {w\over z}\right)(Dg)(w),
\end{align}
where the modes of $g(z)$ are supposed to be indexed by $\mathbb{Z}^\mathrm{A}$.
\end{lem}
\begin{proof}
\begin{align*}
  (D\delta^{\rm A})\left( {w\over z}\right)g(w) &=
 \sum_{m \in \mathbb{Z}^\mathrm{A}}  \sum_{n \in \mathbb{Z}^\mathrm{A}} m \left(\frac{w}{z}\right)^m g_n \\
 &=  \sum_{m \in \mathbb{Z}^\mathrm{A}}  \sum_{n \in \mathbb{Z}^\mathrm{A}} \left[ (m-n) \left(\frac{w}{z}\right)^{m-n} z^{-n} g_n
 + n \left(\frac{w}{z}\right)^m w^{-n} g_n \right] \\
 &= (D\delta) \left( {w\over z}\right)g(z)-\delta^{\rm A}\left( {w\over z}\right)(Dg)(w).
\end{align*}
\end{proof}

\subsection{$\hbar$ expansions of the relations for $K^{\pm}$ vs. $T$ }\label{sub45}
Combining the relations (\ref{rr-4}) and (\ref{rr-5}),
we have
\begin{align}
{(1-q^{k+1}w/z)(1-q^{-k-1}w/z)\over (1-q^{k+3}w/z)(1-q^{-k-3}w/z)}T(z)K(w)=
{(1-q^{k+1}z/w)(1-q^{-k-1}z/w)\over (1-q^{k+3}z/w)(1-q^{-k-3}z/w)}K(w)T(z).\label{TK=KT}
\end{align}
We study the $\hbar$ expansion
of (\ref{TK=KT}),
considering the terms of the orders up to $\hbar^4$.

\begin{lem}
We have
\begin{align*}
&{(1-q^{k+1}w/z)(1-q^{-k-1}w/z)\over (1-q^{k+3}w/z)(1-q^{-k-3}w/z)}=
1+\hbar^2 \,\,4(k+2)(Ds)\left( {w\over z}\right)\\
&+\hbar^4 \left( {2(k+2)(k+3)^2\over 3} (D^3s)\left( {w\over z}\right)
-{4(k+2)^2\over 3} (Ds)\left( {w\over z}\right)\right)+{\mathcal O}(\hbar^6).
\end{align*}
\end{lem}

\begin{prp}\label{P5}
We have no nontrivial relations by 
taking the terms of the order of $\hbar^0$ and $\hbar^1$ in the $\hbar$ expansion of the relation (\ref{TK=KT}).
Taking the coefficients of the order of $\hbar^2$ in that, we have
\begin{align*}
&-{1\over k+1} K^{(1)}(z)K^{(1)}(w)+{4k(k+2)\over k+1}(Ds)\left({w\over z} \right)\\
&=-{1\over k+1} K^{(1)}(w)K^{(1)}(z)+{4k(k+2)\over k+1}(Ds)\left({z\over w} \right),
\end{align*}
namely
\begin{equation*}
 [K^{(1)}(z),K^{(1)}(w)]=4k(k+2)(D\delta)\left({w\over z} \right),
\end{equation*}
which implies \eqref{N=2comrel-5}.
\end{prp}

\begin{prp}\label{prpT2K1}
Taking the coefficients of the order of $\hbar^3$ in   the $\hbar$ expansion of the relation (\ref{TK=KT}), we have
\begin{align*}
&T^{(2)}(z)K^{(1)}(w)-{1\over k+1} K^{(1)}(z)K^{(2)}(w)\\
&\qquad +4(k+2)(Ds)\left({w\over z} \right)\Bigl( -{1\over k+1} K^{(1)}(z)+{k\over k+1} K^{(1)}(w)\Bigr) \\
=&~K^{(1)}(w)T^{(2)}(z)-{1\over k+1} K^{(2)}(w)K^{(1)}(z)\\
&\qquad+4(k+2)(Ds)\left({z\over w} \right)\Bigl( -{1\over k+1} K^{(1)}(z)+{1\over k+1} K^{(1)}(w)\Bigr),
\end{align*}
namely
\begin{align}\label{T2K1}
&-[T^{(2)}(z),K^{(1)}(w)]+{1\over k+1} [K^{(1)}(z),K^{(2)}(w)] \CR
&=4(k+2)(D\delta)\left({w\over z} \right)\Bigl( -{1\over k+1}K^{(1)}(z)+ {k\over k+1} K^{(1)}(w)\Bigr).
\end{align}
\end{prp}

\begin{dfn}
Define the normal ordered product $\no K^{(1)}(z)K^{(1)}(w)\no$ by
\begin{align*}
&\no K^{(1)}(z)K^{(1)}(w)\no= K^{-(1)}(z)K^{(1)}(w)+K^{(1)}(w)K^{+(1)}(z).
\end{align*}
\end{dfn}

\begin{lem}\label{noK}
We have 
\begin{align*}
 K^{(1)}(z)K^{(1)}(w)= \no K^{(1)}(z)K^{(1)}(w)\no+4k(k+2) (Ds)\left({w\over z} \right),
\end{align*}
and
\begin{align*}
\left((Ds)\left({w\over z} \right)\right)^2={1\over 6} (D^3s)\left({z\over w} \right)-{1\over 6} (Ds)\left({z\over w} \right).
\end{align*}
\end{lem}

\begin{prp}
Taking the coefficients of the order of $\hbar^4$ in   the $\hbar$ expansion of the relation (\ref{TK=KT}), we have
\begin{align*}
&T^{(3)}(z)K^{(1)}(w)+T^{(2)}(z)K^{(2)}(w)-{1\over k+1} K^{(1)}(z)K^{(3)}(w)\\
&\qquad +4(k+2)(Ds)\left({w\over z} \right)\Bigl(T^{(2)}(z) -{1\over k+1} K^{(1)}(z)K^{(1)}(w)+{k\over k+1} K^{(2)}(w)\Bigr) \\
&\qquad + {2k(k+2)(k+3)^2\over 3(k+1)} (D^3s)\left({w\over z} \right)-
 {4k(k+2)^2\over 3(k+1)} (Ds)\left({w\over z} \right),\\
=&~K^{(1)}(w)T^{(3)}(z)+K^{(2)}(w)T^{(2)}(z)-{1\over k+1} K^{(3)}(w)K^{(1)}(z)\\
&\qquad +4(k+2)(Ds)\left({z\over w} \right)\Bigl(T^{(2)}(z) -{1\over k+1} K^{(1)}(w)K^{(1)}(z)+{k\over k+1} K^{(2)}(w)\Bigr) \\
&\qquad + {2k(k+2)(k+3)^2\over 3(k+1)} (D^3s)\left({z\over w} \right)-
 {4k(k+2)^2\over 3(k+1)} (Ds)\left({z\over w} \right).
 \end{align*}
By using Lemma \ref{noK} above, from these we obtain 
 \begin{align*}
&-[T^{(3)}(z),K^{(1)}(w)]-[T^{(2)}(z),K^{(2)}(w)]+{1\over k+1} [K^{(1)}(z),K^{(3)}(w)]\\
&= +4(k+2)(D\delta)\left({w\over z} \right)\Bigl(T^{(2)}(z) -{1\over k+1}\no  K^{(1)}(z)K^{(1)}(w)\no+
{k\over k+1} K^{(2)}(w)\Bigr) \\
&\qquad + {2k(k+1)(k+2)\over 3} (D^3\delta)\left({w\over z} \right)+
 {4k(k+2)^2\over 3(k+1)} (D\delta)\left({w\over z} \right).
\end{align*}
\end{prp}

\begin{cor}\label{CorT3T1}
Antisymmetrizing this with respect to the interchange $z\leftrightarrow w$, we have 
\begin{align}\label{T3T1}
&-\Bigl( [T^{(3)}(z),K^{(1)}(w)]+[K^{(1)}(z),T^{(3)}(w)]
+[T^{(2)}(z),K^{(2)}(w)]+[K^{(2)}(z),T^{(2)}(w)]\Bigr)\nonumber\\
&+{1\over k+1}\Bigl( [K^{(1)}(z),K^{(3)}(w)]+[K^{(3)}(z),K^{(1)}(w)]\Bigr)  \CR
&= 4(k+2)(D\delta)\left({w\over z} \right)\Bigl(
T^{(2)}(z)+T^{(2)}(w)\Bigr) 
+{4k(k+2)\over k+1}(D\delta)\left({w\over z} \right)\Bigl(
K^{(2)}(z)+K^{(2)}(w)\Bigr) \CR
& -{8(k+2)\over k+1}(D\delta)\left({w\over z} \right)
\no  K^{(1)}(z)K^{(1)}(w)\no  \CR
&\qquad +  {4k(k+1)(k+2)\over 3} (D^3\delta )\left({w\over z} \right)+
{8k(k+2)^2\over 3(k+1)}  (D\delta)\left({w\over z} \right).
\end{align}
\end{cor}

As we explained we should perform the exact $\hbar$ expansions up to $\hbar^4$ of
the relations among $K^{\pm}(z)$ and $T(z)$. The coefficients of $\hbar^4$ involve
$K^{(3)}(z)$ and $T^{(3)}(z)$ as we have seen above and will see the following subsections.
Note that in the definition \ref{CFTcurrents} of the generating currents of the ordinary $\mathcal{N}=2$
SCA we have neither $K^{(3)}(z)$ nor $T^{(3)}(z)$. As we will show,
we can eliminate them from the final commutation relations \eqref{N=2comrel-6} and \eqref{N=2comrel-7}.

\subsection{$\hbar$ expansions of the relations for $K$ vs. $K$ }\label{sub46}
{}From (\ref{rr-2}) and (\ref{rr-3}), we have
\begin{align}
&{(1-q^{2k+2}w/z)(1-q^{-2k-2}w/z)\over (1-q^{2}w/z)(1-q^{-2}w/z)}K(z)K(w) \CR
&\quad ={(1-q^{2k+2}w/z)(1-q^{-2k-2}w/z)\over (1-q^{2}w/z)(1-q^{-2}w/z)}K(w)K(z).\label{KK=KK}
\end{align}
We study the $\hbar$ expansion
of (\ref{KK=KK}),
considering the terms of the orders up to $\hbar^4$.

\begin{lem}
We have
\begin{align*}
&{(1-q^{2k+2}w/z)(1-q^{-2k-2}w/z)\over (1-q^{2}w/z)(1-q^{-2}w/z)}=
1+\hbar^2 \left(-4k(k+2)\right)(Ds)\left( {w\over z}\right)\\
&+\hbar^4 \left( -{8k(k+2)\over 3} (D^3s)\left( {w\over z}\right)
-{4k^2(k+2)^2\over 3} (Ds)\left( {w\over z}\right)\right)+{\mathcal O}(\hbar^6).
\end{align*}
\end{lem}

\begin{prp}
We have no nontrivial relations by 
taking the terms of the order of $\hbar^0$ and $\hbar^1$ in the $\hbar$ expansion of the relation (\ref{KK=KK}).
Taking the coefficients of the order of $\hbar^2$ in that, we have
\begin{equation*}
K^{(1)}(z)K^{(1)}(w)-4k(k+2)(Ds)\left( {w\over z}\right)\\
=~K^{(1)}(w)K^{(1)}(z)-4k(k+2)(Ds)\left( {z\over w}\right),
\end{equation*}
namely
\begin{equation*}
 [K^{(1)}(z),K^{(1)}(w)]=4k(k+2)(D\delta)\left({w\over z} \right),
\end{equation*}
obtaining no new relation.
\end{prp}

\begin{prp}
Taking the terms of the order of $\hbar^3$  in the $\hbar$ expansion of the relation (\ref{KK=KK}), we have
\begin{align*}
&K^{(2)}(z)K^{(1)}(w)+K^{(1)}(z)K^{(2)}(w)-4k(k+2)(Ds)\left( {w\over z}\right)
\Bigl(K^{(1)}(z)+K^{(1)}(w) \Bigr)\\
=&~K^{(1)}(w)K^{(2)}(z)+K^{(2)}(w)K^{(1)}(z)-4k(k+2)(Ds)\left( {z\over w}\right)
\Bigl(K^{(1)}(z)+K^{(1)}(w) \Bigr),
\end{align*}
namely
\begin{align*}
& [K^{(2)}(z),K^{(1)}(w)]+[K^{(1)}(z),K^{(2)}(w)]=
4k(k+2)(D\delta)\left({w\over z} \right)\Bigl(K^{(1)}(z)+K^{(1)}(w) \Bigr).
\end{align*}
\end{prp}

\begin{prp}
Taking the terms of the order of $\hbar^4$  in the $\hbar$ expansion of the relation (\ref{KK=KK}), we have
\begin{align*}
&K^{(3)}(z)K^{(1)}(w)+K^{(2)}(z)K^{(2)}(w)+K^{(1)}(z)K^{(3)}(w)\\
&\qquad -4k(k+2)(Ds)\left( {w\over z}\right)
\Bigl(K^{(2)}(z)+K^{(1)}(z)K^{(1)}(w)+K^{(2)}(w) \Bigr)\\
&\qquad  -{8k(k+2)\over 3} (D^3s)\left( {w\over z}\right)
-{4k^2(k+2)^2\over 3} (Ds)\left( {w\over z}\right)\\
=&~K^{(1)}(w)K^{(3)}(z)+K^{(2)}(w)K^{(2)}(z)+K^{(3)}(w)K^{(1)}(z)\\
&\qquad -4k(k+2)(Ds)\left( {z\over w}\right)
\Bigl(K^{(2)}(z)+K^{(1)}(w)K^{(1)}(z)+K^{(2)}(w) \Bigr)\\
&\qquad  -{8k(k+2)\over 3} (D^3s)\left( {z\over w}\right)
-{4k^2(k+2)^2\over 3} (Ds)\left( {z\over w}\right),
\end{align*}
namely 
\begin{align}\label{K3K1}
&[K^{(3)}(z),K^{(1)}(w)]+[K^{(2)}(z),K^{(2)}(w)]+[K^{(1)}(z),K^{(3)}(w)] \CR
&=4k(k+2)(D\delta)\left( {w\over z}\right)
\Bigl(K^{(2)}(z)+\no K^{(1)}(z)K^{(1)}(w)\no +K^{(2)}(w) \Bigr) \CR
&\qquad + {8k(k+1)^2(k+2)\over 3} (D^3\delta )\left( {w\over z}\right)
-{4k^2(k+2)^2\over 3} (D\delta )\left( {w\over z}\right).
\end{align}
\end{prp}

\begin{prp}\label{prpT3K1}
Combining (\ref{K3K1}) with the relation (\ref{T3T1}) in Corollary \ref{CorT3T1}, we have
\begin{align}\label{T3K1}
&-\Bigl( [T^{(3)}(z),K^{(1)}(w)]+[K^{(1)}(z),T^{(3)}(w)]\Bigr) \CR
 &=
 \Bigl( [T^{(2)}(z),K^{(2)}(w)]+[K^{(2)}(z),T^{(2)}(w)]\Bigr)
+{1\over k+1}[K^{(2)}(z),K^{(2)}(w)] \CR
& +4(k+2)(D\delta)\left({w\over z} \right)\Bigl(
T^{(2)}(z)+T^{(2)}(w)\Bigr) 
 -{4(k+2)^2\over k+1}(D\delta)\left({w\over z} \right)
\no  K^{(1)}(z)K^{(1)}(w) \CR
&\qquad- {4k(k+1)(k+2)\over 3} (D^3\delta)\left({w\over z} \right)+
 {4k(k+2)^3\over 3(k+1)} (D\delta)\left({w\over z} \right).
\end{align}
\end{prp}

\begin{proof}
We have
\begin{align*}
&-\Bigl( [T^{(3)}(z),K^{(1)}(w)]+[K^{(1)}(z),T^{(3)}(w)]\Bigr)
= \Bigl( [T^{(2)}(z),K^{(2)}(w)]+[K^{(2)}(z),T^{(2)}(w)]\Bigr) \\
& \qquad -{1\over k+1}\Bigl( [K^{(1)}(z),K^{(3)}(w)]+[K^{(3)}(z),K^{(1)}(w)]\Bigr) 
 +4(k+2)(D\delta)\left({w\over z} \right)\Bigl( T^{(2)}(z)+T^{(2)}(w)\Bigr) \\
& \qquad +{4k(k+2)\over k+1}(D\delta)\left({w\over z} \right)\Bigl(
K^{(2)}(z)+K^{(2)}(w)\Bigr)
 -{8(k+2)\over k+1}(D\delta)\left({w\over z} \right)
\no  K^{(1)}(z)K^{(1)}(w)\no \\
&\qquad + {4k(k+1)(k+2)\over 3} (D^3\delta)\left({w\over z} \right)+
 {8k(k+2)^2\over 3(k+1)} (D\delta)\left({w\over z} \right)\\
 &=
 \Bigl( [T^{(2)}(z),K^{(2)}(w)]+[K^{(2)}(z),T^{(2)}(w)]\Bigr)
+{1\over k+1}[K^{(2)}(z),K^{(2)}(w)]\\
& +4(k+2)(D\delta)\left({w\over z} \right)\Bigl(
T^{(2)}(z)+T^{(2)}(w)\Bigr) 
 -{4(k+2)^2\over k+1}(D\delta)\left({w\over z} \right)
\no  K^{(1)}(z)K^{(1)}(w)\no \\
&\qquad- {4k(k+1)(k+2)\over 3} (D^3\delta)\left({w\over z} \right)+
 {4k(k+2)^3\over 3(k+1)} (D\delta)\left({w\over z} \right).
\end{align*}

\end{proof}

\subsection{$\hbar$ expansions of the relations for $T$ vs. $T$ }\label{sub47}
Finally, we study the $\hbar$ expansion
of the relation (\ref{rr-11}),
considering the terms of the orders up to $\hbar^4$.

\begin{lem}
We have
\begin{align}
&[T(z),T(w)]
=\hbar^2\,\,{1\over (k+1)^2} [K^{(1)}(z),K^{(1)}(w)]\nonumber \\
&-
\hbar^3\,\, {1\over k+1}\Bigl( [K^{(1)}(z),T^{(2)}(w)]+[T^{(2)}(z),K^{(1)}(w)]\Bigr) \nonumber\\
&-\hbar^4\,\,{1\over k+1}\Bigl( [K^{(1)}(z),T^{(3)}(w)]+[T^{(3)}(z),K^{(1)}(w)]\Bigr)
+\hbar^4[T^{(2)}(z),T^{(2)}(w)]+{\mathcal O}(\hbar^5).
\end{align}
\end{lem}

\begin{lem}
Both in the {\rm NS} sector and in the {\rm R} sector, we have
\begin{align*}
&W(z)={1\over [k+1]} \widetilde{K}^{\rm A}(z)-c^{\rm A}_0(k+2) +\widetilde{T}^{\rm A}(z)\\
&=
{k\over (k+1)^2(k+2)}+ \hbar\,\, 
\left(-{1\over (k+1)^2}\right)K^{(1)}(z)\\
&+\hbar^2\,\, 
\left({1\over k+1}T^{(2)}(z) +{1\over 2(k+1)^2}K^{(2)}(z) -
{k(k+3)\over 12(k+1)(k+2)} \right)+{\mathcal O}(\hbar^3).
\end{align*}
\end{lem}

\begin{prp}
Taking the coefficients of $\hbar^2$ in (\ref{rr-11}), we have
\begin{align*}
&{1\over (k+1)^2} [K^{(1)}(z),K^{(1)}(w)]={4k (k+2)\over (k+1)^2} (D\delta)\left( {w\over z}\right).
\end{align*}
\end{prp}

\begin{prp}
Taking the coefficients of $\hbar^3$ in (\ref{rr-11}), we have
\begin{align*}
&-{1\over k+1} \Bigl( [K^{(1)}(z),T^{(2)}(w)]+[T^{(2)}(z),K^{(1)}(w)]\Bigr)\\
&=
-{8(k+2)\over (k+1)^2} (D\delta)\left( {w\over z}\right)K^{(1)}(w)
-{4(k+2)\over (k+1)^2} \delta\left( {w\over z}\right)(DK^{(1)})(w),
\end{align*}
which, using Lemma \ref{Ddelta} below,  is simplified as 
\begin{align}
& [K^{(1)}(z),T^{(2)}(w)]+[T^{(2)}(z),K^{(1)}(w)]
=
{4(k+2)\over k+1}(D\delta)\left( {w\over z}\right)\Bigl( K^{(1)}(z)+K^{(1)}(w)\Bigr).\label{K1T2}
\end{align}
\end{prp}

\begin{lem}\label{Ddelta}
We have
\begin{align}
  (D\delta)\left( {w\over z}\right)f(w)= (D\delta)\left( {w\over z}\right)f(z)-\delta\left( {w\over z}\right)(Df)(w).
\end{align}
\end{lem}

\begin{prp}\label{P6}
Combining (\ref{K1T2}) with (\ref{T2K1}) in Proposition \ref{prpT2K1}, we have
\begin{align*}
&  [K^{(1)}(z),T^{(2)}(w)+{1\over k+1}K^{(1)}(w)]=
4(k+2)(D\delta)\left( {w\over z}\right)K^{(1)}(w).
\end{align*}
\end{prp}
Since the constant term is irrelevant to the commutation relation, this is nothing but \eqref{N=2comrel-6}.

\begin{prp}
Taking the coefficients of $\hbar^4$ in (\ref{rr-11}), we have
\begin{align}\label{T3r11}
&-{1\over k+1} \Bigl( [K^{(1)}(z),T^{(3)}(w)]+[T^{(3)}(z),K^{(1)}(w)]\Bigr)+[T^{(2)}(z),T^{(2)}(w)] \CR
&=
{4(k+2)^2\over k+1} (D\delta)\left( {w\over z}\right)\Bigl( T^{(2)}(z)+T^{(2)}(w)\Bigr)
+{4(k+2)\over k+1} (D\delta)\left( {w\over z}\right)\Bigl( K^{(2)}(z)+K^{(2)}(w)\Bigr) \CR
&
-{2(k+2)^2\over (k+1)^2} (D\delta)\left( {w\over z}\right)\Bigl( K^{-(1)}(z)K^{(1)}(z)+K^{(1)}(z)K^{+(1)}(z) \CR
&\qquad\qquad\qquad\qquad +K^{-(1)}(w)K^{(1)}(w)+K^{(1)}(w)K^{+(1)}(w)\Bigr) \CR
&+{8k(k+2)\over 3} (D^3\delta)\left( {w\over z}\right)+
{4k(k+2)^2\over 3(k+1)^2} (D\delta)\left( {w\over z}\right).
\end{align}
\end{prp}

\begin{prp}\label{P7}
Combining (\ref{T3r11}) with the relation (\ref{T3K1}) in Proposition \ref{prpT3K1}, we have
\begin{align}
&[T^{(2)}(z),T^{(2)}(w)]+{1\over k+1}  \Bigl( [T^{(2)}(z),K^{(2)}(w)]+[K^{(2)}(z),T^{(2)}(w)]\Bigr) \CR
&\qquad+
{1\over (k+1)^2}[K^{(2)}(z),K^{(2)}(w)] \CR
&=
4(k+2)(D\delta)\left( {w\over z}\right)\Bigl( T^{(2)}(z)+T^{(2)}(w)\Bigr)
+{4(k+2)\over k+1} (D\delta)\left( {w\over z}\right)\Bigl( K^{(2)}(z)+K^{(2)}(w)\Bigr) \CR
&+4k (k+2)\Bigl((D^3\delta)\left( {w\over z}\right)- (D\delta)\left( {w\over z}\right)\Bigr)
+{4k(k+2)(2k+1)\over 3(k+1)} (D\delta)\left( {w\over z}\right),
\end{align}
which is recast as
\begin{align*}
&[\Lv(z),\Lv(w)]=
 (D\delta) \left({w\over z}\right)\Bigl( \Lv(z)+\Lv(w)\Bigr) 
+ {1\over 12}  {3k\over k+2} \Bigl((D^3\delta) \left({w\over z}\right)
- \delta\left({w\over z}\right) \Bigr).
\end{align*}

\end{prp}

\begin{proof}

\begin{align*}
&[T^{(2)}(z),T^{(2)}(w)]+{1\over k+1}  \Bigl( [T^{(2)}(z),K^{(2)}(w)]+[K^{(2)}(z),T^{(2)}(w)]\Bigr) \CR
& \qquad + {1\over (k+1)^2}[K^{(2)}(z),K^{(2)}(w)]\\
&=
4(k+2)(D\delta)\left( {w\over z}\right)\Bigl( T^{(2)}(z)+T^{(2)}(w)\Bigr)
+{4(k+2)\over k+1} (D\delta)\left( {w\over z}\right)\Bigl( K^{(2)}(z)+K^{(2)}(w)\Bigr)\\
&
-{2(k+2)^2\over (k+1)^2} (D\delta)\left( {w\over z}\right)
\no \Bigl( K^{(1)}(z)-K^{(1)}(w)\Bigr)^2\no  \\
&+4k (k+2)\Bigl((D^3\delta)\left( {w\over z}\right)- (D\delta)\left( {w\over z}\right)\Bigr)
+{4k(k+2)(2k+1)\over 3(k+1)} (D\delta)\left( {w\over z}\right).
\end{align*}
We use the lemma below to obtain the desired result. 
\end{proof}


\begin{lem}
We have
\begin{align*}
&(D\delta)\left( {w\over z}\right)
\no \Bigl( K^{(1)}(z)-K^{(1)}(w)\Bigr)^2\no=0.
\end{align*}
\end{lem}

\section{Wakimoto representation of the quantum affine algebra $U_q(\widehat{\mathfrak{sl}}_2)$}

Before embarking on the construction of the Heisenberg representation of $\Svir$,
we need briefly recall the Wakimoto representation of the quantum affine algebra $U_q(\widehat{\mathfrak{sl}}_2)$ 
\cite{Matsuo:1992va}, \cite{Matsuo:1994nc} (see also \cite{Shiraishi:1992nqz}),
fixing our notations and recalling some operator product expansion (OPE) formulas.

\subsection{Heisenberg algebras and vertex operators $V^\pm(z)$, $Y^\pm(z)$ and  $W_\pm(z)$ }
\begin{dfn}\label{Matsuo-parafermion}
Introduce Heisenberg algebras 
generated by $\alpha_n,\overline{\alpha}_n,\beta_n$ $(n\in \mathbb{Z})$ and $Q_\alpha,Q_{\overline{\alpha}},Q_\beta$, 
with the commutation relations:
\begin{align*}
&[\alpha_n,\alpha_m]={[2n][kn]\over n}\delta_{n+m,0},\qquad\qquad \,\,[\alpha_n,Q_\alpha]=\delta_{n,0},\\
&[\overline{\alpha}_n,\overline{\alpha}_m]=-{[2n][kn]\over n}\delta_{n+m,0},\qquad\quad \,\,\, 
[\overline{\alpha}_n,Q_{\overline{\alpha}}]=-\delta_{n,0}\\
&[\beta_n,\beta_m]={[2n][(k+2)n]\over n}\delta_{n+m,0},\qquad [\beta_n,Q_{\beta}]=\delta_{n,0},
\end{align*}
and all the other commutators being vanishing.
\end{dfn}
We regard
the non negative Fourier modes  $\alpha_n,\overline{\alpha}_n,\beta_n$ $(n\geq 0)$ 
being the annihilation operators, and 
the negative Fourier modes  $\alpha_n,\overline{\alpha}_n,\beta_n$ $(n< 0)$ and 
$Q_\alpha,Q_{\overline{\alpha}},Q_\beta$ being the creation operators. Accordingly 
we use the symbol $:\bullet:$ for the normal ordering for the Heisenberg generators. 
Namely we move all the creation operators to the left of annihilation ones given in the symbol $:\bullet:$.

\begin{dfn}\label{Wakimoto-vertex}
Let $V^\pm(z)$, $Y^\pm(z)$ and  $W_\pm(z)$ be the vertex operators as
\begin{align*}
&V^\pm(z)=
e^{\pm 2Q_\alpha}z^{\pm {1\over k}\alpha_0}
:\exp\left( \mp \sum_{m\neq 0} q^{\mp {k\over 2} |m|} {z^{-m}\over [km]} \alpha_m\right):,\\
&Y^\pm(z)=e^{\pm 2Q_{\overline{\alpha}}}z^{\pm {1\over k}\overline{\alpha}_0}
:\exp\left( \mp \sum_{m\neq 0} q^{\mp {k\over 2} |m|} {z^{-m}\over [km]} \overline{\alpha}_m\right):,\\
&Z_\pm(z)=
\exp\left( \mp(q-q^{-1}) \sum_{m=1}^\infty  z^{\mp m} {[m]\over [2m]} \overline{\alpha}_{\pm m}\right)
q^{\mp {1\over 2} \overline{\alpha}_0},\\
&W_\pm(z)=
\exp\left( \mp(q-q^{-1}) \sum_{m=1}^\infty  z^{\mp m} {[m]\over [2m]} \beta_{\pm m}\right)
q^{\mp {1\over 2} \beta_0}.
\end{align*}
\end{dfn}

\subsection{Wakimoto representation of $U_q(\widehat{\mathfrak{sl}}_2)$}

\begin{dfn}\label{parafermion}
Introduce the following shorthand notations
\begin{align*}
&e_\pm(z)=~:Y^+(z)Z_\pm(q^{\mp {k+2\over 2}}z) W_\pm (q^{\mp{k\over 2}}z):,\\
&f_\pm(z)=~:Y^-(z)Z_\pm(q^{\pm{k+2\over 2}}z) W_\pm (q^{\pm{k\over 2}}z)^{-1}:.
\end{align*}
\end{dfn}

In the undeformed case the $\mathcal{N}=2$ SCA and the affine Lie algebra $\widehat{\mathfrak{sl}}_2$
have a common sector, called $\mathbb{Z}_k$ parafermion \cite{Fateev:1985mm}, \cite{Fateev:1985ig}, \cite{Zamolodchikov:1986gh}.
In the Wakimoto representation of $U_q(\widehat{\mathfrak{sl}}_2)$, the $q$-deformed parafermion sector
is generated by $\overline{\alpha}_n,\beta_n$ with the zero modes $Q_{\overline{\alpha}},Q_\beta$.
The operators $e_\pm(z)$ and $f_\pm(z)$ are fundamental vertex operators from the deformed parafermion.
One of the earliest references for deformed parafermion is \cite{DF0}.
In \cite{Jimbo:2000ff} the deformed parafermion derived from the Wakimoto representation of $U_q(\widehat{\mathfrak{sl}}_2)$
was employed for a free field computation of the the Andrews-Baxter-Forrester models in regime II. See also \cite{Konno}.

\begin{dfn}\label{EandF}
Set
\begin{align}
&E(z)=E_+(z)-E_-(z),\qquad E_\epsilon(z)=+{1\over q-q^{-1}} V^+(z) e_\epsilon (z),\label{E}\\
&F(z)=F_+(z)-F_-(z),\qquad F_\epsilon (z)=-{1\over q-q^{-1}} V^-(z) f_\epsilon (z),\label{F}
\end{align}
and 
\begin{align}
&\psi_\pm (z)=~:V^+(q^{\pm k/2}z) V^-(q^{\mp k/2}z):.\label{psi}
\end{align}
\end{dfn}

\begin{thm}\label{Wakimoto}
The operators $E(z),F(z)$ and $\psi_\pm(z)$ satisfy the defining relations for the 
quantum affine algebra $U_q(\widehat{\mathfrak{sl}}_2)$
\begin{align}
&\psi_\pm(z)\psi_\pm(w)=\psi_\pm(w)\psi_\pm(z),\label{Uq-1}\\
&\psi_-(z)\psi_+(w)={g(q^{-k}z/w)\over g(q^{+k}z/w)}\psi_+(w)\psi_-(z),\label{Uq-2}\\
&\psi_-(z)E(w)=g(q^{-k/2}z/w)E(w)\psi_-(z),\quad \psi_-(z)F(w)=g(q^{+k/2}z/w)^{-1}F(w)\psi_-(z),\label{Uq-3}\\
&E(z)\psi_+(w)=g(q^{-k/2}z/w)\psi_+(w)E(z),\quad F(z)\psi_+(w)=g(q^{+k/2}z/w)^{-1}\psi_+(w)F(z),\label{Uq-4}\\
&(z-q^{+2} w)E(z)E(w)+(w-q^{+2}z)E(w)E(z)=0,\label{Uq-5}\\
&(z-q^{-2} w)F(z)F(w)+(w-q^{-2}z)F(w)F(z)=0,\label{Uq-6}\\
&[E(z),F(w)]={1\over q-q^{-1}}
\Biggl(\delta\left( q^k{ w\over z}\right) \psi_+(q^{k/2}w)- \delta\left(q^{-k} {w\over  z}\right) \psi_-(q^{-k/2}w)\Biggr), \label{Uq-7}
\end{align}
where $g(z)^{\pm 1}$ are  the invertible Taylor series
\begin{align*}
g(z)={q^{-2}-z\over 1-q^{-2}z}=(q^{-2}-z)\sum_{n\geq 0} q^{-2n}z^n,\quad
g(z)^{-1}={q^{2}-z\over 1-q^{2}z}=(q^{2}-z)\sum_{n\geq 0} q^{2n}z^n.
\end{align*}

\end{thm}

For the reader's convenience, we recall the proof of Theorem \ref{Wakimoto} in Appendix \ref{proof-Wakimoto}.



\section{Twist of the $U(1)$ boson from $q$ deformed $Y$-algebra}\label{Yalgebra}

\subsection{$Y$-algebra and its gluing}

In \cite{Gaiotto:2017euk} a vertex operator algebra (VOA) called $Y$-algebra was introduced. 
The algebra denoted by $Y_{L,M,N}[\Psi]$ is indexed by three non-negative integers and has 
a parameter $\Psi$.  Associated with the $Y$-algebra is a five-brane junction with $D3$ branes (see Figure \ref{5brane-web} and
Table \ref{D5-D3}). The integers $(L,M,N)$ represent the number of $D3$ branes stretched between $5$ branes. 
The figure \ref{5brane-web} describes the $2$-$3$ plane of the brane configuration of Table \ref{D5-D3}. 
In this brane configuration the two dimensional plane with a complex coordinate $z=x_0 + ix_1$ is common to $D3$, $NS5$ and $D5$ branes.
The $Y$-algebra is regarded as a VOA on this complex plane and hence we have a $Y$-algebra associated with each trivalent vertex. 
Introducing parameters $\epsilon_i,~i=1,2,3$ with $\epsilon_1 + \epsilon_2 + \epsilon_3=0$ the parameter $\Psi$ is expressed as 
\begin{equation}
\Psi= - \frac{\epsilon_2}{\epsilon_1}.
\end{equation}
we can denote the $Y$-algebra by $Y_{N_1, N_2, N_2}^{\epsilon_1, \epsilon_2, \epsilon_3} := Y_{N_1, N_2, N_3} [ - \frac{\epsilon_2}{\epsilon_1}]$. 
Then we have a symmetry under a simultaneous cyclic permutation of $N_i$ and $\epsilon_i$. 
In the five-brane web the slope represents the five brane charge $(p,q)$, on which the $S$ duality group $SL(2, \mathbb{Z})$ acts. 
The $3$ branes are invariant under  $SL(2, \mathbb{Z})$. 
The $SL(2, \mathbb{Z})$ action on the parameter $\Psi$  is
\begin{equation}
M \cdot \Psi = \frac{p_1 \Psi + p_2}{q_1 \Psi + q_2}, \qquad
M = \left(
\begin{array}{cc}
p_1 & p_2 \\
q_1 & q_2
\end{array}
\right) \in SL(2, \mathbb{Z}).
\end{equation}


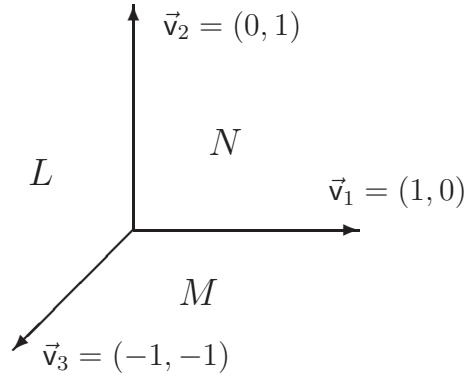
\begin{figure}[h]
\unitlength 2mm
\begin{center}
\begin{picture}(25,25)
\thicklines
\put(10,10){\vector(1,0){15}}
\put(10,10){\vector(0,1){15}}
\put(10,10){\vector(-1,-1){8}}
\put(15,15){\large$N$}
\put(13,5){\large$M$}
\put(3,13){\large$L$}
\put(23,12){$\vec{\mathsf{v}}_1=(1,0)$}
\put(12,23){$\vec{\mathsf{v}}_2=(0,1)$}
\put(4,1){$\vec{\mathsf{v}}_3 =(-1,-1)$}
\end{picture}
\end{center}
\caption{$5$-brane junction with $D3$ brane configuration (see also the table \ref{D5-D3}).
Our convention of the ordering is counterclockwise. The orientation of the edges is outgoing.}
\label{5brane-web}
\end{figure}

\begin{table}[h]
\begin{tabular}{|c||c|c|c|c|c|c|c|c|c|c|}\hline
Coordinates  &  0 & 1 &2 &  3 & 4 &5&  6 & 7 &8 & 9\\ \hline
D3 & $\circ$  & $\circ$ & $\circ$ &  $\circ$   & $-$ & $-$ & $-$ & $-$ & $-$ & $-$ \\ \hline 
NS 5 &  $\circ$ & $\circ$ & $-$ &  $\circ$   & $\circ$ & $\circ$ & $\circ$ & $-$ & $-$ & $-$ \\ \hline
D5 & $\circ$  & $\circ$ & $\circ$ & $-$  & $\circ$ & $\circ$ & $\circ$ & $-$ & $-$ & $-$ \\ \hline
\end{tabular}
\bigskip
\caption{Configuration of $5$ brane junction with $D3$ branes}
\label{D5-D3}
\end{table}

The vacuum character of $Y_{L,M,N}[\Psi]$ coincides with the character of the MacMahon module of $W_{1+\infty}$ algebra, 
or the affine $\mathfrak{gl}_1$ Yangian with a \lq\lq pit\rq\rq\ at $(L+1, M+1, N+1)$. This mathematically means that it is isomorphic to 
the $W_{1+\infty}$ algebra quotient by the (monomial) ideal $\mathcal{I}_{L,M,N}$ coming from the pit (See also \cite{Creutzig:2020zaj}). 
In particular, when $L=0$, it reduces to a Fock module. Consequently the $Y$-algebra $Y_{0,M,N}$ is identified with
the $\mathcal{W}$ algebra associated with the Lie superalgebra $\mathfrak{sl}_{N \vert M}$.

The $Y$-algebra of our concern is $Y_{0,1,2}[\Psi]$. In \cite{Gaiotto:2017euk} by examining the character it was shown that
$Y_{0,1,2}[\Psi]$ is related to the $\mathbb{Z}_\kappa$ parefermion algebra $\mathrm{Pf}_{\kappa} := SU(2)_\kappa/ U(1)_{2\kappa}$;
\begin{equation}
Y_{0,1,2}[\Psi] = \mathrm{Pf}_{\Psi -2} \times U(1)_{\Psi^{-1}(\Psi-1)(\Psi-2)},
\end{equation}
where $U(1)_\ell$ denotes $U(1)$ current algebra with level $\ell$. In fact from the general argument $Y_{0,1,2}[\Psi]$ is
the $\mathcal{W}$ algebra associated with $\mathfrak{sl}_{2 \vert 1}$ and its relation to the parafermion algebra 
from the viewpoint of the quantum Hamiltonian reduction was discussed in \cite{DF1}, \cite{DF2}. 
See also \cite{Kojima:2019ewe} for a recent study on the deformed $\mathcal{W}$ superalgebra of $\mathfrak{sl}_{2\vert 1}$.
Furthermore, in \cite{Prochazka:2017qum}, \cite{Prochazka:2018tlo} it was proposed that by gluing the $Y$-algebras according a web of five brane junction, 
we can systematically construct the VOA of $\mathcal{W}$ algebra type. 
Figure \ref{ALEweb} is an example of the five brane web describing the ALE space of $A_1$ type
(the Eguchi-Hanson space). The parameter $\Psi$ is shifted by 1, because the direction of the $NS$-$5$ brane 
is $(0,1)$ at $v_1$, but it is $(1,1)$ at $v_2$. This means the parameters $(\epsilon_1', \epsilon_2', \epsilon_3')$ at $v_2$
are related to $(\epsilon_1, \epsilon_2, \epsilon_3)$ at $v_1$ by $\epsilon_1' = \epsilon_1, \epsilon_2'= \epsilon_1 + \epsilon_2$.
Hence $\Psi' = - \epsilon_2' / \epsilon_1' = -1 + \Psi$. 


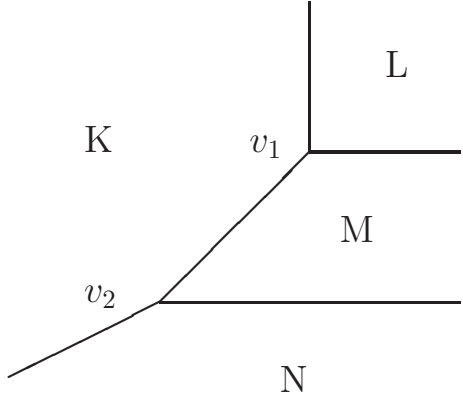
\begin{figure}[bht]
\unitlength 2mm
\begin{center}
\begin{picture}(25,30)
\thicklines
\put(10,10){\line(-2,-1){10}}
\put(10,10){\line(1,0){20}}
\put(10,10){\line(1,1){10}}
\put(20,20){\line(1,0){10}}
\put(20,20){\line(0,1){10}}
\put(5,20){\large{K}}
\put(25,25){\large{L}}
\put(22,14){\large{M}}
\put(18,4){\large{N}}
\put(16,20){\large$v_1$}
\put(5,10){\large$v_2$}
\end{picture}
\end{center}
\caption{Gluing of two $Y$-algebras $Y_{K,M,L}[\Psi]$ and $Y_{K,N,M}[\Psi-1]$
according to the toric diagram of ALE space of type $A_1$.}
\label{ALEweb}
\end{figure}

From the viewpoint of constructing VOA by gluing of $Y$-algebras, 
the current algebra $\widehat{\mathfrak{gl}}_2 = \widehat{\mathfrak{sl}}_2 \times \mathfrak{u}_1$
and the $\mathcal{N}=2$ superconformal algebra with additional $\mathfrak{u}_1$ factor are 
obtained from the brane configurations described in Figure \ref{conifold}; \cite{Gaiotto:2017euk}, \cite{Prochazka:2017qum}.
For example, in \cite{Harada:2018bkb} it is shown how the $\mathcal{N}=2$ unitary minimal models are 
realized by gluing plane partitions which are representation spaces of the $Y$-algebra. 
Since the brane configurations consist of two vertices, each algebra is constructed by gluing two $Y$-algebras.
Common to the both case is $Y_{0,1,2}$ which is the parafermion algebra and the
difference is the second $Y$-algebra $Y_{0,0,1}$ which is $\mathfrak{u}_1$ algebra.


\begin{figure}[bht]
\unitlength 2mm
\begin{center}
\begin{picture}(30,25)
\thicklines
\put(0,13){\line(-1,-1){5}}
\put(0,13){\line(1,0){8}}
\put(0,13){\line(0,1){8}}
\put(-5,8){\line(0,-1){8}}
\put(-5,8){\line(-1,0){8}}
\put(-6,13){\large 0}
\put(3,16){\large 2}
\put(0,6){\large 1}
\put(-10,3){\large 0}
\put(24,7){\line(-2,-1){10}}
\put(24,7){\line(1,0){15}}
\put(30,13){\line(-1,-1){6}}
\put(30,13){\line(1,0){8}}
\put(30,13){\line(0,1){8}}
\put(20,13){\large 0}
\put(33,16){\large 2}
\put(32,9){\large 1}
\put(30,3){\large 0}
\end{picture}
\end{center}
\caption{$\mathcal{N}=2$ superconformal algebra $\times U(1)$ (left) vs. $\widehat{\mathfrak{sl}}_2 \times U(1)$ (right)
by the gluing of two $Y$-algebras. }
\label{conifold}
\end{figure}
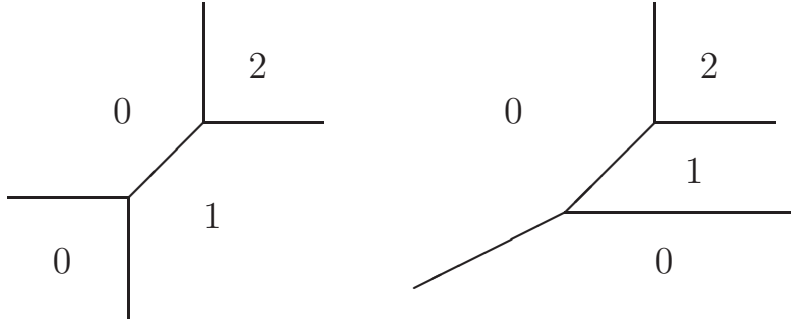

\subsection{Deformed $Y$-algebra and Ding-Iohara-Miki algebra}

What is relevant to us is a $q$ deformed version of the $Y$-algebra, 
which is discussed in \cite{Harada:2020woh}, \cite{Harada:2021xnm}.
Since the quantum toroidal algebra of type $\mathfrak{gl}_1$, or the Ding-Iohara-Miki (DIM)
algebra is a $q$ deformation of the $W_{1+\infty}$ algebra, or affine $\mathfrak{gl}_1$ Yangian algebra,
we can employ a Fock representation of the DIM algebra as a basic building block \cite{FHHSY}. 

Recall that the quantum toroidal algebra of type $\mathfrak{gl}_1$ has three parameters $(q_1,q_2,q_3)$ with $q_1q_2q_3=1$.
The algebra enjoys a $S_3$ symmetry under the permutation of $q_i$. 
The relation to the parameters of the $Y$-algebra is given by $q_i = e^{\epsilon_i}$. 
The Fock representation breaks the $S_3$ symmetry, we have to choose a preferred direction to reduce
the MacMahon module to a Fock module. Accordingly there are three types of the Fock module
whose central charge is $c=q_i^{\frac{1}{2}},~i=1,2,3$ \cite{BFM}. 
It is defined by the deformed Heisenberg algebra;
\beq\label{BFMdic}
\left[ a_n, a_m \right] = - \frac{n}{\kappa_n} (q_i^{\frac{n}{2}} - q_i^{-\frac{n}{2}})^3 \delta_{n+m, 0},
\eeq
with
\begin{align}\label{kappa}
\kappa_n &= \prod_{i=1}^{3} (q_i^{\frac{n}{2}}-q_i^{-\frac{n}{2}}) = \prod_{i=1}^{3} (q_i^{n}-1) 
= \prod_{i=1}^{3} (1 - q_i^{-n}) = \sum_{i=1}^{3} (q_i^{n}-q_i^{-n}).
\end{align}
We follow \cite{BFM} for the normalization of the deformed Heisenberg algebra. 
The multiple tensor product
\begin{equation}\label{multiple-Fock}
\overbrace{\mathcal{F}_1 \otimes \cdots \otimes \mathcal{F}_1}^L
~\otimes~
\overbrace{\mathcal{F}_2 \otimes \cdots \otimes \mathcal{F}_2}^M
~\otimes~ 
\overbrace{\mathcal{F}_3 \otimes \cdots \otimes \mathcal{F}_3}^N
\end{equation}
gives a free field representation of the deformed $Y$-algebra $q$-$Y_{L,M,N}$,
where $\mathcal{F}_i$ stands for the Fock module with the central charge $q_i^{\frac{1}{2}}$. 
When $L=M=0$, it reproduces the construction of the deformed $\mathcal{W}_N$ algebra from
the $N$-tuple tensor product of the Fock module of the quantum toroidal algebra.\footnote{Precisely
speaking we have to decouple an appropriate $U(1)$ factor.}

When we have the tensor product of the Fock modules of {\it different} central charge\footnote{When the central charges
are the same, there are a pair of the screening currents for each adjacent pair of Fock modules.},
$\mathcal{F}_{c_1}^{(1)} \otimes \mathcal{F}_{c_2}^{(2)}$,
we have a {\it unique} screening current of the DIM algebra \cite{BFM};
\begin{equation}\label{DIM-screening}
S(z) = ~e^{\frac{\epsilon_2}{\epsilon_1}Q^{(1)} - \frac{\epsilon_1}{\epsilon_2}Q^{(2)}}
z^{\frac{\epsilon_2}{\epsilon_1}a_0^{(1)} - \frac{\epsilon_1}{\epsilon_2}a_2^{(2)} + \frac{\epsilon_2}{\epsilon_3}}
\cdot \exp \left( -\sum_{n>0}\frac{1}{n}s_{-n}z^n \right) \exp \left( \sum_{n>0} \frac{1}{n}s_{n}z^{-n} \right),
\end{equation}
where
\begin{align}\label{screening}
s_n &= \frac{c_2^{n} - c_2^{-n}}{c_1^{n} - c_1^{-n}} a_n^{(1)} - c_2^{n} \frac{c_1^{n} - c_1^{-n}}{c_2^{n} - c_2^{-n}} a_n^{(2)}, \CR
s_{-n} &= c_1^{-n} \frac{c_2^{n} - c_2^{-n}}{c_1^{n} - c_1^{-n}} a_n^{(1)} - \frac{c_1^{n} - c_1^{-n}}{c_2^{n} - c_2^{-n}} a_n^{(2)},
\end{align}
and $a_n^{(1)} = a_n \otimes 1, a_n^{(2)} = 1 \otimes a_n, Q^{(1)} = Q \otimes 1, Q^{(2)} = 1 \otimes Q$.
The commutation relation of the zero modes is $[a_0, Q] = 1$, which implies that $S(z)$ is actually fermionic.

\subsection{ $\hbox{q-W}(\mathfrak{gl}_{2\vert 1})$ and deformed parafermion} 

Let us parametrize $q_i$ as follows;
\beq\label{DIMparameters}
q_1 = \gq^{-1}\gd =\gq^{-(k+2)} , \qquad q_2=\gq^2, \qquad q_3 = \gq^{-1} \gd^{-1} =\gq^k,
\eeq
where $k$ is going to be identified with the level of the quantum affine algebra $U_\gq(\widehat{\mathfrak{gl}}_2)$.
Note that $\gq$ here agrees the deformation parameter of $U_\gq(\widehat{\mathfrak{gl}}_2)$. 
In \cite{FJM} (see also \cite{Feigin:2013fga}) it was argued that the quantum affine algebra $U_\gq(\widehat{\mathfrak{gl}}_2)$
can be uplifted to the quantum toroidal algebra $U_{\gq,\gd}(\widehat{\widehat{\mathfrak{gl}}}_2)$ with $q_3=\gq^k$.
Let us consider the quantum toroidal algebra of type $A_{N-1}$ ; $U_{\gq,\gd}(\widehat{\widehat{\mathfrak{gl}}}_N)$.
Then we have
\begin{prp}[\cite{Feigin:2013fga}, \cite{FJM}]
In the quantum toroidal algebra $U_{\gq,\gd}(\widehat{\widehat{\mathfrak{gl}}}_N) $ with parameters \eqref{DIMparameters},
there are mutually commuting DIM algebras $\mathcal{E}'_{1,m}, m= 0, \ldots, N-1$; 
$$
U_{\gq,\gd}(\widehat{\widehat{\mathfrak{gl}}}_N) \supset 
\mathcal{E}'_{1,0} \otimes \mathcal{E}'_{1,1}
\otimes \cdots \otimes  \mathcal{E}'_{1,N-1},
$$
where the twisted parameters of the $m$-th DIM algebra $\mathcal{E}'_{1,m}$ are
\beq\label{q-parameters}
q_1^{(m)} = q_2, \qquad q_2^{(m)}= q_1^{m+1} q_3^{-N+m+1}, \qquad q_3^{(m)}=  q_1^{-m}q_3^{N-m}.
\eeq
\end{prp}
It is remarkable that the values of the twisted parameters \eqref{q-parameters} exactly match with what we obtain
from the brane web (or the toric diagram) of the ALE space of type $A_{N-1}$.\footnote{The diagram has
$N-1$ internal edges that correspond to a chain of $N-1$ rational curves whose intersection pairing agrees with
$(-1)$ times the $A_{N-1}$ Cartan matrix. The chain of $N-1$ rational curves comes from a resolution of the $A_{N-1}$ singularity.}
The toric diagram of the ALE space of type $A_{N-1}$ has $N$ vertices and the slopes of edges at the $k$-th vertex are
$\mathsf{v}_1^{(k)}=(1,0),\mathsf{v}_2^{(k)}=(k,1)$ and $\mathsf{v}_3^{(k)}=(-1-k,-1)$.\footnote{The ordering of commuting DIM
algebras and the vertices in the diagram is reversed.}

In the case of $N=2$ we have two commuting DIM algebras $\mathcal{E}'_{1,1}[q_2, q_1^2, q_1^{-1}q_3]$ and 
$\mathcal{E}'_{1,0}[q_2, q_1q_3^{-1}, q_3^2]$. The first toroidal algebra will produce the deformed $\mathcal{W}$ algebra
$\hbox{q-W}(\mathfrak{gl}_{2\vert 1})$, which is the parafermion sector and the second algebra is identified with
the deformed Heisenberg algebra of the $U(1)$ boson sector (See figure \ref{A-gluing}).
According to the prescription \eqref{multiple-Fock} we take the tensor product
$\mathcal{F}^{(1)}_{\gd^{-1}} \otimes \mathcal{F}^{(2)}_{q_1}
\otimes \mathcal{F}^{(3)}_{\gd^{-1}}$ of three Fock module of $\mathcal{E}'_{1,1}[q_2, q_1^2, q_1^{-1}q_3]$.
Note that $[q_2, q_1^2, q_1^{-1}q_3] = [\gq^{2}, q_1^2, \gd^{-2}]$.
In our normalization \eqref{BFMdic}, we thus employ the following deformed bosons;\footnote{The boson of the $k$-th factor 
of the tensor product is denoted by $a_n^{(k)}$.}
\beqa
&&
\left[ a_n^{(1)}, a_m^{(1)} \right] = \left[ a_n^{(3)}, a_m^{(3)} \right]
= -n \frac{(\gd^n - \gd^{-n})^2}{(\gq^n - \gq^{-n})(q_1^n - q_1^{-n})}\delta_{n+m,0}, \\
&&
\left[ a_n^{(2)}, a_m^{(2)} \right] 
= n \frac{(q_1^{n} - q_1^{-n})^2}{(\gq^n - \gq^{-n})(\gd^n - \gd^{-n})}\delta_{n+m,0}.
\eeqa


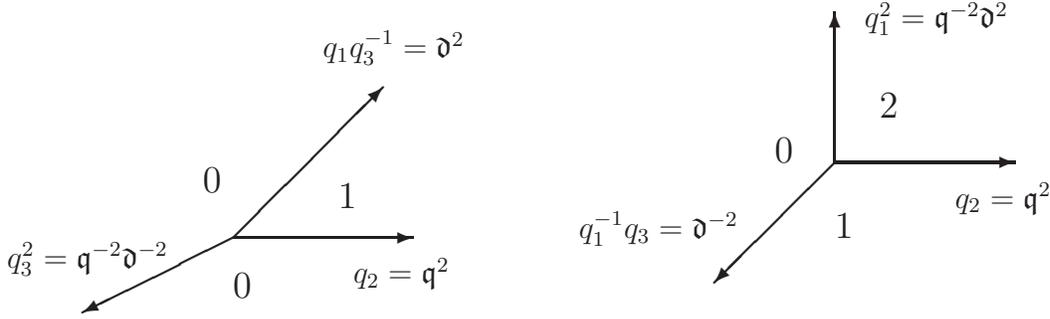
\begin{figure}[t]
\unitlength 2mm
\begin{center}
\begin{picture}(50,20)
\thicklines
\put(40,10){\vector(1,0){12}}
\put(40,10){\vector(0,1){10}}
\put(40,10){\vector(-1,-1){8}}
\put(43,13){\large$2$}
\put(40,5){\large$1$}
\put(36,10){\large$0$}
\put(48,7){$q_2=\gq^2$}
\put(42,19){$q_1^2=\gq^{-2}\gd^2$}
\put(23,5){$q_1^{-1}q_3=\gd^{-2}$}

\put(0,5){\vector(1,0){12}}
\put(0,5){\vector(1,1){10}}
\put(0,5){\vector(-2,-1){10}}
\put(7,7){\large$1$}
\put(0,1){\large$0$}
\put(-2,8){\large$0$}
\put(8,2){$q_2=\gq^2$}
\put(6,17){$q_1q_3^{-1}=\gd^2$}
\put(-15,3){$q_3^2= \gq^{-2}\gd^{-2}$}
\end{picture}
\end{center}
\caption{Gluing for $U_q(\widehat{\mathfrak{sl}}_2)$}
\label{A-gluing}
\end{figure}

By the formula \eqref{screening}, there are two fermionic screening currents in the form \eqref{DIM-screening},
with the following modes;
\begin{align}
s_n^{(12)} &= - \frac{q_1^{n} - q_1^{-n}}{\gd^{n} - \gd^{-n}} a_{n}^{(1)}
+ q_1^{n} \frac{\gd^{n} - \gd^{-n}}{q_1^{n} - q_1^{-n}} a_{n}^{(2)},
\CR
s_{-n}^{(12)} &= - \gd^{n} \frac{q_1^{n} - q_1^{-n}}{\gd^{n} - \gd^{-n}} a_{-n}^{(1)} 
+ \frac{\gd^{n} - \gd^{-n}}{q_1^{n} - q_1^{-n}} a_{-n}^{(2)},
\end{align}
and
\begin{align}
s_n^{(23)} &= - \frac{\gd^{n} - \gd^{-n}}{q_1^{n} - q_1^{-n}} a_{n}^{(2)} 
+ \gd^{-n} \frac{q_1^{n} - q_1^{-n}}{\gd^{n} - \gd^{-n}} a_{n}^{(3)},
\CR
s_{-n}^{(23)} &= - q_1^{-n} \frac{\gd^{n} - \gd^{-n}}{q_1^{n} - q_1^{-n}} a_{-n}^{(2)} 
+ \frac{q_1^{n} - q_1^{-n}}{\gd^{n} - \gd^{-n}} a_{-n}^{(3)}.
\end{align}
There are two possibilities of the choice of the root system of $\mathfrak{gl}_{2\vert 1}$.
One of them is purely fermionic and the Dynkin diagram has two fermionic nodes;
\begin{picture}(20,3)(0,0)
\thicklines
\put(0,0){$\otimes$}
\put(8,3){\line(1,0){12}}
\put(20,0){$\otimes$}
\end{picture}\hspace{4mm}. The fermionic screening currents $S^{(12)}(z)$ and $S^{(23)}(z)$ 
correspond to these fermionic nodes.

Up to overall normalization the $U(1)$ boson associated with the Cartan subalgebra of $\hbox{q-W}(\mathfrak{gl}_{2\vert 1})$ is 
fixed by the commutativity with these screening currents;
\begin{align}
h_n &= (q_1^n -q_1^{-n}) \left(\frac{a_n^{(1)}}{\gd^n - \gd^{-n}} - \gd^{n} \frac{a_n^{(2)}}{q_1^{n} - q_1^{-n}} 
+  \gq^{n} \frac{a_n^{(3)}}{\gd^{n} - \gd^{-n}} \right), \CR
h_{-n} &= (q_1^n -q_1^{-n}) \left(\gq^{-n} \frac{a_{-n}^{(1)}}{\gd^n - \gd^{-n}} - \gd^{-n} \frac{a_{-n}^{(2)}}{q_1^{n} - q_1^{-n}} 
+  \frac{a_{-n}^{(3)}}{\gd^{n} - \gd^{-n}} 
\right).
\end{align}
The commutation relation is 
\begin{align}
[h_n, h_m] &=
- n (q_1^{n} - q_1^{-n})^2 \left(\frac{\gq^{n} + \gq^{-n}}{(q_1^{n} - q_1^{-n})(\gq^{n} - \gq^{-n})}
- \frac{1}{(\gd^{n} - \gd^{-n})(\gq^{n} - \gq^{-n})} \right) \delta_{n+m,0} \CR
&= n \frac{(q_1^{n} - q_1^{-n})}{(\gq^n - \gq^{-n})(\gd^{n} - \gd^{-n})}
\left( (q_1^{n} - q_1^{-n})- (\gd^{n} - \gd^{-n})(\gq^{n} + \gq^{-n}) \right) \delta_{n+m,0}\CR
&= n \frac{(q_1^{n} - q_1^{-n})(q_3^{n} - q_3^{-n})}{(\gq^n - \gq^{-n})(\gd^{n} - \gd^{-n})} \delta_{n+m,0}.
\end{align}

The commutation relations for the screening currents are
\beq
\left[ s_{n}^{I}, s_{m}^J \right] = n~\delta_{n+m,0} \times
\begin{cases} 1 . \qquad (I=J), \\
- \frac{\gd^n - \gd^{-n}}{\gq^n - \gq^{-n}} , \qquad (I \neq J)
\end{cases}
\eeq
for $I, J = (12), (23)$. Hence, the orthogonal combinations are
\begin{align}
s^{(+)}_{\pm n} :=& \frac{[n]}{\sqrt{2}n}(s_{\pm n}^{(12)} + s_{\pm n}^{(23)}), \\
s^{(-)}_{\pm n} :=& \frac{[n]}{\sqrt{2}n}(s_{\pm n}^{(12)} - s_{\pm n}^{(23)}).
\end{align}
Their commutation relations are
\begin{align}
\left[ s^{(+)}_{n}, s^{(+)}_{m} \right] 
&= \frac{[n]}{n(\gq - \gq^{-1})} ((\gq^{n} - \gq^{-n}) - (\gd^{n} - \gd^{-n})) \delta_{n+m,0} \CR
&= - \frac{[n]}{n(\gq - \gq^{-1})}(q_1^{n/2} - q_1^{-n/2})(q_3^{n/2} +  q_3^{-n/2})  \delta_{n+m,0}, 
\\
\left[ s^{(-)}_{n}, s^{(-)}_{m} \right] 
&= \frac{[n]}{n(\gq - \gq^{-1})}((\gq^{n} - \gq^{-n}) + (\gd^{n} - \gd^{-n})) \delta_{n+m,0}\CR
&=  - \frac{[n]}{n(\gq - \gq^{-1})}(q_3^{n/2} - q_3^{-n/2})(q_1^{n/2} + q_1^{-n/2}) \delta_{n+m,0}.
\end{align}
Then the scaling 
\begin{align}
\beta_{n} &= \left(\frac{(q_1^{n/2} + q_1^{-n/2})(q_2^{n/2} + q_2^{-n/2})}{q_3^{n/2} + q_3^{-n/2}} \right)^{1/2}s^{(+)}_{n}, \\
\overline{\alpha}_{n} &= \left(\frac{(q_2^{n/2} + q_2^{-n/2})(q_3^{n/2} + q_3^{-n/2})}{q_1^{n/2} + q_1^{-n/2}} \right)^{1/2} s^{(-)}_{n}
\end{align}
reproduces the standard parafermion sector for the Wakimoto representation of $U_\gq (\widehat{\mathfrak{sl}}_2)$ with level $k$
(see Definition \ref{Matsuo-parafermion});
\begin{align}
\left[ \beta_{n}, \beta_{m}\right] &= \frac{[2n][(k+2)n]}{n}  \delta_{n+m,0},
\\
\left[ \overline{\alpha}_{n}, \overline{\alpha}_{m} \right] &=  - \frac{[2n][kn]}{n}  \delta_{n+m,0}.
\end{align}

\subsection{Adding $U(1)$ sector}

We can obtain a free field representation of $U_\gq(\widehat{\mathfrak{gl}}_2)$ 
by adding a $U(1)$ boson to the deformed parafermion sector.
From the diagram (Figure \ref{A-gluing}), we expect the appropriate deformed Heisenberg algebra to be added is 
the Fock module $\mathcal{F}_{q_3}$ of the second DIM algebra $\mathcal{E}'_{1,0}[q_2, q_1q_3^{-1}, q_3^2]$.
Hence we introduce the deformed Heisenberg algebra
\begin{equation}\label{tilde-a-com}
[ \widetilde{a}_{n}, \widetilde{a}_{m} ] = - n\frac{(q_3^n - q_3^{-n})^2}{(\gq^n - \gq^{-n})(\gd^{n} - \gd^{-n})} \delta_{n+m,0}.
\end{equation}
Then we can check the combination
\begin{equation}
\alpha_n = \frac{[n]}{n} (h_n  + \widetilde{a}_{n})
\end{equation}
reproduces the commutation relation for the $U(1)$ boson sector of the Wakimoto representation of $U_\gq(\widehat{\mathfrak{sl}}_2)$ 
in Definition \ref{Matsuo-parafermion} as follows;
\begin{align}
[ \alpha_{n}, \alpha_{m} ] 
&= \frac{[n]^2}{n} \frac{(q_3^n - q_3^{-n})}{(\gq^n - \gq^{-n})(\gd^n - \gd^{-n})} 
\left((q_1^{n} - q_1^{-n}) - (q_3^{n} - q_3^{-n}) \right) \delta_{n+m,0} \CR
&= \frac{(q_2^{n}-q_2^{-n})(q_3^n - q_3^{-n})}{n (\gq - \gq^{-1})^2} \delta_{n+m,0} =  \frac{[2n][kn]}{n} \delta_{n+m,0}.
\end{align}


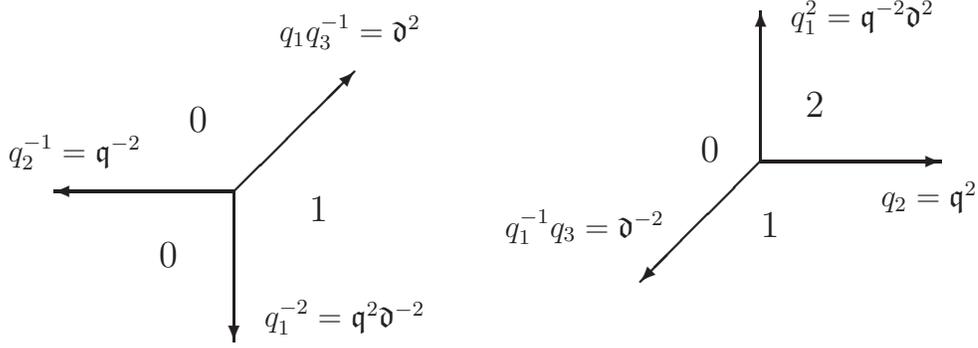
\begin{figure}[t]
\unitlength 2mm
\begin{center}
\begin{picture}(50,20)
\thicklines
\put(40,12){\vector(1,0){12}}
\put(40,12){\vector(0,1){10}}
\put(40,12){\vector(-1,-1){8}}
\put(43,15){\large$2$}
\put(40,7){\large$1$}
\put(36,12){\large$0$}
\put(48,9){$q_2=\gq^2$}
\put(42,21){$q_1^2=\gq^{-2}\gd^2$}
\put(23,7){$q_1^{-1}q_3=\gd^{-2}$}

\put(5,10){\vector(1,1){8}}
\put(5,10){\vector(0,-1){10}}
\put(5,10){\vector(-1,0){12}}
\put(10,8){\large$1$}
\put(0,5){\large$0$}
\put(2,14){\large$0$}
\put(8,20){$q_1q_3^{-1}=\gd^2$}
\put(-10,12){$q_2^{-1}=\gq^{-2}$}
\put(7,1){$q_1^{-2}=\gq^2\gd^{-2}$}
\end{picture}
\end{center}
\caption{Gluing for $\mathcal{N}=2$ superconformal algebra}
\label{B-gluing}
\end{figure}


On the other hand, from the diagram for the $\mathcal{N}=2$ Virasoro algebra  (See Figure \ref{B-gluing}),
we now employ the DIM algebra with different parameters as $U(1)$ sector to obtain $\mathcal{N}=2$ Virasoro $\times U(1)$.
We introduce
\beq\label{bar-a-com}
\left[ \overline{a}_{n}, \overline{a}_{m} \right] 
= n \frac{(\gq^n - \gq^{-n})^2}{(q_1^{n} - q_1^{-n})(\gd^{n} - \gd^{-n})} \delta_{n+m,0}.
\eeq
This is what we expect from the deformed boson from 
$\hbox{q-W}(\mathfrak{gl}_{1\vert 0})[q_1^{-2}, \gd^{2}, \gq^{-2}] \sim \mathcal{F}_{\gq}$.
This means that compared with the parameter for the $U(1)$ boson of $U_\gq(\widehat{\mathfrak{gl}}_2)$, we should make 
a replacement of parameter
\begin{equation}\label{T-transform}
(\gq, \gd) \longrightarrow  (\gq\gd^{-1}, \gd),
\end{equation}
which is nothing but the (inverse of) $T$ transformation in the base $(\gq, \gd)$.
In fact the left vertex in Figure \ref{B-gluing} is obtained by
applying the transformation \eqref{T-transform} to the corresponding vertex in Figure \ref{A-gluing}.
Similarly the commutation relation \eqref{bar-a-com} is obtained from \eqref{tilde-a-com}
by the same transformation.

Recall that the $U(1)$ boson vertex operator $V^\pm(z)$ in the Wakimoto
representation of $U_q(\widehat{\mathfrak{sl}}_2)$ involves the oscillator mode $\alpha_n$ as follows;
\begin{equation}
:\exp\left( \mp \sum_{n\neq 0} q^{\mp {k\over 2} |n|} {z^{-n}\over [kn]} \alpha_n \right):.
\end{equation}
See Definition \ref{Wakimoto-vertex}. To figure out the $U(1)$ boson for the deformed $\mathcal{N}=2$ SCA,
let us look at the combination
\begin{equation}
q_3^{\mp\frac{|n|}{2}} \frac{\gq - \gq^{-1}}{q_3^n - q_3^{-n}}\alpha_n
= q_3^{\mp\frac{|n|}{2}} \frac{\gq^n - \gq^{-n}}{n(q_3^n - q_3^{-n})}(h_n + \widetilde{a}_n).
\end{equation}
We keep the first term which comes from the parafermion sector but replace $\widetilde{a}_n$ with $\overline{a}_n$. 
We should also change the coefficient of the second term according to the transformation \eqref{T-transform};
\begin{equation}
q_3^{\mp\frac{|n|}{2}}\frac{\gq^n - \gq^{-n}}{n(q_3^n - q_3^{-n})} 
\longrightarrow - q_3^{\mp\frac{|n|}{2}}\gd^{\mp\frac{|n|}{2}} \frac{q_1^n - q_1^{-n}}{n(\gq^n - \gq^{-n})}.
\end{equation}
Hence, we see that the $U(1)$ boson for the deformed $\mathcal{N}=2$ SCA is given by\footnote{
The vertex operator $\widetilde{V}^\pm(z)$ does not involve the monomial factor $q_3^{\mp\frac{|n|}{2}}$.
See Definition \ref{tildeV}.}
\begin{equation}
\widetilde{\alpha}_n = q_3^{\mp\frac{|n|}{2}} \left( \frac{[n]}{n} h_{n} - \gd^{\mp\frac{|n|}{2}} 
\frac{[kn]}{n}\frac{q_1^{n} - q_1^{-n}}{\gq^n - \gq^{-n}} 
\cdot \overline{a}_{n} \right).
\end{equation}
The commutation relation is 
\begin{align}
\left[ \widetilde{\alpha}_n, \widetilde{\alpha}_m \right] 
&=  q_3^{\mp |n|} \frac{(q_1^n - q_1^{-n})(q_3^n - q_3^{-n})}{n(\gq - \gq^{-1})^2} \left[ \frac{\gq^n - \gq^{-n}}{\gd^n - \gd^{-n}}
+ \frac{\gd^{\mp |n|} (q_3^{n} - q_3^{-n})}{\gd^n - \gd^{-n}} \right] \delta_{n+m,0} \CR
&= \frac{[(k+2)n][kn]}{n} \delta_{n+m,0},
\end{align}
which exactly agrees with the commutation relation \eqref{tildealpha} we used in section \ref{Twist-Wakimoto}.

\subsection{Towards an uplift to quantum toroidal algebra} 

In the last section we see that $H_{1,n} := \alpha_n$ give the Cartan modes of $U_\gq(\widehat{\mathfrak{gl}}_2)$. 
The combination which is orthogonal to $H_{1, \pm n}$ is
\begin{align}
Z_{n} &:= \frac{[n]}{n} \left[\sqrt{q_1^n -q_1^{-n}} \left(\frac{a_n^{(1)}}{\gd^n - \gd^{-n}} - \gd^{n} 
\frac{a_n^{(2)}}{q_1^{n} - q_1^{-n}} 
+  \gq^{n} \frac{a_n^{(3)}}{\gd^{n} - \gd^{-n}} \right) + \frac{1}{\sqrt{q_3^{n} -q_3^{-n}}} \widetilde{a}_n \right], \CR
Z_{-n} &:= \frac{[n]}{n} \left [\sqrt{q_1^n -q_1^{-n}} \left(\gq^{-n} 
\frac{a_{-n}^{(1)}}{\gd^n - \gd^{-n}} - \gd^{-n} \frac{a_{-n}^{(2)}}{q_1^{n} - q_1^{-n}} 
+  \frac{a_{-n}^{(3)}}{\gd^{n} - \gd^{-n}} \right) + \frac{1}{\sqrt{q_3^{n} -q_3^{-n}}} \widetilde{a}_{-n} \right]. \CR
&&
\end{align}
We obtain $U_\gq(\widehat{\mathfrak{sl}}_2)$ after decoupling $Z_{\pm n}$.

Now let us introduce 
\begin{equation}
H_{0, \pm n} = - \frac{[n]}{n} (q_3^{\mp n} h_{\pm n} + q_1^{\mp n} \widetilde{a}_{\pm n} ).
\end{equation}
It is clear that the commutation relation of $H_{0,\pm n}$ is the same as $H_{1,\pm n}$.
We can see that they give the Cartan modes of the quantum toroidal algebra $U_{\gq,\gd}(\widehat{\widehat{\mathfrak{gl}}}_2)$.
\beq
\left[ H_{i,n}, H_{j,m} \right] 
=  a_{ij}(n) \frac{C^n - C^{-n}}{\gq - \gq^{-1}} \delta_{n+m,0},
\eeq
where the $(\gq, \gd)$ deformed Cartan matrix is
\beq\label{gl2Cartan}
a_{ij} (n) = \frac{[n]}{n} \times
\begin{cases}
\gq^n + \gq^{-n} \quad (i=j), \\
- \gd^n - \gd^{-n} \quad (i \neq j),
\end{cases}
\qquad
i,j \in \{ 0, 1 \}
\eeq
and we put $C=q_3=\gq^{k}$ for the center. In fact we can compute
\beqa
\left[ H_{0,n}, H_{1,-n} \right] 
&=& - q_3^{-n} \frac{(\gq^n - \gq^{-n})(q_1^n - q_1^{-n})(q_3^n - q_3^{-n})}{n(\gq - \gq^{-1})^2(\gd^n - \gd^{-n})}
+ q_1^{-n}\frac{(\gq^n - \gq^{-n})(q_3^n - q_3^{-n})^2}{n(\gq - \gq^{-1})^2(\gd^n - \gd^{-n})} \CR
&=&   \frac{(\gq^n - \gq^{-n})(q_3^n - q_3^{-n})}{n(\gq - \gq^{-1})^2(\gd^n - \gd^{-n})}
\left( - q_3^{-n} (q_1^n - q_1^{-n}) + q_1^{-n} (q_3^n - q_3^{-n}) \right) \CR
&=& -  \frac{(\gq^n - \gq^{-n})(q_3^n - q_3^{-n})(\gd^n + \gd^{-n})}{n(\gq - \gq^{-1})^2}.
\eeqa
Hence, the Wakimoto representation of the quantum affine algebra $U_\gq(\widehat{\mathfrak{gl}}_2)$ can be uplifted to 
the evaluation representation of the quantum toroidal algebra $U_{\gq,\gd}(\widehat{\widehat{\mathfrak{gl}}}_2)$.
See also \cite{Feigin:2018moi} for a Heisenberg representation of $U_{\gq,\gd}(\widehat{\widehat{\mathfrak{gl}}}_2)$.

\begin{figure}[t]
\centering
\includegraphics{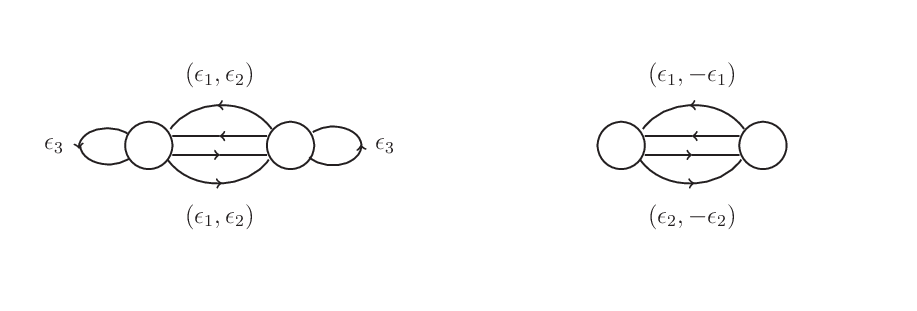}

\vspace{-16mm}
\caption{Quivers for the quantum toroidal algebra $U_{\gq,\gd}(\widehat{\widehat{\mathfrak{gl}}}_2)$ (left) and
 $U_{\gq,\gd}(\widehat{\widehat{\mathfrak{gl}}}_{1\vert 1})$ (right). The parameters associated with edges 
 satisfy the Calabi-Yau condition $\epsilon_1 + \epsilon_2 + \epsilon_3 =0$.}\label{Fig:quivers}
\end{figure}

It is an interesting problem to see that the twisted Wakimoto representation of 
the deformed $\mathcal{N}=2$ superconformal algebra allows a similar
uplift to a representation of some quantum toroidal algebra. 
The recent proposal of quiver quantum toroidal algebra 
\cite{Noshita:2021ldl}, \cite{Galakhov:2021vbo}, \cite{Noshita:2021dgj}, \cite{Bao:2022jhy} seems 
to provide a good starting point to tackle this task. 
The quivers corresponding to the local toric $CY_3$ geometries $\mathbb{C} \times (\mathbb{C}^2/\mathbb{Z}_2)$ 
and the resolved conifold are given in Figure \ref{Fig:quivers}.
(See also Fig. A.1 in \cite{Bao:2022jhy}; Fig.27 (a) and Fig.7 in \cite{Noshita:2021dgj}.)
The toric diagrams of these local $CY_3$ geometries agree with the web-diagrams for the $Y$-algebras we are looking at. 
In \cite{Galakhov:2021vbo} and \cite{Noshita:2021dgj} 
the structure functions of the quantum toroidal algebra are defined via data of the quiver diagram. 
According to eq. (4.2.10) in \cite{Noshita:2021ldl}, given a quiver, we can write down the commutation relation 
of the Cartan modes as\footnote{Compare this with (C.6) in \cite{Galakhov:2021vbo}.}
\beq
\left[ H_{i, r} , H_{j,s} \right] = \delta_{r+s, 0} \frac{C^r - C^{-r}}{r} \left( \sum_{I \in \{ j \to i\}} q_I^r 
- \sum_{I \in \{i \to j\}} q_I^{-r} \right). 
\eeq
Applying the dictionary;
\begin{align*}
\epsilon_1  &\longrightarrow  q_1=\gd \gq^{-1}, \\
\epsilon_2  &\longrightarrow  q_3= \gd^{-1} \gq^{-1}, \\
\epsilon_3  &\longrightarrow  q_2 = \gq^{2},
\end{align*}
to the left quiver in Figure \ref{Fig:quivers},
we have 
\beq
\left[ H_{i, r} , H_{i,s} \right] = \delta_{r+s, 0} \frac{C^r - C^{-r}}{r} \left( q_2^r - q_2^{-r} \right)
= \delta_{r+s, 0} \frac{C^r - C^{-r}}{r} \left(\gq^{2r} - \gq^{-2r} \right),
\eeq
and 
\beqa
\left[ H_{i, r} , H_{j,s} \right] &=& \delta_{r+s, 0} \frac{C^r - C^{-r}}{r}
 \left( q_1^{r} + q_3^{r} - q_1^{-r} - q_3^{-r} \right) \CR
 &=& - \delta_{r+s, 0} \frac{C^r - C^{-r}}{r}(\gq^r - \gq^{-r})(\gd^r + \gd^{-r}),
\qquad (i \neq j). 
\eeqa
which reproduces the deformed Cartan matrix \eqref{gl2Cartan} up to the normalization factor $(\gq^r - \gq^{-r})$. 
On the other hand, applying the same dictionary to the right quiver in Figure \ref{Fig:quivers}
\beq\label{diagonal=0}
\left[ H_{i, r} , H_{i,s} \right] = 0,
\eeq
and
\beqa\label{off-diagonal}
\left[ H_{i, r} , H_{j,s} \right] &=& \delta_{r+s, 0}~\epsilon_{ij}~\frac{C^r - C^{-r}}{r}
 \left( q_3^{r} + q_3^{-r} - q_1^{r} - q_1^{-r} \right) \CR
 &=& - \delta_{r+s, 0}~\epsilon_{ij}~\frac{C^r - C^{-r}}{r}(\gq^r - \gq^{-r})(\gd^r - \gd^{-r}), \qquad (i \neq j),
\eeqa
where $\epsilon_{01} = - \epsilon_{10} =1$.
In Section 4.4 of \cite{Noshita:2021dgj}, this is identified as the quantum toroidal algebra associated 
with the Lie superalgebra $\mathfrak{gl}_{1\vert 1}$. However, from the viewpoint of defining 
the quantum toroidal algebra based on the Cartan matrix \cite{Bezerra:2019dmp}, the case of $\mathfrak{gl}_{1\vert 1}$
looks quite irregular. For example, the commutation relations \eqref{diagonal=0} and \eqref{off-diagonal} would
imply the zero Cartan matrix in the limit $\gd \to 1$. 
It seems non-trivial to realize the commutation relations \eqref{diagonal=0} and \eqref{off-diagonal} in terms of deformed free bosons. 
Hence, an uplift of the $q$-deformed $\mathcal{N}=2$ SCA
to the quantum toroidal algebra $U_{\gq,\gd}(\widehat{\widehat{\mathfrak{gl}}}_{1\vert 1})$  is not straightforward. We leave this issue for future work.


\section{Twisted Wakimoto representation for realizing $\Svir$}\label{Twist-Wakimoto}

In the undeformed case the $\mathcal{N}=2$ superconformal algebra and the affine Lie algebra $\widehat{\mathfrak{sl}}_2$
have a common sector, called parafermion sector \cite{Fateev:1985mm}, \cite{Fateev:1985ig}. 
In the Wakimoto representation the parafermion sector
is generated by $\overline{\alpha}_n,\beta_n$ with the zero modes $Q_{\overline{\alpha}},Q_\beta$.
To construct the Wakimoto representation of $\Svir$, we will keep the ($q$-deformed) parafermion sector
and twist the $U(1)$ boson following the result in the previous section.
 
\begin{dfn}\label{modified-Cartan}
While keeping the generators $\overline{\alpha}_n,\beta_n$ and $Q_{\overline{\alpha}},Q_\beta$ as they are, 
replace the generators $\alpha_n$ and $Q_\alpha$ with the 
the modified ones $\widetilde{\alpha}_n$ and $Q_{\widetilde{\alpha}}$
satisfying the commutation relations
\begin{align}\label{tildealpha}
&[\widetilde{\alpha}_n,\widetilde{\alpha}_m]={[(k+2)n][kn]\over n}\delta_{n+m,0},\qquad\qquad 
\,\,[\widetilde{\alpha}_n,Q_{\widetilde{\alpha}}]=\delta_{n,0},
\end{align}
and all the other commutators being vanishing.
\end{dfn}

The commutation relation \eqref{tildealpha} of $\widetilde{\alpha}_n$ 
is the same as that of the Heisenberg generators $H_m$ of $\Svir$
introduced in section \ref{Heisenberg}. 
Hence, we have $H_m=\widetilde{\alpha}_m$ in the Wakimoto representation of $\Svir$.

\begin{dfn}\label{tildeV}
Introduce the modified vertex operators $\widetilde{V}^\pm(z)$ as\footnote{
There is a freedom of making $\widetilde{\alpha}_{\pm m} \to c_m \widetilde{\alpha}_{\pm m}$.
For example, we have $c_m=q^{\mp k|m|/2}$ in the case of $V^\pm (z)$, 
see Definition \ref{Wakimoto-vertex}. We may eliminate $c_m$ by a redefinition of 
the Heisenberg generators $H_m$, which affects the commutation relations among generators of $\Svir$. }
\begin{align*}
&\widetilde{V}^\pm (z)=e^{\pm (k+2)Q_{\widetilde{\alpha}}}z^{\pm {1\over k} \widetilde{\alpha}_0}
:\exp\left( \mp \sum_{m\neq 0} {z^{-m} \over [km]} \widetilde{\alpha}_m\right):.
\end{align*}

\end{dfn}

Recall that the basic vertex operators $e_{\pm}(x)$ and $f_{\pm}(z)$ 
of the parafermion sector are introduced in Definition \ref{parafermion}.

\begin{dfn}\label{mod-Wakimoto}
Let $\mathbf{K}^\pm(z),\mathbf{K}(z),\mathbf{T}(z)$ and  $\mathbf{G}^\pm(z)$ be the following combinations of the vertex operators
\begin{align}
&\mathbf{K}^\pm(z)=q^{\widetilde{\alpha}_0}
\exp\left((q-q^{-1})\sum_{\pm m>0} \widetilde{\alpha}_m z^{-m}\right),\label{Kpm-al}\\
&\mathbf{K}(z)=~\mathbf{K}^-(z)\mathbf{K}^+(z)=~:\widetilde{V}^+(q^{k}z)\widetilde{V}^-(q^{-k}z):,\\
&\mathbf{T}(z)= ~
:\widetilde{V}^+(q^{-1}z)\widetilde{V}^-(q^{+1}z): \mathbf{T}_{\mathrm{PF}}(z), \CR
&\mathbf{T}_{\mathrm{PF}}(z)= ~q^{-1} :e_+(q^{-{k+2\over 2}}z)f_-(q^{{k+2\over 2}}z):
-{[k+2]\over [k+1]} :e_+(q^{-{k+2\over 2}}z)f_+(q^{{k+2\over 2}}z):  \CR
&\qquad\qquad + q^{+1} :e_-(q^{-{k+2\over 2}}z)f_+(q^{{k+2\over 2}}z):, \label{T-parafermi} \\
&\mathbf{G}^\pm(z)= z^{1/2} \sum_{\epsilon=\pm}  \epsilon \,\mathbf{G}_\epsilon^\pm(z)
= z^{1/2} \left( \mathbf{G}_{+}^\pm(z) - \mathbf{G}_{-}^\pm(z) \right),
\end{align}
where
\begin{align}\label{G+-def}
&\mathbf{G}_\epsilon^+(z)= + {1\over q-q^{-1}}\widetilde{V}^+(q^{-k-2}z) e_\epsilon(q^{-{3k+4\over2}}z),\\
&\mathbf{G}_\epsilon^-(z)=-  {1\over q-q^{-1}}\widetilde{V}^-(q^{+k+2}z) f_\epsilon(q^{+{3k+4\over2}}z).
\label{G--def}
\end{align}
\end{dfn}
Note that there is an additional factor $z^{1/2}$ in the definition of $\mathbf{G}^\pm(z)$.
Namely, we suppose that in the NS sector $\mathbf{G}_\epsilon^{\pm}(z)$ are expanded in the integral powers of $z$
and in the R sector they are expanded in the half-integral powers of $z$. See Remark \ref{NSvsR}.
Thus it is the mode expansion of $\mathbf{G}_\epsilon^{\pm}(z)$, not $\mathbf{G}^\pm(z)$, that agrees 
with the standard convention of two dimensional superconformal field theory. 
This in particular means that we can employ the same momentum lattice of zero modes as
the NS and R sectors of the standard superconformal theory. 
The definition of $\mathbf{G}^\pm(z)$ may be compared with that of $E(z)$ and $F(z)$ for $U_q(\widehat{\mathfrak{sl}}_2)$
in Definition \ref{EandF}.
The expression of $\mathbf{T}_{\mathrm{PF}}(z)$ in the language of $U_q(\widehat{\mathfrak{sl}}_2)$ is
given in Appendix \ref{TGandTT}. See also the operator $L(z)$ in Prop. 4.2 of \cite{DF1}.


\subsection{(Dual) Fock representations of $\Svir$ in the {\rm NS} and {\rm R} sectors}
\label{sec:Fock-rep}
Recall that 
the zero modes of the twisted Wakimoto representation are $\widetilde{\alpha}_0, \overline{\alpha}_0, \beta_0$
and $Q_{\widetilde{\alpha}}, Q_{\overline{\alpha}}, Q_{\beta}$ with the commutation relations;
\begin{equation}
[\widetilde{\alpha}_0,Q_{\widetilde{\alpha}}]=1, \qquad
[\overline{\alpha}_0,Q_{\overline{\alpha}}]=-1, \qquad
[\beta_0,Q_{\beta}]=1.
\end{equation}
Set  
\begin{equation}
\vert \xi,\rho,\sigma\rangle =e^{\xi Q_{\widetilde{\alpha}} +\rho Q_{\overline{\alpha}}+\sigma Q_\beta}\vert 0\rangle
\end{equation}
with
\begin{equation*}
\widetilde{\alpha}_n \vert 0\rangle = \overline{\alpha}_n \vert 0\rangle 
= \beta_n \vert 0\rangle =0, \qquad n \geq 0.
\end{equation*}
The zero mode dependence of $\mathbf{G}_{\epsilon}^\pm(z)$ reads
\begin{equation*}
\mathbf{G}_{\epsilon}^\pm(z) \sim e^{\pm (k+2)Q_{\widetilde{\alpha}}} 
\left(q^{\mp(k+2)}z\right)^{\pm \frac{1}{k}\widetilde{\alpha}_0} \cdot 
e^{\pm 2 Q_{\overline{\alpha}}} \left(q^{\mp{k+2\over 2}}z\right)^{\pm \frac{1}{k}\overline{\alpha}_0}
q^{-\frac{\epsilon}{2}\overline{\alpha}_0} q^{\mp\frac{\epsilon}{2}\beta_0}.
\end{equation*}
Take the additional factor $z^{1/2}$, as in ${\bf G}^{\pm}(z)=z^{1/2}(\mathbf{G}_{+}^\pm(z)-\mathbf{G}_{-}^\pm(z))$, into account. 

\begin{lem}\label{G-zero}
We have
\begin{align*}
&{\bf G}^{\pm}(z) \vert \xi,\rho,\sigma\rangle
\propto z^{{1\over 2}\pm{\xi-\rho \over k}}~{\rm Osc.}(z)\vert\xi\pm(k+2),\rho\pm 2,\sigma\rangle,
\end{align*}
where ${\rm Osc.}(z)$ stands for some invertible element in the algebra of negative Fourier modes, {\it i.e.} 
some power series in $z$
with a non vanishing constant term.
\end{lem}

In the NS sector, ${\bf G}^{\pm}(z)$ should be expanded in the half-integral powers of $z$. 
The condition for  $\vert \xi,\rho,\sigma\rangle$ being a highest weight vector 
gives us ${\bf G}^{\pm}_m\vert \xi,\rho,\sigma\rangle=0$ for $m=1/2,3/2,\ldots$, 
which in view of Lemma \ref{G-zero} requires $\xi=\rho$. 

\begin{lem} When $\xi=\rho$, 
\begin{align*}
&{\bf K}^\pm(z)  |\xi,\rho,\sigma\rangle=q^{\rho} |\xi,\rho,\sigma\rangle+\cdots,\\
&{\bf T}(z) |\xi,\rho,\sigma\rangle=
q^{-{2\over k}\xi}
\left(
q^{-1}q^{{k+2\over k} \rho }q^{-\sigma}
 -{[k+2]\over [k+1]}q^{{k+2\over k} \rho } q^\rho 
  +q^{+1}q^{{k+2\over k} \rho }q^{\sigma} \right) |\xi,\rho,\sigma\rangle+\cdots\\
  &\qquad =
\left(
q^{-1}q^{\rho }q^{-\sigma}
 -{[k+2]\over [k+1]}q^{2\rho}
  +q^{+1}q^{\rho }q^{\sigma} \right) |\xi,\rho,\sigma\rangle+\cdots.
\end{align*}
\end{lem}
Compare these with the definition of parameters $u$ and $h$ for the Verma module in the NS sector (see \eqref{hwv}). 
From the parametrization \eqref{huv} of $h$, we can identity 
\begin{equation}
u=q^\rho,\qquad v=q^\sigma.
\end{equation}
Note that
\begin{align*}
|\rho+(k+2)n,\rho+2n,\sigma\rangle \mbox{  has the $x$-degree } n, \mbox{ and the $p$-degree } n^2/2.
\end{align*}

Set
$$
{\mathcal F}(\xi,\rho,\sigma) \seteq
\mathbb{C}[ \widetilde{\alpha}_{-1}, \widetilde{\alpha}_{-2},\ldots, \overline{\alpha}_{-1},\overline{\alpha}_{-2},\ldots,
\beta_{-1},\beta_{-2},\ldots]|\xi,\rho,\sigma\rangle,
$$
and
\begin{align}
 {\mathcal F}_{\rm NS}(u,v) \seteq
 \bigoplus_{n \in \mathbb{Z}} {\mathcal F}(\rho+(k+2)n,\rho+2n,\sigma). \label{NS-Fock}
\end{align}
\begin{prp}
Thanks to the Theorem \ref {Wakimoto-Svir} below, 
we have a representation of $\Svir$ in the {\rm NS} sector on the 
Fock space $ {\mathcal F}_{\rm NS}(u,v)$
with the highest weight vector $|\rho,\rho,\sigma\rangle$ satisfying
\begin{align*}
&{\bf K}^\pm_0|\rho,\rho,\sigma\rangle=u |\rho,\rho,\sigma\rangle,\qquad 
{\bf T}_0|\rho,\rho,\sigma\rangle=h(u,v) |\rho,\rho,\sigma\rangle,
\end{align*}
where $h(u,v)$ is defined in (\ref{huv}). 
We have the character for $ {\mathcal F}_{\rm NS}(u,v)$
\begin{align*}
{\rm ch}_{ {\mathcal F}_{\rm NS}(u,v)}(p,x)=\prod_{i=0}^\infty {(1+p^{i+1/2}x)(1+p^{i+1/2}x^{-1})\over (1-p^{i+1}) (1-p^{i+1})},
\qquad  |p|, |x| <1.
\end{align*}

\end{prp}

\begin{rmk}
Observe that the character ${\rm ch}_{ {\mathcal F}_{\rm NS}(u,v)}(p,x)$ and the conjectural  Verma module character
${\rm ch}_{\rm NS}(p,x)$ in (\ref{ch-NS}) are identical.
Hence this Fock representation is expected to be irreducible for generic $u$ and $v$. 
\end{rmk}

In the R sector, ${\bf G}^{\pm}(z)$ should be expanded in the integral powers of $z$. 
The condition for  $\vert \xi,\rho,\sigma\rangle$ being a highest weight vector 
gives us ${\bf G}^{+}_m\vert \xi,\rho,\sigma\rangle=0$ for $m=0,1,\ldots$, and 
${\bf G}^{-}_m\vert \xi,\rho,\sigma\rangle=0$ for $m=1,2,\ldots$, 
which in view of Lemma \ref{G-zero} requires $\xi = \rho + k/2$. 

\begin{lem}
When $\xi = \rho + k/2$, 
\begin{align*}
&{\bf K}^\pm(z)  |\xi,\rho,\sigma\rangle=q^{\xi} |\xi,\rho,\sigma\rangle+\cdots,\\
&{\bf T}(z) |\xi,\rho,\sigma\rangle=
q^{-{2\over k}\xi}
\left(
q^{-1}q^{{k+2\over k} \rho }q^{-\sigma}
 -{[k+2]\over [k+1]}q^{{k+2\over k} \rho } q^\rho 
  +q^{+1}q^{{k+2\over k} \rho }q^{\sigma} \right) |\xi,\rho,\sigma\rangle+\cdots\\
&\qquad =
q^{-1}
\left(
q^{-1}q^{\rho}q^{-\sigma}
 -{[k+2]\over [k+1]}q^{2\rho}
  +q^{+1}q^{\rho}q^{\sigma} \right) |\xi,\rho,\sigma\rangle+\cdots.
\end{align*}
\end{lem}
Recall we have defined the highest weight conditions in the $R$ sector with some $q$-shifts 
attached to $u$ and $h$ (see \eqref{hwvR}). 
We see these $q$-shifts are consistently derived 
with the identifications 
\begin{equation}
q^{k/2}u=q^\xi,\qquad u=q^\rho,\qquad 
 v=q^\sigma.
\end{equation}

We have
\begin{align*}
|k/2+\rho+(k+2)n,\rho+2n,\sigma\rangle \mbox{  has the $x$-degree } n, \mbox{ and the $p$-degree } n(n+1)/2.
\end{align*}

Set
\begin{align}
 {\mathcal F}_{\rm R}(u,v) \seteq
 \bigoplus_{n \in \mathbb{Z}} {\mathcal F}(k/2+\rho+(k+2)n,\rho+2n,\sigma). 
\end{align}
\begin{prp}
Thanks to the Theorem \ref {Wakimoto-Svir} below, 
we have a representation of $\Svir$ in the {\rm R} sector on the 
Fock space $ {\mathcal F}_{\rm R}(u,v)$
with the highest weight vector $|k/2+\rho,\rho,\sigma\rangle$ satisfying
\begin{align*}
&{\bf K}^\pm_0|k/2+\rho,\rho,\sigma\rangle=q^{k/2}u |k/2+\rho,\rho,\sigma\rangle,\qquad 
{\bf T}_0|k/2+\rho,\rho,\sigma\rangle=q^{-1}h(u,v) |k/2+\rho,\rho,\sigma\rangle,
\end{align*}
where $h(u,v)$ is defined in (\ref{huv}). 
We have the character for $ {\mathcal F}_{\rm R}(u,v)$
\begin{align*}
{\rm ch}_{ {\mathcal F}_{\rm R}(u,v)}(p,x)=\prod_{i=0}^\infty {(1+p^{i+1}x)(1+p^{i}x^{-1})\over (1-p^{i+1}) (1-p^{i+1})},
\qquad |p|, |x| <1.
\end{align*}
\end{prp}

\begin{rmk}
Observe that the character ${\rm ch}_{ {\mathcal F}_{\rm R}(u,v)}(p,x)$ and the conjectural  Verma module character
${\rm ch}_{\rm R}(p,x)$ in (\ref{ch-R}) are identical.
Hence we expect that this Fock representation is irreducible for generic $u$ and $v$. 
\end{rmk}

Next, we turn to the dual Fock space generated by 
\begin{equation}
\langle  \xi,\rho,\sigma\vert =
\langle 0 \vert  e^{-\xi Q_{\widetilde{\alpha}} -\rho Q_{\overline{\alpha}}-\sigma Q_\beta},
\end{equation}
with
\begin{equation*}
\langle 0 \vert \widetilde{\alpha}_n  = \langle 0 \vert  \overline{\alpha}_n 
=\langle 0 \vert  \beta_n  =0, \qquad n \leq 0.
\end{equation*}


\begin{lem}\label{G-zero-dual}
We have
\begin{align*}
&\langle  \xi,\rho,\sigma\vert  {\bf G}^{\pm}(z)
\propto z^{{1\over 2}\pm{\xi\mp(k+2)-(\rho\mp 2) \over k}}~\langle  \xi\mp(k+2) ,\rho\mp 2 ,\sigma\vert  {\rm Osc.}(z)\\
&\quad \quad\qquad \qquad = z^{-{1\over 2}\pm{\xi-\rho \over k}}~\langle  \xi\mp(k+2) ,\rho\mp 2 ,\sigma\vert  {\rm Osc.}(z),
\end{align*}
where ${\rm Osc.}(z)$ stands for some invertible element in the algebra of positive Fourier modes, {\it i.e.} 
some power series in $z^{-1}$
with a non vanishing constant term.
\end{lem}

In the NS sector ${\bf G}^{\pm}(z)$ should be expanded in the half-integral powers of $z$. 
The condition for  $\langle \xi,\rho,\sigma\vert$ being a highest weight vector 
gives us $\langle \xi,\rho,\sigma\vert {\bf G}^{\pm}_m=0$ for $m=-1/2,-3/2,\ldots$, 
which in view of Lemma \ref{G-zero-dual} requires $\xi=\rho$. 
\begin{lem} When $\xi=\rho$, 
\begin{align*}
&\langle \xi,\rho,\sigma\vert {\bf K}^\pm(z) =q^{\rho}\langle \xi,\rho,\sigma\vert+\cdots,\\
&\langle \xi,\rho,\sigma\vert{\bf T}(z)=
q^{-{2\over k}\xi}
\left(
q^{-1}q^{{k+2\over k} \rho }q^{-\sigma}
 -{[k+2]\over [k+1]}q^{{k+2\over k} \rho } q^\rho 
  +q^{+1}q^{{k+2\over k} \rho }q^{\sigma} \right)\langle \xi,\rho,\sigma\vert  +\cdots\\
  &\qquad =
\left(
q^{-1}q^{\rho }q^{-\sigma}
 -{[k+2]\over [k+1]}q^{2\rho}
  +q^{+1}q^{\rho }q^{\sigma} \right)\langle \xi,\rho,\sigma\vert +\cdots.
\end{align*}
\end{lem}
Compare these with the definition of parameters $u$ and $h$ for the Verma module in the NS sector (see \eqref{hwv-dual}). 
From the parametrization \eqref{huv} of $h$, we can identity 
\begin{equation}
u=q^\rho,\qquad v=q^\sigma.
\end{equation}

Set
$$
{\mathcal F}^*(\xi,\rho,\sigma) \seteq
\langle \xi,\rho,\sigma\vert
\mathbb{C}[ \widetilde{\alpha}_{1}, \widetilde{\alpha}_{2},\ldots, \overline{\alpha}_{1},\overline{\alpha}_{2},\ldots,\beta_1,\beta_2,\ldots],
$$
and
\begin{align}
 {\mathcal F}^*_{\rm NS}(u,v) \seteq
 \bigoplus_{n \in \mathbb{Z}} {\mathcal F}^*(\rho+(k+2)n,\rho+2n,\sigma). \label{NS-dualFock}
\end{align}
\begin{prp}
We have a representation of $\Svir$ in the {\rm NS} sector on the dual
Fock space $ {\mathcal F}^*_{\rm NS}(u,v)$
with the highest weight vector $\langle \rho,\rho,\sigma\vert$ satisfying
\begin{align*}
&\langle \rho,\rho,\sigma\vert {\bf K}^\pm_0=u\langle \rho,\rho,\sigma\vert,\qquad 
\langle \rho,\rho,\sigma\vert {\bf T}_0=h(u,v)\langle \rho,\rho,\sigma\vert,
\end{align*}
where $h(u,v)$ is defined in (\ref{huv}). 

\end{prp}

In the R sector, ${\bf G}^{\pm}(z)$ should be expanded in the integral powers of $z$. 
The condition for  $\langle  \xi,\rho,\sigma\vert$ being a highest weight vector 
gives us $\langle  \xi,\rho,\sigma\vert {\bf G}^{+}_m=0$ for $m=1,2,\ldots$, and 
$\langle  \xi,\rho,\sigma\vert {\bf G}^{-}_m=0$ for $m=0,1,\ldots$, 
which in view of Lemma \ref{G-zero-dual} requires $\xi = \rho + k/2$. 

\begin{lem}
When $\xi = \rho + k/2$, 
\begin{align*}
&\langle  \xi,\rho,\sigma\vert {\bf K}^\pm(z) =q^{\xi}\langle  \xi,\rho,\sigma\vert+\cdots,\\
&\langle  \xi,\rho,\sigma\vert{\bf T}(z)=
q^{-{2\over k}\xi}
\left(
q^{-1}q^{{k+2\over k} \rho }q^{-\sigma}
 -{[k+2]\over [k+1]}q^{{k+2\over k} \rho } q^\rho 
  +q^{+1}q^{{k+2\over k} \rho }q^{\sigma} \right) \langle  \xi,\rho,\sigma\vert+\cdots\\
&\qquad =
q^{-1}
\left(
q^{-1}q^{\rho}q^{-\sigma}
 -{[k+2]\over [k+1]}q^{2\rho}
  +q^{+1}q^{\rho}q^{\sigma} \right) \langle  \xi,\rho,\sigma\vert+\cdots.
\end{align*}
\end{lem}
Hence we have the identification of the parameters:
\begin{equation}
q^{k/2}u=q^\xi,\qquad u=q^\rho,\qquad 
 v=q^\sigma.
\end{equation}

Set
\begin{align}
 {\mathcal F}^*_{\rm R}(u,v) \seteq
 \bigoplus_{n \in \mathbb{Z}} {\mathcal F}^*(k/2+\rho+(k+2)n,\rho+2n,\sigma). 
\end{align}
\begin{prp}
We have a representation of $\Svir$ in the {\rm R} sector on the dual
Fock space $ {\mathcal F}_{\rm R}(u,v)$
with the highest weight vector $\langle vk/2+\rho,\rho,\sigma\vert$ satisfying
\begin{align*}
&\langle k/2+\rho,\rho,\sigma\vert {\bf K}^\pm_0=q^{k/2}u \langle k/2+\rho,\rho,\sigma\vert,\quad 
\langle k/2+\rho,\rho,\sigma\vert {\bf T}_0=q^{-1}h(u,v) \langle k/2+\rho,\rho,\sigma\vert,
\end{align*}
where $h(u,v)$ is defined in (\ref{huv}). 
\end{prp}


\subsection{Commutation relations of the generating currents}

In this subsection we prove;

\begin{thm}\label{Wakimoto-Svir}
The operators $\mathbf{K}(z),\mathbf{T}(z)$ and  $\mathbf{G}^\pm(z)$ given in Definition \ref{mod-Wakimoto}
satisfy the relations in Proposition \ref{generating_functions}.
\end{thm}


\subsubsection{Relations for $\mathbf{K}^\pm(z)$ vs. $\mathbf{K}^\pm(w)$,  $\mathbf{K}^\pm(z)$ vs. 
$\mathbf{G}^\pm(w)$, and  $\mathbf{K}^\pm(z)$ vs. $\mathbf{T}(w)$}

Since the operators $\mathbf{K}^\pm(z)$ and $\widetilde{V}^\pm(z)$ involve 
only the twisted $U(1)$ modes $\widetilde{\alpha}_n$ and $Q_{\widetilde{\alpha}}$, it is straightforward 
to obtain the OPE's among them from the commutation relations given by Definition \ref{modified-Cartan}.
In order to write down these OPE's, it is convenient to introduce
the invertible Fourier series $\widetilde{g}(z)$ by
\begin{equation}
\widetilde{g}(z)^{\pm 1}=q^{\pm (k+2)}
\exp\left( \pm  \sum_{m>0}  {1\over m}(q^{(k+2)m}-q^{-(k+2)m}) z^m \right)=
{q^{\pm (k+2)}-z\over 1-q^{\pm (k+2)} z}.
\end{equation}
Note that when we evaluate the OPE of $A(u)B(v)$, the ordering $|u| > |v|$ is always assumed.

\begin{lem}\label{tVtV}
The OPE's among the vertex operators $\mathbf{K}^\pm(z)$ and $\widetilde{V}^\pm(z)$ read 
as follows;
\begin{enumerate}
\item
\begin{align*}
&\mathbf{K}^\pm(z)\mathbf{K}^\pm(w)=~:\mathbf{K}^\pm(z)\mathbf{K}^\pm(w):,\\
&\mathbf{K}^-(z)\mathbf{K}^+(w)=~:\mathbf{K}^-(z)\mathbf{K}^+(w):,\\
&
\mathbf{K}^+(w)\mathbf{K}^-(z)=~:\mathbf{K}^-(z)\mathbf{K}^+(w):
\frac{(1-q^{2}z/w)(1-q^{-2}z/w)}{(1-q^{2k+2}z/w)(1-q^{-2k-2}z/w)}.
\end{align*}
\item
\begin{align*}
&\mathbf{K}^-(z) \widetilde{V}^\pm (w)=~:\mathbf{K}^-(z) \widetilde{V}^\pm (w):q^{\pm (k+2)},\\
&
\widetilde{V}^\pm (w)\mathbf{K}^-(z) =~:\mathbf{K}^-(z) \widetilde{V}^\pm (w):q^{\pm (k+2)}\widetilde{g}(z/w)^{\mp 1},\\
&\widetilde{V}^\pm (z )\mathbf{K}^+(w)=~:\widetilde{V}^\pm (z )\mathbf{K}^+(w):,\\
&\mathbf{K}^+(w)\widetilde{V}^\pm (z )=~:\widetilde{V}^\pm (z )\mathbf{K}^+(w):\widetilde{g}(z/w)^{\pm 1}.
\end{align*}
\item
\begin{align*}
&\widetilde{V}^\pm(z)\widetilde{V}^\pm (w)=~:\widetilde{V}^\pm(z)\widetilde{V}^\pm (w):z^{+{k+2\over k}}
\exp\left(  -\sum_{m>0} {[(k+2)m]\over m [km]} (w/z)^m\right),\\
&\widetilde{V}^\pm(z)\widetilde{V}^\mp(w)=~:\widetilde{V}^\pm(z)\widetilde{V}^\mp(w):z^{-{k+2\over k}}
\exp\left( + \sum_{m>0}  {[(k+2)m]\over m [km]} (w/z)^m\right).
\end{align*}
\end{enumerate}
\end{lem}

Since the operator $\mathbf{K}^\pm(z)$ is independent of the modes of the parafermion sector,
there are no contributions from the parafermion sector in the relations 
$\mathbf{K}^\pm(z)$ vs. $\mathbf{G}^\pm(w)$ and  $\mathbf{K}^\pm(z)$ vs. $\mathbf{T}(w)$.
Note that the $U(1)$ part of $\mathbf{T}(z)$ is given by 
\begin{equation*}
\mathbf{T}_\mathrm{U(1)}(z) = :\widetilde{V}^+(q^{-1}z)\widetilde{V}^-(q^{+1}z): 
= q^{-\frac{2}{k}\widetilde{\alpha}_0} 
: \exp \left( - (q - q^{-1}) \sum_{m \neq 0} \frac{[m]}{[km]}\widetilde{\alpha}_{-m}  z^m \right) :.
\end{equation*}
Then the OPE's in Lemma \ref{tVtV} are enough to obtain the following relations.
\begin{prp}
\begin{align*}
&\mathbf{K}^\pm(z)\mathbf{K}^\pm(w)=\mathbf{K}^\pm(w)\mathbf{K}^\pm(z),\\
&\mathbf{K}^-(z)\mathbf{K}^+(w)=\mathbf{K}^+(w)\mathbf{K}^-(z)
\frac{(1-q^{2k+2}z/w)(1-q^{-2k-2}z/w)}{(1-q^{2}z/w)(1-q^{-2}z/w)},\\
&\mathbf{K}^-(z)\mathbf{T}(w)=\mathbf{T}(w)\mathbf{K}^-(z) \widetilde{g}(q z/w) \widetilde{g}(q^{-1} z/w)^{-1},\\
&\mathbf{T}(z)\mathbf{K}^+(w)=\mathbf{K}^+(w)\mathbf{T}(z) \widetilde{g}(q z/w) \widetilde{g}(q^{-1} z/w)^{-1},\\
&\mathbf{K}^-(z)\mathbf{G}^\pm(w)=\mathbf{G}^\pm(w)\mathbf{K}^-(z) \widetilde{g}(q^{\pm (k+2)} z/w)^{\pm 1},\\
&\mathbf{G}^\pm(z)\mathbf{K}^+(w)=\mathbf{K}^+(w)\mathbf{G}^\pm(z) \widetilde{g}(q^{\mp (k+2)} z/w)^{\mp 1}.
\end{align*}
\end{prp}

Thus, we recover the relations \eqref{rr-2} -- \eqref{rr-7} in Proposition \ref{generating_functions}.

\subsubsection{Relations for
$\mathbf{G}^\pm(z)$ vs. $\mathbf{G}^\pm(w)$}

Since $\mathbf{G}^\pm(z)$ involves the vertex operators from the parafermion sector,
we need the OPE's among $e_{\pm}(z)$ and $f_{\pm}(z)$ which are worked out in Appendix \ref{proof-Wakimoto}
(see Propositions \ref{ee}, \ref{ff} and \ref{ef}.) 

\begin{prp}
We have
\begin{align*}
&\mathbf{G}^\pm(z)\mathbf{G}^\pm(w)+\mathbf{G}^\pm(w)\mathbf{G}^\pm(z)=0,\\
&
\mathbf{G}^+(z)\mathbf{G}^-(w)
+
\mathbf{G}^-(w)\mathbf{G}^+(z)\\
=& 
{1\over (q-q^{-1})^2}\Biggl(
\delta^{\rm A}\left(q^{4k+4} { w\over z}\right){1\over [k+1]}\mathbf{K}(q^{2k+2}w)
-\delta^{\rm A}\left(q^{2k+4} { w\over  z}\right) \\
&\qquad\qquad +\delta^{\rm A}\left(q^{2k+2} { w\over   z}\right)  \mathbf{T}(q^{k+1}w) \Biggr),\qquad ({\rm A}={\rm NS},{\rm R}).
\end{align*}

\end{prp}

There relations are nothing but \eqref{rr-8} and \eqref{rr-9} in Proposition \ref{generating_functions}.

\begin{proof}
Combining the formulas in Lemma \ref{tVtV} (3) and the OPE's among $e_{\pm}(z)$ 
and $f_{\pm}(z)$, we compute the factors coming from the normal ordering.
There is a nice cancellation between the OPE coefficients of the $U(1)$ boson  part 
and the parafermion part, which leads to the simple relations 
for $\mathbf{G}^\pm(z)$ vs. $\mathbf{G}^\pm(w)$.
Firstly, the normal ordering of the zero modes gives\footnote{In the case of $U_q(\widehat{\mathfrak{sl}}_2)$, we do not have 
the factor $z$ from the normal ordering of the zero modes. The additional factor $z$ is responsible 
for the fermionic nature of $\mathbf{G}^\pm(z)$.}
\begin{equation*}
 (q^{\mp(k+2)}z)^{\frac{k+2}{k}} (q^{\mp\frac{1}{2}(3k+4)}z)^{-\frac{2}{k}} = z q^{\mp(k+1)},
\end{equation*}
where the first factor on the left hand side is from $\widetilde{V}^{\pm}(q^{\mp(k+2)}z)$
and the second factor is from the parafermion sector. 
The normal ordering of the oscillators produces the factor
\begin{equation*}
\exp \left( - \sum_{m>0} \frac{[(k+2)m]}{m [km]} (w/z)^m \right)
\exp \left( \sum_{m>0} \frac{[2m]}{m [km]} q^{\mp km}  (w/z)^m  \right) 
= (1 - q^{\pm 2}(w/z))
\end{equation*}
with the additional factor $(q^{\pm \epsilon_1}z - q^{\pm \epsilon_2}w)/(z - q^{\pm 2} w)$. 
Multiplying all these factors, we obtain
\begin{align*}
&\mathbf{G}^\pm_{\epsilon_1}(z)\mathbf{G}^\pm_{\epsilon_2}(w)=~q^{\mp(k+1)}
:\mathbf{G}^\pm_{\epsilon_1}(z)\mathbf{G}^\pm_{\epsilon_2}(w):
(q^{\pm \epsilon_1}z-q^{\pm \epsilon_2}w),
\end{align*}
hence 
\begin{align*}
&\mathbf{G}^\pm_{\epsilon_1}(z)\mathbf{G}^\pm_{\epsilon_2}(w)+
\mathbf{G}^\pm_{\epsilon_2}(w)\mathbf{G}^\pm_{\epsilon_1}(z)=0.
\end{align*}

Similarly for the OPE's between $\mathbf{G}^+_{\epsilon_1}(z)$ and $\mathbf{G}^-_{\epsilon_2}(w)$,
we have 
\begin{equation*}
 (q^{-k-2}z)^{-\frac{k+2}{k}} (q^{-\frac{1}{2}(3k+4)}z)^{\frac{2}{k}} = z^{-1} q^{k+1}
\end{equation*}
and 
\begin{equation*}
\exp \left( \sum_{m>0} \frac{[(k+2)m]}{m [km]} (q^{(2k+4)}w/z)^m \right)
\exp \left( - \sum_{m>0} \frac{[2m]}{m [km]} (q^{(3k+4)}w/z)^m \right) 
= (1-q^{2k+2}w/z)^{-1}.
\end{equation*}
Hence, together with the factors that depend on $\epsilon_1$ and $\epsilon_2$, we have
\begin{align*}
&\mathbf{G}^+_{\epsilon_1}(z)\mathbf{G}^-_{\epsilon_2}(w)=~:\mathbf{G}^+_{\epsilon_1}(z)\mathbf{G}^-_{\epsilon_2}(w):
q^{k+1}
{1\over z-q^{2k+2}w}{q^{- \epsilon_1}z-q^{(k+1)\epsilon_2+3k+4}w\over z-q^{k \epsilon_2+3k+4}w},\\
&\mathbf{G}^-_{\epsilon_2}(w)\mathbf{G}^+_{\epsilon_1}(z)=~:\mathbf{G}^+_{\epsilon_1}(z)\mathbf{G}^-_{\epsilon_2}(w):
q^{-k-1}
{1\over w-q^{-2k-2}z}{q^{(k+1)\epsilon_2+3k+4}w-q^{- \epsilon_1}z\over q^{k \epsilon_2+3k+4}w-z}.
\end{align*}
By using the lemma below and the relation\footnote{Since the mode expansion of $\mathbf{G}^\pm_{\epsilon}(z)$ is
flipped, compared with $\mathbf{G}^\pm(z)$, we should take it into account, when we apply Lemma \ref{delta-flip}.}
\begin{align*}
:\mathbf{G}^+_{\epsilon_1}(z)\mathbf{G}^-_{\epsilon_2}(w):\delta^\mathrm{NS} (q^\alpha w/z)
&= :\mathbf{G}^+_{\epsilon_1}(q^\alpha w)\mathbf{G}^-_{\epsilon_2}(w):\delta^\mathrm{NS}(q^\alpha w/z), \qquad\hbox{in the NS sector,}\\
:\mathbf{G}^+_{\epsilon_1}(z)\mathbf{G}^-_{\epsilon_2}(w):\delta^\mathrm{NS} (q^\alpha w/z)
&= :\mathbf{G}^+_{\epsilon_1}(q^\alpha w)\mathbf{G}^-_{\epsilon_2}(w): \delta(q^\alpha w/z), \qquad\hbox {in the R sector,}
\end{align*}
we finally obtain
\begin{align*}
&\mathbf{G}^+(z) \mathbf{G}^-(w) + \mathbf{G}^-(w) \mathbf{G}^+(z) 
= z^{1/2} w^{1/2} \sum_{\epsilon_1, \epsilon_2} \epsilon_1\epsilon_2 \left(\mathbf{G}^+_{\epsilon_1}(z) \mathbf{G}^-_{\epsilon_2}(w) 
+ \mathbf{G}^-_{\epsilon_2}(w) \mathbf{G}^+_{\epsilon_1}(z) \right) \CR
&=~\frac{1}{(q-q^{-1})^2} \left[
{1\over [k+1]} \delta^{\rm A}\left(q^{4k+4} {w\over z}\right) :\widetilde{V}^+(q^{3k+2} w)\widetilde{V}^-(q^{k+2} w):
-~\delta^{\rm A}\left(q^{2k+4} {w\over z}\right) \right.\CR
& \qquad\qquad 
\left. +\delta^{\rm A}\left(q^{2k+2} {w\over z}\right) :\widetilde{V}^+(q^k w)\widetilde{V}^-(q^{k+2} w): 
\mathbf{T}_\mathrm{PF}(q^{k+1}w) \right] \CR
&=~\frac{1}{(q-q^{-1})^2} \left[
{1\over [k+1]} \delta^{\rm A}\left(q^{4k+4} {w\over z}\right) \mathbf{K}(q^{2k+2}w) 
-~\delta^{\rm A}\left(q^{2k+4} {w\over z}\right)
+\delta^{\rm A}\left(q^{2k+2} {w\over z}\right) \mathbf{T}(q^{k+1}w) \right].
\end{align*}
\end{proof}

\begin{lem}
We have
\begin{align*}
&{q^{k+1} z^{1/2} w^{1/2} \over z-q^{2k+2}w}{q^{- \epsilon_1}z-q^{(k+1)\epsilon_2+3k+4}w\over z-q^{k \epsilon_2+3k+4}w}+
{q^{-k-1} z^{1/2} w^{1/2} \over w-q^{-2k-2}z}{q^{(k+1)\epsilon_2+3k+4}w-q^{- \epsilon_1}z\over q^{k \epsilon_2+3k+4}w-z}\\
&=
\begin{cases}
\displaystyle 
{[k+2]\over [k+1]}\delta^\mathrm{NS} \left(q^{2k+2} {w\over z}\right)  - 
{1\over [k+1]}\delta^\mathrm{NS} \left(q^{4k+4} {w\over z}\right) \qquad&(\epsilon_1,\epsilon_2)=(+,+), \\
q^{-1}\delta^\mathrm{NS} \left(q^{2k+2} {w\over z}\right) \qquad&(\epsilon_1,\epsilon_2)=(+,-), \\
q^{+1}\delta^\mathrm{NS} \left(q^{2k+2} {w\over z}\right) \qquad&(\epsilon_1,\epsilon_2)=(-,+), \\
\delta^\mathrm{NS} \left(q^{2k+4} {w\over z}\right) \qquad&(\epsilon_1,\epsilon_2)=(-,-),
\end{cases}
\end{align*}
and
\begin{align*}
&:\mathbf{G}^+_{\epsilon_1}(q^{2k+2}w)\mathbf{G}^-_{\epsilon_2}(w):~=~
-{1\over (q-q^{-1})^2}:\widetilde{V}^+(q^k w)\widetilde{V}^-(q^{k+2} w)e_{\epsilon_1}(q^{k\over 2}w)f_{\epsilon_2}(q^{3k+4\over 2}w):,\\
&:\mathbf{G}^+_{+}(q^{4k+4}w)\mathbf{G}^-_{+}(w):~=~
-{1\over (q-q^{-1})^2}:\widetilde{V}^+(q^{3k+2} w)\widetilde{V}^-(q^{k+2} w):,\\
&:\mathbf{G}^+_{-}(q^{2k+4}w)\mathbf{G}^-_{-}(w):~=~
-{1\over (q-q^{-1})^2}.
\end{align*}
\end{lem}
\begin{proof}
We check the case $(\epsilon_1,\epsilon_2)=(+,+)$. Other cases are similar. 
\begin{align*}
&{q^{k+1} z^{1/2} w^{1/2} \over z-q^{2k+2}w}{q^{-1}z-q^{4k+5}w\over z-q^{4k+4}w}+
{q^{-k-1} z^{1/2} w^{1/2} \over w-q^{-2k-2}z}{q^{4k+5}w-q^{-1}z\over q^{4k+4}w-z}\\
&=  \frac{[k+2]}{[k+1]} \frac{ q^{k+1} z^{1/2} w^{1/2}}{z-q^{2k+2}w} 
- \frac{1}{[k+1]} \frac{q^{2k+2} z^{1/2} w^{1/2}}{z-q^{4k+4}w} \\
& \qquad + \frac{[k+2]}{[k+1]} \frac{ q^{-k-1} z^{1/2} w^{1/2}}{w-q^{-2k-2}z} 
- \frac{1}{[k+1]} \frac{q^{-2k-2} z^{1/2} w^{1/2}}{w-q^{-4k-4}z} \\
&=   \frac{[k+2]}{[k+1]} \frac{ (q^{2k+2} w/z)^{\frac{1}{2}}}{1-q^{2k+2}w/z} 
- \frac{1}{[k+1]} \frac{(q^{4k+4}w/z)^{\frac{1}{2}}}{1-q^{4k+4}w/z} \\
& \qquad + \frac{[k+2]}{[k+1]} \frac{(q^{-2k-2}z/w)^{\frac{1}{2}}}{1-q^{-2k-2}z/w} 
- \frac{1}{[k+1]} \frac{(q^{-4k-4}z/w)^{\frac{1}{2}}}{1-q^{-4k-4}z/w} \\
&= \frac{[k+2]}{[k+1]} \delta^\mathrm{NS} \left(q^{2k+2} {w\over z}\right) 
- \frac{1}{[k+1]} \delta^\mathrm{NS} \left(q^{4k+4} {w\over z}\right) . 
\end{align*}
\end{proof}

Proofs of G-T relation and T-T relation namely \eqref{rr-10} -- \eqref{rr-11} 
in Proposition \ref{generating_functions} are relegated to Appendix \ref{TGandTT}.


\vspace{5mm}

\begin{ack}
We would like to thank M.Bershtein, O.Blondeau-Fournier, M.Fukuda, R.Kodera, H.Konno, 
Y.Matsuo, H.Nakajima, G.Noshita, R.Ohkawa, Y.Ohkubo, Y.Saito, Y.Sugawara and Y.Yoshida 
for useful discussions. 
Our work is supported in part by Grants-in-Aid for Scientific Research (Kakenhi);
23K03087 (H.K.) and 19K03512, 19K03530, 21K03180, 24K06753 (J.S.).
\end{ack}

\medskip

The data used for the current study will be made available on reasonable request.
The authors have no competing interests to declare that are relevant to the content of this article.

\medskip



\appendix

\section{Details of Wakimoto representation of $U_q(\widehat{\mathfrak{sl}}_2)$}\label{proof-Wakimoto}

\subsection{Relations for $\psi_\pm(z)$ vs. $\psi_\pm(w)$,  $\psi_\pm(z)$ vs. $E(w)$, and  $\psi_\pm(z)$ vs. $F(w)$}
One finds from the definition (\ref{psi}) and the OPE's given in Lemma \ref{VV} below that
\begin{align*}
&\psi_\pm(z)\psi_\pm(w)=~:\psi_\pm(z)\psi_\pm(w):,\quad \psi_\pm(w)\psi_\pm(z)=~:\psi_\pm(z)\psi_\pm(w):,\\
&\psi_-(z)\psi_+(w)=~:\psi_-(z)\psi_+(w):,\quad
\psi_+(w)\psi_-(z)=~:\psi_-(z)\psi_+(w): g(q^kz/w)g(q^{-k}z/w)^{-1},\\
&\psi_-(z) V^\pm(w)=~:\psi_-(z) V^\pm(w):q^{\mp 2} ,\quad\, V^\pm(w)\psi_-(z) =~:\psi_-(z) V^\pm(w):
q^{\mp 2} g(q^{\mp k/2}z/w)^{\mp 1},\\
&\psi_+(z) V^\pm(w)=~:\psi_-(z) V^\pm(w):g(q^{\mp k/2}z/w)^{\mp1},\quad 
V^\pm(w)\psi_+(z)=~:\psi_-(z) V^\pm(w):.
\end{align*}
Hence in view of the definitions  (\ref{E}) and (\ref{F})  we have the relations (\ref{Uq-1})-(\ref{Uq-4}).

\begin{lem}\label{VV}
The OPE's among the vertex operators $V^\pm(z)$ and $V^\pm(w)$ read
\begin{align*}
&V^\pm(z)V^\pm (w)=~:V^\pm(z)V^\pm (w):z^{+{2\over k}}
\exp\left(  -\sum_{m>0} {[2m]\over m [km]} q^{\mp km} (w/z)^m\right),\\
&V^\pm(z)V^\mp(w)=~:V^\pm(z)V^\mp(w):z^{-{2\over k}}
\exp\left( + \sum_{m>0}  {[2m]\over m [km]} (w/z)^m\right).
\end{align*}
\end{lem}

\subsection{Relations for $E(z)$ vs. $E(w)$}

We move on to the check of the relation (\ref{Uq-5}).

\begin{prp}\label{ee}
We have
\begin{align*}
z^{{2\over k}}
\exp\left(  -\sum_{m>0} {[2m]\over m [km]} q^{-km} (w/z)^m\right)
e_{\epsilon_1}(z)e_{\epsilon_2}(w)=~
:e_{\epsilon_1}(z)e_{\epsilon_2}(w): {q^{\epsilon_1}z-q^{\epsilon_2}w \over z-q^2 w}.
\end{align*}
\end{prp}

\begin{proof}
This can be checked by performing straightforward calculations using the OPE's in Lemmas \ref{YY} and \ref{for-ee} below.
\end{proof}

\begin{cor}
In view of the definition  (\ref{E}) and the OPE's in Lemma \ref{VV}, we have
\begin{align*}
E_{\epsilon_1}(z)E_{\epsilon_2}(w)=~
:E_{\epsilon_1}(z)E_{\epsilon_2}(w): {q^{\epsilon_1}z-q^{\epsilon_2}w \over z-q^2 w}.
\end{align*}
Hence we have (\ref{Uq-5}).
\end{cor}

\begin{proof}
We have
\begin{align*}
(z-q^2 w)E_{\epsilon_1}(z)E_{\epsilon_2}(w)=~
:E_{\epsilon_1}(z)E_{\epsilon_2}(w): (q^{\epsilon_1}z-q^{\epsilon_2}w ),\\
(w-q^2 z)E_{\epsilon_2}(w)E_{\epsilon_1}(z)=~
:E_{\epsilon_1}(z)E_{\epsilon_2}(w): (q^{\epsilon_2}w-q^{\epsilon_1}z ).
\end{align*}
Therefore we have
\begin{align*}
&(z-q^2 w)E_{\epsilon_1}(z)E_{\epsilon_2}(w)+(w-q^2 z)E_{\epsilon_2}(w)E_{\epsilon_1}(z)=0,\\
&(z-q^2 w)E(z)E(w)+(w-q^2 z)E(w)E(z)=0,
\end{align*}
proving  (\ref{Uq-5}).
\end{proof}

\begin{lem}\label{YY}
The OPE's among the vertex operators $Y^\pm(z)$ and $Y^\pm(w)$ read
\begin{align*}
&Y^\pm(z)Y^\pm (w)=~:Y^\pm(z)Y^\pm (w):z^{-{2\over k}}
\exp\left(  +\sum_{m>0} {[2m]\over m [km]} q^{\mp km} (w/z)^m\right),\\
&Y^\pm(z)Y^\mp(w)=~:Y^\pm(z)Y^\mp(w):z^{+{2\over k}}
\exp\left( - \sum_{m>0}  {[2m]\over m [km]} (w/z)^m\right).
\end{align*}
\end{lem}

\begin{lem}\label{for-ee}
The non trivial OPE's we need for the calculation of the products 
$e_{\epsilon_1}(z)e_{\epsilon_2}(w)$ are the following; 
\newline
\noindent
(1) for the cases $\epsilon_1=+, \epsilon_2=\pm $
\begin{align*}
&Z_+(q^{-{k+2\over 2}}z) Y^+(w)=~:Z_+(z) Y^+(w): q^{+1}
\exp\left(  \sum_{m>0}{1\over m}( q^{2m}-1)(w/z)^m\right),
\end{align*}
(2) for the cases $\epsilon_1=\pm, \epsilon_2=-$
\begin{align*}
&Y^+(z) Z_-(q^{{k+2\over 2}}w)= ~:Y^+(z) Z_-(q^{{k+2\over 2}}w):
 \exp\left( \sum_{m>0} {1\over m}( q^{2m}-1) (w/z)^m\right),
\end{align*}
(3) for the case $\epsilon_1=+,\epsilon_2=-$
\begin{align*}
&Z_+(q^{-{k+2\over 2}}z)W_+ (q^{-{k\over 2}}z)Z_-(q^{{k+2\over 2}}w)W_- (q^{{k\over 2}}w)\\
=&:Z_+(q^{-{k+2\over 2}}z)W_+ (q^{-{k\over 2}}z)Z_-(q^{{k+2\over 2}}w)W_- (q^{{k\over 2}}w):
\exp\left(-\sum_{m>0} {1\over m }(q^{2m}-2+q^{-2m})(w/z)^m\right),
\end{align*}
and (4) for the cases $\epsilon_1=-,\epsilon_2=\pm$
\begin{align*}
&Z_-(q^{+{k+2\over 2}}z) Y^+(w)=~:Z_-(q^{+{k+2\over 2}}z) Y^+(w): q^{-1}.
\end{align*}
\end{lem}

\subsection{Relations for $F(z)$ vs. $F(w)$}
\begin{prp}\label{ff}
We have
\begin{align*}
z^{{2\over k}}
\exp\left(  -\sum_{m>0} {[2m]\over m [km]} q^{+km} (w/z)^m\right)
f_{\epsilon_1}(z)f_{\epsilon_2}(w)=~
:f_{\epsilon_1}(z)f_{\epsilon_2}(w): {q^{-\epsilon_1}z-q^{-\epsilon_2}w \over z-q^{-2}w}.
\end{align*}

\end{prp}

\begin{proof}
This can be checked by performing straightforward calculations using the OPE's in Lemma \ref{YY} and Lemma \ref{for-ff} below.
\end{proof}

\begin{cor}
In view of the definition  (\ref{F}) and the OPE's in Lemma \ref{VV}, we have
\begin{align*}
F_{\epsilon_1}(z)F_{\epsilon_2}(w)=~
:F_{\epsilon_1}(z)F_{\epsilon_2}(w): {q^{-\epsilon_1}z-q^{-\epsilon_2}w \over z-q^{-2}w}.
\end{align*}
Hence we have (\ref{Uq-6}).
\end{cor}

\begin{proof}
We have
\begin{align*}
&( z-q^{-2}w)F_{\epsilon_1}(z)F_{\epsilon_2}(w)=~
:F_{\epsilon_1}(z)F_{\epsilon_2}(w): (q^{-\epsilon_1}z-q^{-\epsilon_2}w ),\\
&( w-q^{-2}z)F_{\epsilon_2}(w)F_{\epsilon_1}(z)=~
:F_{\epsilon_1}(z)F_{\epsilon_2}(w): (q^{-\epsilon_2}w-q^{-\epsilon_1}z ).
\end{align*}
Therefore we have
\begin{align*}
& (z-q^{-2}w)F_{\epsilon_1}(z)F_{\epsilon_2}(w)+( w-q^{-2}z)F_{\epsilon_2}(w)F_{\epsilon_1}(z)=0,\\
&(z-q^{-2}w)F(z)F(w)+( w-q^{-2}z)F(w)F(z)=0,
\end{align*}
proving  (\ref{Uq-6}).
\end{proof}

\begin{lem}\label{for-ff}
The non trivial OPE's we need for the calculation of the products 
$f_{\epsilon_1}(z)f_{\epsilon_2}(w)$ are the following; 
\newline
\noindent
(1) for the cases $\epsilon_1=+, \epsilon_2=\pm $
\begin{align*}
&Z_+(q^{+{k+2\over 2}}z) Y^-(w)=~:Z_+(z) Y^-(w): q^{-1}
\exp\left( - \sum_{m>0}{1\over m}(1- q^{-2m})(w/z)^m\right),
\end{align*}
(2) for the cases $\epsilon_1=\pm, \epsilon_2=-$
\begin{align*}
&Y^-(z) Z_-(q^{-{k+2\over 2}}w)=~ :Y^-(z) Z_-(q^{-{k+2\over 2}}w):
 \exp\left(- \sum_{m>0} {1\over m}( 1-q^{-2m}) (w/z)^m\right),
\end{align*}
(3) for the case $\epsilon_1=+,\epsilon_2=-$
\begin{align*}
&Z_+(q^{{k+2\over 2}}z)W_+ (q^{{k\over 2}}z)^{-1}W_- (q^{-{k\over 2}}z)^{-1}Z_-(q^{-{k+2\over 2}}w)\\
=&:Z_+(q^{{k+2\over 2}}z)W_+ (q^{{k\over 2}}z)^{-1}W_- (q^{-{k\over 2}}z)^{-1}Z_-(q^{-{k+2\over 2}}w):
\exp\left(-\sum_{m>0} {1\over m }(q^{2m}-2+q^{-2m})(w/z)^m\right),
\end{align*}
and (4) for the cases $\epsilon_1=-,\epsilon_2=\pm$
\begin{align*}
&Z_-(q^{-{k+2\over 2}}z) Y^-(w)=~:Z_-(q^{-{k+2\over 2}}z) Y^-(w): q^{+1}.
\end{align*}
\end{lem}

\subsection{Relations for $E(z)$ vs. $F(w)$}

\begin{prp}\label{ef}
We have
\begin{align*}
&z^{-{2\over k}}
\exp\left(  \sum_{m>0} {[2m]\over m [km]} (w/z)^m\right)
e_{\epsilon_1}(z)f_{\epsilon_2}(w)=~
:e_{\epsilon_1}(z)f_{\epsilon_2}(w): {q^{-\epsilon_1}z-q^{(k+1)\epsilon_2}w \over z- q^{k\epsilon_2}w}\\
&\qquad\qquad=~
:e_{\epsilon_1}(z)f_{\epsilon_2}(w): {q^{-(k+1)\epsilon_1}z-q^{\epsilon_2}w \over q^{-k\epsilon_1}z- w},\\
&z^{-{2\over k}}
\exp\left(  \sum_{m>0} {[2m]\over m [km]} (w/z)^m\right)
f_{\epsilon_1}(z)e_{\epsilon_2}(w)=~
:f_{\epsilon_1}(z)e_{\epsilon_2}(w): {q^{\epsilon_1}z-q^{-(k+1)\epsilon_2}w \over z- q^{-k\epsilon_2}w}\\
&\qquad\qquad =
~
:f_{\epsilon_1}(z)e_{\epsilon_2}(w): {q^{(k+1)\epsilon_1}z-q^{-\epsilon_2}w \over q^{k\epsilon_1}z- w},
\end{align*}
and 
\begin{align*}
&:e_{\pm}(q^{\pm k} w)f_{\pm }(w): ~=1.
\end{align*}
\end{prp}

\begin{proof}
This can be checked by performing straightforward calculations using the OPE's in Lemma \ref{YY} together with 
Lemma \ref{for-ef} and Lemma \ref {for-fe} presented below.
\end{proof}

\begin{cor}
In view of the definitions  (\ref{E}), (\ref{F})  and the OPE's in Lemma \ref{VV}, we have
\begin{align*}
E_{\epsilon_1}(z)F_{\epsilon_2}(w)=~
:E_{\epsilon_1}(z)F_{\epsilon_2}(w): {q^{-\epsilon_1}z-q^{(k+1)\epsilon_2}w \over z- q^{k\epsilon_2}w},\\
F_{\epsilon_2}(w)E_{\epsilon_1}(z)=~
:E_{\epsilon_1}(z)F_{\epsilon_2}(w): {q^{(k+1)\epsilon_2}w-q^{-\epsilon_1}z \over q^{k\epsilon_2}w- z},
\end{align*}
and 
\begin{align*}
&:E_{\pm}(q^{\pm k} w)F_{\pm}(w): ~=-{1\over (q-q^{-1})^2}\psi_\pm (q^{\pm k/2}w).
\end{align*}
Hence we have (\ref{Uq-7}).
\end{cor}

\begin{proof}
Note that we have
\begin{align*}
{q^{-\epsilon_1}z-q^{(k+1)\epsilon_2}w \over z- q^{k\epsilon_2}w}-{q^{(k+1)\epsilon_2}w-q^{-\epsilon_1}z \over q^{k\epsilon_2}w- z}
=(q^{-\epsilon_1}-q^{\epsilon_2})\delta\left( q^{k\epsilon_2} {w\over z}\right).
\end{align*}
Therefore we have
\begin{align*}
&[E_{\epsilon_1}(z),F_{\epsilon_2}(w)]=
(q^{-\epsilon_1}-q^{\epsilon_2})\delta\left( q^{k\epsilon_2} {w\over z}\right):E_{\epsilon_1}( q^{k\epsilon_2} w)F_{\epsilon_2}(w):,\\
&[E(z),F(w)]={1\over q-q^{-1}}
\left( \delta\left( q^{k} {w\over z}\right)\psi_+ (q^{k/2}w)- \delta\left( q^{-k} {w\over z}\right)\psi_- (q^{-k/2}w) \right),
\end{align*}
proving  (\ref{Uq-7}).
\end{proof}

\begin{lem}\label{for-ef}
The nontrivial OPE's we need for the calculation of the products
$e_{\epsilon_1}(z)f_{\epsilon_2}(w)$ are the following; 
\newline
\noindent
(1) for the cases $\epsilon_1=+, \epsilon_2=\pm $
\begin{align*}
&Z_+(q^{-{k+2\over 2}}z) Y^-(w)=~:Z_+(q^{-{k+2\over 2}}z) Y^-(w): q^{-1}
\exp\left( - \sum_{m>0}{1\over m}( q^{2m}-1)q^{km}(w/z)^m\right),
\end{align*}
(2) for the cases $\epsilon_1=\pm, \epsilon_2=-$
\begin{align*}
&Y^+(z) Z_-(q^{-{k+2\over 2}}w)=~ :Y^+(z) Z_-(q^{-{k+2\over 2}}w):
 \exp\left( \sum_{m>0} {1\over m}(1- q^{-2m}) q^{-km}(w/z)^m\right),
\end{align*}
(3) for the case $\epsilon_1=+,\epsilon_2=-$
\begin{align*}
&Z_+(q^{-{k+2\over 2}}z)W_+ (q^{-{k\over 2}}z)Z_-(q^{-{k+2\over 2}}w)W_- (q^{-{k\over 2}}w)^{-1}\\
=&:Z_+(q^{-{k+2\over 2}}z)W_+ (q^{-{k\over 2}}z)Z_-(q^{-{k+2\over 2}}w)W_- (q^{-{k\over 2}}w)^{-1}:\\
&\times
\exp\left(- \sum_{m>0} {1\over m }\left(
q^{km}(1-q^{2m})+q^{-km}(1-q^{-2m})
\right)(w/z)^m\right),
\end{align*}
and (4) for the cases $\epsilon_1=-,\epsilon_2=\pm$
\begin{align*}
&Z_-(q^{+{k+2\over 2}}z) Y^-(w)=~:Z_-(q^{+{k+2\over 2}}z) Y^-(w): q^{+1}.
\end{align*}
\end{lem}

\begin{lem}\label{for-fe}
The nontrivial OPE's we need for the calculation of the products
$f_{\epsilon_1}(z)e_{\epsilon_2}(w)$ are the following;
\newline
\noindent
(1) for the cases $\epsilon_1=+, \epsilon_2=\pm $
\begin{align*}
&Z_+(q^{+{k+2\over 2}}z) Y^+(w)=~:Z_+(q^{+{k+2\over 2}}z) Y^+(w): q^{+1}
\exp\left( - \sum_{m>0}{1\over m}( q^{-2m}-1)q^{-km}(w/z)^m\right),
\end{align*}
(2) for the cases $\epsilon_1=\pm, \epsilon_2=-$
\begin{align*}
&Y^-(z) Z_-(q^{+{k+2\over 2}}w)=~ :Y^-(z) Z_-(q^{+{k+2\over 2}}w):
 \exp\left( \sum_{m>0} {1\over m}(1- q^{2m}) q^{km}(w/z)^m\right),
\end{align*}
(3) for the case $\epsilon_1=+,\epsilon_2=-$
\begin{align*}
&Z_+(q^{+{k+2\over 2}}z)W_+ (q^{+{k\over 2}}z)^{-1}Z_-(q^{+{k+2\over 2}}w)W_- (q^{+{k\over 2}}w)\\
=&:Z_+(q^{+{k+2\over 2}}z)W_+ (q^{+{k\over 2}}z)^{-1}Z_-(q^{+{k+2\over 2}}w)W_- (q^{+{k\over 2}}w):\\
&\times
\exp\left(- \sum_{m>0} {1\over m }\left(
q^{km}(1-q^{2m})+q^{-km}(1-q^{-2m})
\right)(w/z)^m\right),
\end{align*}
and (4) for the cases $\epsilon_1=-,\epsilon_2=\pm$
\begin{align*}
&Z_-(q^{-{k+2\over 2}}z) Y^+(w)=~:Z_-(q^{-{k+2\over 2}}z) Y^+(w): q^{-1}.
\end{align*}

\end{lem}


\section{Proof of $T$-$G$ and $T$-$T$ relations}\label{TGandTT}

In this appendix we show \eqref{rr-10} and \eqref{rr-11} in Proposition \ref{generating_functions}.
%
For later convenience let us introduce the notations for the energy-momentum tensor of the $U(1)$ boson sector and 
the parafermion sector.
\begin{equation}
\mathbf{T}(z) = \mathbf{T}_{\mathrm{U(1)}}(z) \cdot \mathbf{T}_{\mathrm{PF}}(z),
\qquad \mathbf{T}_{\mathrm{U(1)}}(z) =~:\widetilde{V}^{+}(q^{-1}z)\widetilde{V}^{-}(qz):,
\end{equation}
where $\mathbf{T}_{\mathrm{PF}}(z)$ is given by the following sum\footnote{The sum comes from the fact
that the parafermion sector is identified with the deformed $W$-algebra of $\mathfrak{sl}_{2\vert 1}$ 
as explained in Section \ref{Yalgebra}. See also \cite{Jimbo:2000ff}.};
\begin{align}
\mathbf{T}_{\mathrm{PF}}(z)=&{~} \Lambda_1(z)+ \Lambda_0(z)+\Lambda_{-1}(z),\\
\Lambda_1(z)=&{~}q^{-1} :e_+(q^{-\frac{k+2}{2}}z)f_-(q^{\frac{k+2}{2}}z):,\CR
\Lambda_0(z)=&{~}-{[k+2]\over [k+1]}  :e_+(q^{-\frac{k+2}{2}}z)f_+(q^{\frac{k+2}{2}}z):,\CR
\Lambda_{-1}(z) =&{~}q :e_-(q^{-\frac{k+2}{2}}z)f_+(q^{\frac{k+2}{2}}z): \nonumber.
\end{align}
Recall that, compared with the Wakimoto representation of quantum affine algebra $U_q(\widehat{\mathfrak{sl}}_2)$,
the $U(1)$ boson sector is twisted, but the parafermion sector is kept intact.
Hence we can use the result of Appendix \ref{proof-Wakimoto} for the parafermion sector.

\subsection{$T$-$G$ relation}

Recall that in the main text we defined 
\begin{align*}
&\mathbf{G}_\epsilon^+(z)= + {1\over q-q^{-1}}\widetilde{V}^+(q^{-k-2}z) e_\epsilon(q^{-{3k+4\over2}}z),\\
&\mathbf{G}_\epsilon^-(z)=-  {1\over q-q^{-1}}\widetilde{V}^-(q^{+k+2}z) f_\epsilon(q^{+{3k+4\over2}}z).
\end{align*}

Since $\mathbf{T}(z)$ is bilinear in the vertex operators we need the following lemmas to compute
the commutation relations with $\mathbf{G}^{\pm}(w)$. 
These lemmas are easily derived from the Wick's theorem for the normal ordered products. 

\begin{lem}\label{Triple-V}
We have the following OPE relations between $\widetilde{V}^{\pm}(z)$ and the $q$-shifted normal ordered product of $\widetilde{V}^{\pm}(w)$;
\begin{align*}
&\widetilde{V}^{\pm}(z):\widetilde{V}^+(q^\mu w)\widetilde{V}^-( w) :\\
=&~\exp\left( \pm \sum_{m>0} {[(k+2)m]\over m[k m]} ( 1 -q^{\mu m})(w/z)^m\right)
:\widetilde{V}^{\pm}(z)\widetilde{V}^+(q^\mu w)\widetilde{V}^-( w):,\\
&:\widetilde{V}^+(q^\mu w)\widetilde{V}^-( w) :\widetilde{V}^{\pm}(z)\\
=&~q^{\pm \mu{k+2\over k}}
\exp\left( \pm \sum_{m>0} {[(k+2)m]\over m[k m]} (1 -q^{-\mu m})(z/w)^m\right)
:\widetilde{V}^{\pm}(z)\widetilde{V}^+(q^\mu w)\widetilde{V}^-( w):.
%
\end{align*}
\end{lem}

The corresponding OPE relations in the parafermion sector are obtained from the result in Appendix \ref{proof-Wakimoto} 
(see Propositions \ref{ee}, \ref{ff} and \ref{ef}).

\begin{lem}\label{Triple-ef}
We have 
\newline
\noindent
(1)~For the OPE relations $e_\epsilon(z)$ with $\Lambda_i(w)$
\begin{align*}
&\exp \left(\sum_{m>0}{[2m]\over m[km]}
\left( -q^{-km} q^{-(k+2)m} +1\right)(w/z)^m\right)
e_{\epsilon_1}(q^{\frac{k+2}{2}}z):e_{\epsilon_2}(q^{-\frac{k+2}{2}}w)f_{\epsilon_3}(q^{\frac{k+2}{2}}w):\\
=&{~}
{q^{\epsilon_1} z-q^{-k-2}q^{\epsilon_2}w\over z-q^{-k-2}q^2 w}\cdot {q^{-(k+1)\epsilon_1}z-q^{\epsilon_3}w\over q^{-k \epsilon_1}z-w}
:e_{\epsilon_1}(q^{\frac{k+2}{2}}z)e_{\epsilon_2}(q^{-\frac{k+2}{2}}w)f_{\epsilon_3}(q^{\frac{k+2}{2}}w):,\\
&q^{-\frac{2}{k}(k+2)} \cdot
\exp \left(\sum_{m>0}{[2m]\over m[km]}
\left( -q^{-km} q^{(k+2)m} +1\right)(z/w)^m\right)
 :e_{\epsilon_2}(q^{-\frac{k+2}{2}}w)f_{\epsilon_3}(q^{\frac{k+2}{2}}w):e_{\epsilon_1}(q^{\frac{k+2}{2}}z)\\
=&{~}
{q^{\epsilon_2} q^{-k-2}w-q^{\epsilon_1}z\over q^{-k-2} w-q^{2}z}\cdot {q^{(k+1)\epsilon_3}w-q^{-\epsilon_1}z\over q^{k \epsilon_3}w-z}
:e_{\epsilon_1}(q^{\frac{k+2}{2}}z)e_{\epsilon_2}(q^{-\frac{k+2}{2}}w)f_{\epsilon_3}(q^{\frac{k+2}{2}}w):,\\
\end{align*}
(2)~For the OPE relations $f_\epsilon(z)$ with $\Lambda_i(w)$
\begin{align*}
&\exp \left(\sum_{m>0}{[2m]\over m[km]}
\left(q^{-(k+2)m} -q^{km}\right)(w/z)^m\right)
 f_{\epsilon_1}(q^{\frac{k+2}{2}}z):e_{\epsilon_2}(q^{-\frac{k+2}{2}}w)f_{\epsilon_3}(q^{\frac{k+2}{2}}w):\\
=&{~}
{q^{(k+1)\epsilon_1} z-q^{-k-2}q^{-\epsilon_2}w\over q^{k \epsilon_1 }z-q^{-k-2} w}\cdot
 {q^{-\epsilon_1}z-q^{-\epsilon_3}w\over z-q^{-2}w}
:f_{\epsilon_1}(q^{\frac{k+2}{2}}z)e_{\epsilon_2}(q^{-\frac{k+2}{2}}w)f_{\epsilon_3}(q^{\frac{k+2}{2}}w):,\\
&q^{\frac{2}{k}(k+2)}
\exp \left(\sum_{m>0}{[2m]\over m[km]}
\left(  q^{(k+2)m} -q^{km}\right)(z/w)^m\right)
 :e_{\epsilon_2}(q^{-\frac{k+2}{2}}w)f_{\epsilon_3}(q^{\frac{k+2}{2}}w):f_{\epsilon_1}(q^{\frac{k+2}{2}}z)\\
=&{~}
{q^{-(k+1)\epsilon_2} q^{-k-2}w-q^{\epsilon_1}z\over q^{-k\epsilon_2}q^{-k-2} w-z}\cdot
 {q^{-\epsilon_3}w-q^{-\epsilon_1}z\over w-q^{-2}z}
:f_{\epsilon_1}(q^{\frac{k+2}{2}}z)e_{\epsilon_2}(q^{-\frac{k+2}{2}}w)f_{\epsilon_3}(q^{\frac{k+2}{2}}w):.\\
\end{align*}

\end{lem}


Applying Lemma \ref{Triple-V} with $z \to q^{\mp (k+2)} z, w \to qw$ and $\mu=-2$ we obtain the OPE relations 
of $\mathbf{T}_{\mathrm{U(1)}}(w)$ with $\mathbf{G}^{\pm}(z)$ (see the $U(1)$ boson sector in the definition of $\mathbf{G}^{\pm}$),
\begin{align}
&\widetilde{V}^{\pm}(q^{\mp (k+2)} z)\mathbf{T}_{\mathrm{U(1)}}(w) \CR
&=
\exp\left( \pm \sum_{m>0} {[(k+2)m]\over m[k m]} (q^m -q^{-m})q^{\pm (k+2)m} (w/z)^m\right)
:\widetilde{V}^{\pm}(q^{\mp (k+2)} z)\mathbf{T}_{\mathrm{U(1)}}(w):, \label{G+-T_U1}\\
&\mathbf{T}_{\mathrm{U(1)}}(w)\widetilde{V}^{\pm}(q^{\mp (k+2)} z) \CR
&=
q^{\mp 2{k+2\over k}}
\exp\left( \mp \sum_{m>0} {[(k+2)m]\over m[k m]} (q^m -q^{-m})q^{\mp (k+2)m} (z/w)^m\right)
:\widetilde{V}^{\pm}(q^{\mp (k+2)} z)\mathbf{T}_{\mathrm{U(1)}}(w): \label{G--T_U1}.
\end{align}

%
%

Similarly from Lemma \ref{Triple-ef} we can compute the commutation relations of $\mathbf{T}_{\mathrm{PF}}(q^{\mp(k+1)}w)$ 
and $\mathbf{G}^{\pm}(z)$. We apply Lemma \ref{Triple-ef} by substituting $z \to q^{-2k - 3}z$ for $e_\epsilon (z)$
and $z \to q^{k + 1}z$ for $f_\epsilon(z)$. Hence we have $(w/z) \to q^{k+2}(w/z)$ for the case $\mathbf{G}^+(z)$,
while $w/z$ is invariant for the case $\mathbf{G}^-(z)$.
Using the relation $:e_{\pm}(q^{\pm k}z)f_{\pm}(z):=1$ (See Prop. A.9), we have
\begin{align}\label{G+-T_PF}
&{1-q^{2} w/z\over 1- w/z}
\exp \left(\sum_{m>0}{[2m]\over m[km]}
( -q^{-km}  + q^{(k+2)m} )(w/z)^m\right) 
\mathbf{G}^+(z)
\mathbf{T}_{\mathrm{PF}} (q^{-k-1}w)
\CR
-&q^{-{2(k+2)\over k}}
{1-q^{2} z/w\over 1-z/w}
\exp \left(\sum_{m>0}{[2m]\over m[km]}
( -q^{-km} +  q^{-(k+2)m}) (z/w)^m \right) 
\mathbf{T}_{\mathrm{PF}} (q^{-k-1}w) 
\mathbf{G}^+(z)\CR
=&~q {[k+2]\over [k+1]} \delta\left( q^{2k+2} w\over z\right) z^{1/2}
:\widetilde{V}^+(q^{-k-2}z) \left( e_+(q^{-\frac{1}{2}(3k+4)}w)- e_-(q^{-\frac{1}{2}(3k+4)}w) \right):,
\end{align}

\begin{align}\label{G--T_PF}
&{1-q^{-2} w/z\over 1-w/z}
\exp \left(\sum_{m>0}{[2m]\over m[km]}
\left( q^{-(k+2)m} -q^{km} \right)(w/z)^m\right) \mathbf{G}^-(z) \mathbf{T}_{\mathrm{PF}} (q^{k+1}w)\CR
-&q^{{2(k+2)\over k}}
{1-q^{-2} z/w\over 1-z/w}
\exp \left(\sum_{m>0}{[2m]\over m[km]}
\left(  q^{(k+2)m} -q^{km}\right)(z/w)^m\right) \mathbf{T}_{\mathrm{PF}} (q^{k+1}w) \mathbf{G}^-(z) \CR
=&~q^{-1}{[k+2]\over [k+1]}
\delta\left( q^{-2k-2} w\over z\right) z^{1/2} :\widetilde{V}^+(q^{k+2}z)  
\left( f_+(q^{\frac{1}{2}(3k+4)}w)- f_-(q^{\frac{1}{2}(3k+4)}w) \right):.
\end{align}

Combining \eqref{G+-T_U1} and \eqref{G--T_U1}, where we substitute $w \to q^{\mp(k+1)} w$, 
with \eqref{G+-T_PF} and \eqref{G--T_PF},
we obtain the total contribution to the commutation relations of $\mathbf{G}^{\pm}(z)$ and $\mathbf{T}(q^{\mp(k+1)}w)$.
It is remarkable the structure functions from the $U(1)$ boson sector and the parafermion sector exactly cancel and
we have commutation relations without a structure function. For example the following identity implies such a cancellation.
\begin{align}\label{cancellation}
& \frac{1-q^2z}{1-z} \exp \left( \sum_{m>0} \frac{[2m]}{m[km]} (-q^{-km} + q^{\pm (k+2)m}) z^m \right) \CR
=&{~} \exp \left(\sum_{m>0}\frac{1}{m[km]}\left( [km] (1-q^{2m}) + [2m] (-q^{-km} + q^{\pm (k+2)m})\right)  z^m \right) \CR
=&{~} \exp \left( \pm \sum_{m>0}\frac{[(k+2)m]}{m[km]}q^{\pm m}(q^m - q^{-m})  z^m\right).
\end{align}
Thus, we obtain\footnote{We use $z^{1/2}\delta\left( q^{\pm(2k+2)}{ w\over z}\right) 
= q^{\pm(k+1)}w^{1/2}\delta^\mathrm{NS}\left( q^{\pm(2k+2)}{ w\over z}\right)$.}
\begin{align}
& \mathbf{G}^{+}(z) \mathbf{T}(q^{-k-1}w) - \mathbf{T}(q^{-k-1}w) \mathbf{G}^{+}(z) \CR
=&{~}q^{k+2}{[k+2]\over [k+1]}
\delta^\mathrm{A} \left( q^{2k+2}{ w\over z}\right) w^{1/2} \CR
& \qquad \times : \left( e_+(q^{-\frac{1}{2}(3k+4)}w)- e_-(q^{-\frac{1}{2}(3k+4)}w) \right)
\widetilde{V}^{+}(q^{-k}w) \mathbf{T}^{+}_{\mathrm{U(1)}}(q^{-k-1}w): \CR
=&{~}(q-q^{-1})q^{k+2}{[k+2]\over [k+1]}
\delta^\mathrm{A} \left( q^{2k+2}{ w\over z}\right) :\mathbf{G}^{+}(w)\mathbf{K}(w):,
\end{align}
and
\begin{align}
& \mathbf{G}^{-}(z) \mathbf{T}(q^{k+1}w) - \mathbf{T}(q^{k+1}w) \mathbf{G}^{-}(z) \CR
=&{~}q^{-k-2} {[k+2]\over [k+1]}
\delta^\mathrm{A} \left( q^{-(2k+2)}{ w\over z}\right) w^{1/2} \CR
& \qquad \times :\left( f_+(q^{\frac{1}{2}(3k+4)}w)- f_-(q^{\frac{1}{2}(3k+4)}w) \right)
 \widetilde{V}^{-}(q^{k} w) \mathbf{T}^{-}_{\mathrm{U(1)}}(q^{k+1}w): \CR
=&{~}-(q-q^{-1})q^{-k-2}{[k+2]\over [k+1]}
\delta^\mathrm{A} \left( q^{-(2k+2)}{ w\over z}\right) :\mathbf{G}^{-}(w)\mathbf{K}(w):.
\end{align}
Since the zero mode of $\mathbf{K}^\pm(z)$ is $q^{\widetilde{\alpha}_0}$ and 
that of $\mathbf{G}^{\pm}(z)$ involves $e^{\pm (k+2) Q_{\widetilde\alpha}}$, we have
\begin{equation}
\mathbf{K}^{-}(z) \mathbf{G}^{\pm}(z) \mathbf{K}^{+}(z) = q^{\pm(k+2)} :\mathbf{G}^\pm(z) \mathbf{K}(z):.
\end{equation}
Hence, we recover \eqref{rr-10}.
\begin{prp}
The commutation relations of $\mathbf{G}^{\pm}(z)$ and $\mathbf{T}(w)$ are
\begin{align*}
& \mathbf{G}^{\pm}(z) \mathbf{T}(w) - \mathbf{T}(w) \mathbf{G}^{\pm}(z)\\
=& \pm (q-q^{-1}){[k+2]\over [k+1]}
\delta^\mathrm{A} \left( q^{\pm (3k+3)}{ w\over z}\right) 
\mathbf{K}^{-}(q^{\pm(k+1)}w)\mathbf{G}^{\pm}(q^{\pm(k+1)}w)\mathbf{K}^{+}(q^{\pm(k+1)}w).
\end{align*}
\end{prp}
%
%
%

\subsection{$T$-$T$ relation}


Similarly to the case of $G$-$T$ commutation relations, we start with two lemmas which are derived from the Wick's theorem. 
\begin{lem}\label{quadratic-Cartan}
We have the following OPE relation
\begin{align*}
&:\widetilde{V}^+(q^\mu z)\widetilde{V}^-(z):
:\widetilde{V}^+(q^\mu w)\widetilde{V}^-(w): \CR
&= \exp
\left(\sum_{m>0}{[(k+2)m]\over m[km]} ( q^{\frac{\mu}{2}m} - q^{-\frac{\mu}{2}m})^2  (w/z)^m
\right)~:\widetilde{V}^+(q^\mu z)\widetilde{V}^-(z)
\widetilde{V}^+(q^\mu w)\widetilde{V}^-(w):.
\end{align*}
\end{lem}
\begin{proof}
Apply the Wick's theorem with Lemma \ref{tVtV} (3).
The monomial factors coming from the normal ordering of the zero modes exactly cancel.
\end{proof}

\begin{lem}\label{quadratic-parafermi}
\begin{align*}
& \exp\left(
\sum_{m>0}{[2m]\over m[km]}
\left(-q^{-km}+q^{-\mu m}+q^{\mu m}-q^{km}\right)(w/z)^m
\right):e_{\epsilon_1}(q^\mu z)f_{\epsilon_2}( z):
:e_{\epsilon_3}(q^\mu w)f_{\epsilon_4}( w): \CR
&=
{q^{\epsilon_1}z-q^{\epsilon_3}w\over z-q^2 w} \cdot
{q^{-(k+1)\epsilon_1}q^{\mu}z-q^{\epsilon_4}w\over q^{-k\epsilon_1}q^{\mu}z-w} \cdot
{q^{(k+1)\epsilon_2}z-q^{-\epsilon_3}q^{\mu}w\over q^{k\epsilon_2}z-q^{\mu}w} \cdot
{q^{-\epsilon_2}z-q^{-\epsilon_4}w\over z-q^{-2} w} \CR
& \qquad \times :e_{\epsilon_1}(q^\mu z)f_{\epsilon_2}( z) e_{\epsilon_3}(q^\mu w)f_{\epsilon_4}( w):.
\end{align*}
\end{lem}
\begin{proof}
Apply the Wick's theorem with Propositions \ref{ee}, \ref{ff} and \ref{ef}.
The monomial factors coming from the normal ordering of the zero modes cancel in this case too.
\end{proof}

By Lemma \ref{quadratic-Cartan} with $\mu=-2$, we obtain
\begin{equation}\label{T-Cartan-OPE}
\mathbf{T}_{\mathrm{U(1)}}(z) \mathbf{T}_{\mathrm{U(1)}}(w) 
= \exp
\left(\sum_{m>0}{[(k+2)m]\over m[km]} ( q^{m} - q^{-m})^2  (w/z)^m
\right)~:\mathbf{T}_{\mathrm{U(1)}}(z) \mathbf{T}_{\mathrm{U(1)}}(w):.
\end{equation}

To compute the OPE coefficient of $\Lambda_1(z)$ and $\Lambda_1(w)$, we choose
$\epsilon_1=\epsilon_3=+1, \epsilon_2=\epsilon_4=-1$ and $\mu = -(k+2)$ in Lemma \ref{quadratic-parafermi}.
Then we obtain 
\begin{align}\label{PFstructure}
&{(1-q^2 z)(1-q^{-2} z)\over (1-z)^2}\cdot
\exp\left( \sum_{m>0}{[2m]\over m[km]}
\left(-q^{-km}+q^{(k+2) m}+q^{-(k+2) m}-q^{km}\right)z^m
\right)\CR
&= \widetilde{h}(z) = \exp\left(\sum_{m>0} \frac{[(k+2)m]}{m[km]} (q^m - q^{-m})^2 z^m\right).
\end{align}
It is interesting that $\widetilde{h}(z)$ is exactly the same as the OPE coefficient in \eqref{T-Cartan-OPE}. 
In fact this identity is an analogue of what we have seen in proving $G$-$T$ commutation relation (see \eqref{cancellation}).
Other OPE's among $\Lambda_i(z)$ are evaluated similarly by Lemma \ref{quadratic-parafermi}
and we find the following list.
\begin{lem}\label{LambdaOPE}
With the structure function \eqref{PFstructure} the OPE's among $\Lambda_i(z)$ are 
\begin{align*}
&\widetilde{h}(w/z)\Lambda_1(z)\Lambda_1(w)={~}:\Lambda_1(z)\Lambda_1(w):,\\
&\widetilde{h}(w/z)\Lambda_1(z)\Lambda_0(w)={~}
:\Lambda_1(z)\Lambda_0(w):{(1-q^{-2}w/z)(1-q^{2k+4}w/z)\over (1-w/z)(1-q^{2k+2}w/z)},\\
&\widetilde{h}(w/z)\Lambda_1(z)\Lambda_{-1}(w)={~}
:\Lambda_1(z)\Lambda_{-1}(w):{(1-q^{-2}w/z)(1-q^{2k+4}w/z)\over (1-w/z)(1-q^{2k+2}w/z)},
\end{align*}
\begin{align*}
&\widetilde{h}(w/z)\Lambda_0(z)\Lambda_1(w)={~}:\Lambda_0(z)\Lambda_1(w):
{(1-q^{2}w/z)(1-q^{-2k-4}w/z)\over (1-w/z)(1-q^{-2k-2}w/z)}, \\
&\widetilde{h}(w/z)\Lambda_0(z)\Lambda_0(w)={~}:\Lambda_0(z)\Lambda_0(w):
{(1-q^{-2k-4}w/z)(1-q^{2k+4}w/z)\over (1-q^{-2k-2}w/z)(1-q^{2k+2}w/z)}, \\
&\widetilde{h}(w/z)\Lambda_0(z)\Lambda_{-1}(w)={~}
:\Lambda_0(z)\Lambda_{-1}(w):{(1-q^{-2}w/z)(1-q^{2k+4}w/z)\over (1-w/z)(1-q^{2k+2}w/z)},
\end{align*}
\begin{align*}
&\widetilde{h}(w/z)\Lambda_{-1}(z)\Lambda_1(w)={~}:\Lambda_{-1}(z)\Lambda_1(w):
{(1-q^{2}w/z)(1-q^{-2k-4}w/z)\over (1-w/z)(1-q^{-2k-2}w/z)}, \\
&\widetilde{h}(w/z)\Lambda_{-1}(z)\Lambda_0(w)={~}:\Lambda_{-1}(z)\Lambda_0(w):
{(1-q^{2}w/z)(1-q^{-2k-4}w/z)\over (1-w/z)(1-q^{-2k-2}w/z)},\\
&\widetilde{h}(w/z)\Lambda_{-1}(z)\Lambda_{-1}(w)={~}:\Lambda_{-1}(z)\Lambda_{-1}(w):.
\end{align*}
\end{lem}

Summing up OPE relations in Lemma \ref{LambdaOPE} we have
\begin{prp}
\begin{align*}
&\widetilde{h}(w/z)\mathbf{T}_{\mathrm{PF}}(z)\mathbf{T}_{\mathrm{PF}}(w)
-\widetilde{h}(z/w)\mathbf{T}_{\mathrm{PF}}(w)\mathbf{T}_{\mathrm{PF}}(z)\\
=&-{(1-q^{-2})(1-q^{2k+4})\over 1-q^{2k+2}}
\left( \delta\left(q^{2k+2}{w\over z} \right)\mathbf{T}_{\mathrm{PF}}^{(2)}(w)
- \delta\left(q^{-2k-2}{w\over z} \right) \mathbf{T}_{\mathrm{PF}}^{(2)}(z)\right),
\end{align*}
where
\begin{align*}
&\mathbf{T}_{\mathrm{PF}}^{(2)}(z)
={~}:\Lambda_1(q^{2k+2}z)\Lambda_0(z):+
:\Lambda_1(q^{2k+2}z)\Lambda_{-1}(z):\\
&+{[k+1][2k+3]\over [k+2][2k+2]}:\Lambda_0(q^{2k+2}z)\Lambda_0(z):
+:\Lambda_0(q^{2k+2}z)\Lambda_{-1}(z):\\
=&{~}-q^{-1}{[k+2]\over [k+1]}:e_+(q^{-\frac{1}{2}(k+2)}z)f_-(q^{\frac{1}{2}(5k+6)}z):+:e_-(q^{-\frac{1}{2}(k+2)})f_-(q^{\frac{1}{2}(5k+6)}z):\\
&{~}+{[k+2][2k+3]\over [k+1][2k+2]}:e_+(q^{-\frac{1}{2}(k+2)}z)f_+(q^{\frac{1}{2}(5k+6)}z):-
q{[k+2]\over [k+1]}:e_-(q^{-\frac{1}{2}(k+2)}z)f_+(q^{\frac{1}{2}(5k+6)}z):.
\end{align*}
We have used the relation $:e_{+}(q^{k}z)f_{+}(z):{~}=1$ (See Prop. A.9).
\end{prp}
The quadratic relations for $\mathbf{T}_{\mathrm{PF}}^{(2)}(z)$ have been worked out in \cite{Kojima:2019ewe}
and there appears no additional current. It is not straightforward to obtain $\mathbf{T}_{\mathrm{PF}}^{(2)}(z)$
from the standard Miura transformation of the fundamental currents $\Lambda_i(x)$, since it involves 
a diagonal term $:\Lambda_0(q^{2k+2}z)\Lambda_0(z):$.

The commutation relation of $\mathbf{T}(z)$ is the product of the contribution from the $U(1)$ boson part and 
the parafermion part. We observe the structure functions from the $U(1)$ boson part and from the parafermion part 
cancel, which leads to a simple commutation relation of $\mathbf{T}(z)$;
\begin{align}
&\mathbf{T}(z)\mathbf{T}(w)-\mathbf{T}(w)\mathbf{T}(z) \CR
=&{~}-{(1-q^{-2})(1-q^{2k+4})\over 1-q^{2k+2}}
\left( \delta\left(q^{2k+2}{w\over z} \right):\mathbf{T}_{\mathrm{U(1)}}(q^{2k+2}w)\mathbf{T}_{\mathrm{U(1)}}(w): 
\mathbf{T}_{\mathrm{PF}}^{(2)}(w)\right. \CR
&{~}\left.- \delta\left(q^{-2k-2}{w\over z} \right):\mathbf{T}_{\mathrm{U(1)}}(q^{2k+2}z)\mathbf{T}_{\mathrm{U(1)}}(z):
\mathbf{T}_{\mathrm{PF}}^{(2)}(z)\right).
\end{align}

Now we have the following lemmas.

\begin{lem}
\begin{align*}
&:\mathbf{T}_{\mathrm{U(1)}}(q^{2k+2}z) \mathbf{T}_{\mathrm{U(1)}}(z):=
:\widetilde{V}^+(q^{2k+1}z) \widetilde{V}^-(q^{2k+3}z)
\widetilde{V}^+(q^{-1}z) \widetilde{V}^-(qz): \CR
&={~} :\widetilde{V}^-(q^{2k+3}z)
\widetilde{V}^+(q^{-1}z) \mathbf{K}(q^{k+1}z):.
\end{align*}
\end{lem}

\begin{proof}
Since we can exchange the operators $\widetilde{V}^{\pm}(z)$ in the normal ordered product,
it follows from the definitions of $\mathbf{T}_{\mathrm{U(1)}}(z)$ and $ \mathbf{K}(z)$. 
\end{proof}

\begin{lem}
\begin{equation*}
\mathbf{G}^+(z)\mathbf{G}^-(z) ={1\over (q-q^{-1})^3}{1\over [k+2]}
 :\widetilde{V}^+(q^{-k-2)}z) \widetilde{V}^-(q^{k+2}z) \mathbf{T}_{\mathrm{PF}}^{(2)}(q^{-k-1}z):.
\end{equation*}
\end{lem}

\begin{proof}
We have
\begin{align}
\mathbf{G}^+(z)\mathbf{G}^-(z) 
&= - \frac{z}{(q-q^{-1})^2}  \widetilde{V}^+(q^{-k-2}z) \widetilde{V}^-(q^{k+2}z) 
 \left( e_{+} (q^{-\frac{1}{2}(3k+4)}z)- e_{-} (q^{-\frac{1}{2}(3k+4)}z) \right) \CR
& \qquad \times \left( f_{+} (q^{\frac{1}{2}(3k+4)}z) - f_{-} (q^{\frac{1}{2}(3k+4)}z) \right).
\end{align}
The normal ordering of the zero modes of $\widetilde{V}^{\pm}(q^{\mp(k+2)}z)$  and 
those of $e_{\epsilon_1} (q^{-\frac{1}{2}(3k+4)}z )$ and $f_{\epsilon_2} (q^{\frac{1}{2}(3k+4)}z) $ gives
\begin{equation}
(q^{-k-2}z)^{-\frac{k+2}{k}} (q^{-\frac{1}{2}(3k+4)}z)^{\frac{2}{k}} = z^{-1} q^{k+1}.
\end{equation}
From the normal ordering of the oscillators we have the following factor which is independent of $\epsilon_1$ and $\epsilon_2$;
\begin{equation}
\exp \left( \sum_{m>0} \frac{[(k+2)m]}{m [km]} q^{(2k+4)m} \right)
\exp \left( - \sum_{m>0} \frac{[2m]}{m [km]} q^{(3k+4)m} \right) 
= (1-q^{2k+2})^{-1}.
\end{equation}
On the other hand the OPE factor that depends on  $\epsilon_1$ and $\epsilon_2$ is
\begin{equation}
\frac{q^{-\epsilon_1(k+1)} - q^{\epsilon_2 + 3k+4}}{q^{-\epsilon_1 k} - q^{3k+4}}.
\end{equation}
Combining these formulas, we obtain the desired result.
\end{proof}

The above two lemmas imply
\begin{lem}
\begin{align*}
&:\mathbf{T}_{\mathrm{U(1)}}(q^{2k+2}z) \mathbf{T}_{\mathrm{U(1)}}(z) 
\mathbf{T}_{\mathrm{PF}}^{(2)}(z): \CR
&= (q-q^{-1})^3 [k+2]~\mathbf{K}^{-}(q^{k+1}z)
\mathbf{G}^+(q^{k+1}z)\mathbf{G}^-(q^{k+1}z) \mathbf{K}^{+}(q^{k+1}z).
\end{align*}
\end{lem}

%
In summary we have
\begin{prp}
\begin{align*}
&\mathbf{T}(z)\mathbf{T}(w)-\mathbf{T}(w)\mathbf{T}(z)\\
=&-(q-q^{-1})^4{[k+2][k+2]\over [k+1]}
\left( \delta\left(q^{2k+2}{w\over z} \right)
\mathbf{K}^{-}(q^{k+1}w) \mathbf{G}^+(q^{k+1}w)\mathbf{G}^-(q^{-k-1}z) \mathbf{K}^{+}(q^{-k-1}z)
\right.\\
& \left.- \delta\left(q^{-2k-2}{w\over z} \right)
\mathbf{K}^{-}(q^{k+1}z) \mathbf{G}^+(q^{k+1}z)\mathbf{G}^-(q^{-k-1}w) \mathbf{K}^{+}(q^{-k-1}w)\right).
\end{align*}
\end{prp}
By Lemma \ref{GG=W} we finally obtain \eqref{rr-11}.


\section{Screening operators and the vanishing lines}
\label{app:screening}

For the Virasoro algebra the embedding structure
of the Fock modules derived from  the screening operators (BGG or BRST resolution)
plays a key role in the proof of the Kac determinant formula \cite{FF}, \cite{Felder}.
In this appendix as a first step to the proof of our conjecture 
of the Kac determinant of $\Svir$, we investigate the screening operators,
which are intertwines among the Fock representations of $\Svir$ obtained in Section \ref{Twist-Wakimoto}.
We will see that the vanishing lines predicted by the screening
operators exhaust the factors in the Kac determinant. 
Both the fermionic and the bosonic screening operators we employ are the same as those 
for $U_q(\widehat{\mathfrak{sl}}_2)$ \cite{Matsuo:1994nc}.
They involve only the modes $\overline{\alpha}_m, \beta_m, Q_{\overline{\alpha}}, Q_\beta$ in the parafermion sector 
and, hence, have trivial actions on the $U(1)$ sector.

\subsection{Degree operators $J$ and $d$}

Recall that we have introduced the Fock spaces $\mathcal{F}_{NS}(u,v)$ and $\mathcal{F}_{R}(u,v)$
with $u=q^{\rho}, v=q^\sigma$ in Subsection \ref{sec:Fock-rep}.
We consider the representations of $\Svir$ on the Fock spaces $\mathcal{F}_{A}(q^{\rho},q^\sigma)$, where $({\rm A}={\rm NS},{\rm R})$,
and the intertwiners (the screening charges) between these Fock representations.

\begin{dfn}
Set
\begin{align*}
&J={1\over k}(\widetilde{\alpha}_0+\overline{\alpha}_0),\\
&d={\widetilde{\alpha}_0^2 \over 2k(k+2)}-{\overline{\alpha}_0^2 \over 4k}+
{\beta_0(\beta_0+2)\over 4(k+2)}+
\sum_{m>0}{m^2\over [(k+2)m][km]}\widetilde{\alpha}_{-m}\widetilde{\alpha}_{m}\\
&\qquad -
\sum_{m>0}{m^2\over [2m][km]}\overline{\alpha}_{-m}\widetilde{\alpha}_{m}
+
\sum_{m>0}{m^2\over [2m][(k+2)m]}\beta_{-m}\beta_{m},\\
&d_0(\xi,\rho,\sigma)=
{\xi^2 \over 2k(k+2)}-{\rho^2 \over 4k}+
{\sigma(\sigma+2)\over 4(k+2)}.
\end{align*}
\end{dfn}

As we see from the lemme below, $d_0(\xi,\rho,\sigma)$ is the eigenvalue of the degree operator $d$ on the highest weight 
state $\vert \xi,\rho,\sigma\rangle$. 

\begin{lem}\label{J-d}
We have
\begin{align}
&J \vert \xi,\rho,\sigma\rangle={\xi-\rho\over k}  \vert \xi,\rho,\sigma\rangle,
\qquad 
 \langle \xi,\rho,\sigma\vert J={\xi-\rho\over k}  \langle \xi,\rho,\sigma\vert,\nonumber \\
&d \vert \xi,\rho,\sigma\rangle=d_0(\xi,\rho,\sigma)\vert \xi,\rho,\sigma\rangle,
\qquad \langle \xi,\rho,\sigma\vert d= d_0(\xi,\rho,\sigma)\langle \xi,\rho,\sigma\vert,\nonumber\\
&d_0(\rho+(k+2)n,\rho+2 n,\sigma)={n^2\over 2}+d_0(\rho,\rho,\sigma),\nonumber \\
&d_0({k/2}+\rho+(k+2)n,\rho+2 n,\sigma)={n(n+1)\over 2}+d_0({k/2}+\rho,\rho,\sigma),\nonumber\\
&d_0(\xi,\rho,\sigma)=
-1-{\mp \rho+\sigma\over 2}+
d_0(\xi,\rho\pm k,\sigma+k+2),
\label{d_0-S^pm-1}\\
&d_0(\xi,\rho,\sigma)=
{\mp \rho+\sigma\over 2}+
d_0(\xi,\rho\mp k,\sigma-k-2),\label{d_0-S^pm-2}\\
&d_0(\xi,\rho,\sigma)=
{r(\sigma-r+1)\over k+2}+
d_0(\xi,\rho,\sigma-2 r),\label{d_0-S^pm-3}\\
&
d_0(\xi,\rho,\sigma)=-{r(\sigma+r+1)\over k+2}+
d_0(\xi,\rho,\sigma+2 r).\label{d_0-S^pm-4}
\end{align}
\end{lem}

Set for simplicity that $\varepsilon_{\rm NS}=0$ for the NS sector and $\varepsilon_{\rm R}=k/2$ for the R sector.
We use the following notations for the states in the Fock space\footnote{Compare these notations with similar ones for the Verma module.}
\begin{align*}
&
\vert\lambda,\mu,\alpha,\beta;\rho,\sigma\rangle=
\mathbf{K}^-_{-\lambda} \mathbf{T}_{-\mu}\mathbf{G}^+_{-\alpha} \mathbf{G}^-_{-\alpha} 
\vert \varepsilon_{\rm A}+\rho,\rho,\sigma\rangle,\\
&
\langle \lambda,\mu,\alpha,\beta;\rho,\sigma\vert=
\langle \varepsilon_{\rm A}+\rho,\rho,\sigma\vert
\mathbf{G}^+_{\beta}\mathbf{G}^-_{\alpha} \mathbf{T}_{\mu}\mathbf{K}^+_{\lambda}.
\end{align*}
The reason why we should impose the condition $\xi=\varepsilon_{\rm A}+\rho$ on the highest weight vector
$\vert \xi,\rho,\sigma\rangle$ is explained in Subsection \ref{sec:Fock-rep}.
We define the degrees as follows.
\begin{center}
\begin{tabular}{c|c|c}
state& $d$-degree& $J$-degree\\[2mm]\hline
$ |\lambda,\mu,\alpha,\beta;\rho,\sigma\rangle$&
$d_0(\varepsilon_{\rm A}+\rho,\rho,\sigma)+|\lambda|+|\mu|+|\alpha|+|\beta|$& ${\varepsilon_{\rm A}/ k}+\ell(\alpha)-\ell(\beta)$\\[2mm]\hline
$\langle \lambda,\mu,\alpha,\beta;\rho,\sigma|$&$d_0(\varepsilon_{\rm A}+\rho,\rho,\sigma)+ |\lambda|+|\mu|+|\alpha|+|\beta|$
& ${\varepsilon_{\rm A}/ k}+\ell(\alpha)-\ell(\beta)$
\end{tabular}
\end{center}
We also use
\begin{center}
\begin{tabular}{c|c|c}
state& $p$-degree& $x$-degree\\[2mm]\hline
$ |\lambda,\mu,\alpha,\beta;\rho,\sigma\rangle$&
$|\lambda|+|\mu|+|\alpha|+|\beta|$& $\ell(\alpha)-\ell(\beta)$\\[2mm]\hline
$\langle \lambda,\mu,\alpha,\beta;\rho,\sigma|$&$|\lambda|+|\mu|+|\alpha|+|\beta|$
& $\ell(\alpha)-\ell(\beta)$
\end{tabular}
\end{center}
for convenience.

\subsection{Fermionic screening currents $S^\pm(z)$}

\begin{dfn}[\cite{Matsuo:1994nc}, \cite{DF1}]
Set 
\begin{align*}
&S^\pm(z)=
\exp\left( \sum_{m=1}^\infty {z^m\over [2m]} (q^{\pm {k\over 2}m} \beta_{-m}\pm  q^{\pm {k+2\over 2}m} \overline{\alpha}_{-m}) \right)\\
&\qquad \cdot 
\exp\left(-\sum_{m=1}^\infty {z^{-m}\over [2m]} (q^{\pm {k\over 2}m} \beta_{m}\pm  q^{\pm {k+2\over 2}m} \overline{\alpha}_{m})  \right)
e^{(k+2)Q_{\beta} \pm k Q_{\overline{\alpha}}}z^{{1\over 2 }(\beta_0\pm \overline{\alpha}_0)}.
\end{align*}
We call $S^\pm(z)$ the fermionic screening currents.
\end{dfn}

\begin{lem}
We have 
\begin{align*}
&[J,\mathbf{K}^\pm(z)]=0,\,\,
[J ,\mathbf{T}(z)]=0 ,\,\,
[J, \mathbf{G}^\pm(z)]=\pm \mathbf{G}^\pm(z),\,\,
[J, S^\pm(z)]=\mp S^\pm(z),\\
&q^d \mathbf{K}^\pm(z)=\mathbf{K}^\pm(qz)q^d ,\,\, 
q^d \mathbf{T}(z)=\mathbf{T}(qz)q^d ,\,\, 
q^d \mathbf{G}^\pm(z)=\mathbf{G}^\pm(qz)q^d ,\,\, 
q^d S^\pm(z)=qS^\pm(qz) q^d,\\
&S^\pm(z)S^\pm(w)=(z-w):S^\pm(z)S^\pm(w):, \\
& S^\pm(z)S^\mp(w)=z^{k+1}
{(q^{-k}w/z;q^4)_\infty (q^{-k+2}w/z;q^4)_\infty  \over (q^{k+4}w/z;q^4)_\infty (q^{k+2}w/z;q^4)_\infty }
:S^\pm(z)S^\mp(w):.
\end{align*}
\end{lem}

\begin{lem}
We have
\begin{align*}
&\mathbf{ G}^+_\epsilon(z)S^+(w)={1\over q^{-\kappa+\epsilon}z-w}:\mathbf{ G}^+_\pm(z)S^+(w):, \\
& S^+(w)\mathbf{ G}^+_\epsilon(z)={1\over w-q^{-\kappa+\epsilon}z}:S^+(w)\mathbf{ G}^+_\pm(z):,\\
&\mathbf{ G}^-_\epsilon(z)S^+(w)=( q^{\kappa+\epsilon(k+1)}z-w):\mathbf{ G}^-_\pm(z)S^+(w):, \\
& S^+(w)\mathbf{ G}^-_\epsilon(z)=(w-q^{\kappa+\epsilon(k+1)}z):S^+(w)\mathbf{ G}^-_\pm(z):,\\
&\mathbf{ G}^+_\epsilon(z)S^-(w)=( q^{-\kappa-\epsilon(k+1)}z-w):\mathbf{ G}^+_\pm(z)S^-(w):, \\
& S^-(w)\mathbf{ G}^+_\epsilon(z)=(w-q^{-\kappa-\epsilon(k+1)}z):S^-(w)\mathbf{ G}^+_\pm(z):,\\
&\mathbf{ G}^-_\epsilon(z)S^-(w)={1\over q^{\kappa-\epsilon}z-w}:\mathbf{ G}^-_\pm(z)S^-(w):, \\
& S^-(w)\mathbf{ G}^-_\epsilon(z)={1\over w-q^{\kappa-\epsilon}z}:S^-(w)\mathbf{ G}^-_\pm(z):,
\end{align*}
where $\kappa:=\frac{1}{2}(3k+4)$.\footnote{The shift of $q^\kappa$ comes from the definitions \eqref{G+-def}  and \eqref{G--def} of $\mathbf{ G}^\pm_\epsilon$.}
Under the condition that we have the Fourier expansion $S^\pm(z)=\sum_{n\in \mathbb{Z}} S^\pm_n z^{-n}$
(see Lemma \ref{screening-expansion} below),  
we have 
\begin{align*}
&\mathbf{ G}^+(z)S^+(w)+S^+(w)\mathbf{ G}^+(z)={A^+(z)\over (q-q^{-1})w} \Bigl(  \delta(q^{\kappa -1}w/z)
-  \delta(q^{\kappa +1}w/z)\Bigr),\\
&\mathbf{ G}^-(z)S^+(w)+S^+(w)\mathbf{ G}^-(z)=0,\quad \mathbf{ K}^\pm(z)S^+(w)-S^+(w)\mathbf{ K}^\pm(z)=0.
\end{align*}
and
\begin{align*}
&\mathbf{ G}^+(z)S^-(w)+S^-(w)\mathbf{ G}^+(z)=0,\quad \mathbf{ K}^\pm(z)S^-(w)-S^-(w)\mathbf{ K}^\pm(z)=0,\\
&\mathbf{ G}^-(z)S^-(w)+S^-(w)\mathbf{ G}^-(z)={A^-(z)\over (q-q^{-1})w} \Bigl(  \delta(q^{-\kappa-1}w/z)
-  \delta(q^{-\kappa +1}w/z)\Bigr),
\end{align*}
where 
\begin{align*}
A^+(z) \seteq &~z^{1/2}:\mathbf{ G}^+_+(z)S^+(q^{-\kappa +1}z):{~}
=~z^{1/2}:\mathbf{ G}^+_-(z)S^+(q^{-\kappa -1}z):, \\
A^-(z) \seteq &~z^{1/2}:\mathbf{ G}^-_+(z)S^-(q^{\kappa -1}z):{~}
=~z^{1/2}:\mathbf{ G}^-_-(z)S^-(q^{\kappa +1}z):.
\end{align*}

\end{lem}

\begin{lem}\label{screening-expansion}
The necessary and sufficient condition for having 
the Fourier expansion $S^\pm(z)=\sum_{n\in \mathbb{Z}} S^\pm_m z^{-m}$ on $ \mathcal{F}_{\rm A}(q^{\rho'},q^{\sigma'})$,
 $ \mathcal{F}^*_{\rm A}(q^{\rho'},q^{\sigma'})$
is  ${\mp \rho'+\sigma' \over2}\in \mathbb{Z}$.
Under this condition, we have
 the linear maps
\begin{align*}
& S^\pm_m : \mathcal{F}_{\rm A}(q^{\rho'},q^{\sigma'})\rightarrow\mathcal{F}_{\rm A}(q^\rho,q^\sigma):=
\mathcal{F}_{\rm A}(q^{\rho'\pm(k+2)},q^{\sigma'+k+2}),\\
&  S^\pm_m :  \mathcal{F}_{\rm A}^*(q^{\rho'},q^{\sigma'})\rightarrow\mathcal{F}_{\rm A}^*(q^\rho,q^\sigma):=
\mathcal{F}_{\rm A}^*(q^{\rho'\mp(k+2)},q^{\sigma'-k-2}).
\end{align*}

\end{lem}

\begin{proof}
If we have the Fourier expansion $S^\pm(z)=\sum_{n\in \mathbb{Z}} S^\pm_m z^{-m}$ on 
$ \mathcal{F}_{\rm A}(q^{\rho'},q^{\sigma'})$ or
 $ \mathcal{F}^*_{\rm A}(q^{\rho'},q^{\sigma'})$, we have
\begin{align*}
&\{\mbox{$d$-degrees of }S^\pm(z) \vert \varepsilon_{\rm A}+\rho', \rho',\sigma'\rangle\}\subset 
\{\mbox{$d$-degrees of }\mathcal{F}(q^{\varepsilon_{\rm A}+\rho'}, q^{\rho'\pm k},q^{\sigma'+k+2})\},\\
&
\{\mbox{$d$-degrees of }\langle\varepsilon_{\rm A}+ \rho', \rho',\sigma'\vert S^\pm(z)\}\subset  
\{\mbox{$d$-degrees of }\mathcal{F}^*(q^{\varepsilon_{\rm A}+\rho'}, q^{\rho'\mp k},q^{\sigma'-k-2})\}.
\end{align*}
Then from (\ref{d_0-S^pm-1}) and (\ref{d_0-S^pm-2}), we have the condition ${\mp \rho'+\sigma' \over2}\in \mathbb{Z}$.

Suppose that ${\mp \rho'+\sigma' \over2}\in \mathbb{Z}$. 
For any $p\in \mathbb{C}[\widetilde{\alpha}_{-1}, \ldots, \overline{\alpha}_{-1},\ldots,\beta{_{-1},\ldots}]$ 
and 
$p*\in\mathbb{C}[\widetilde{\alpha}_{1}, \ldots, \overline{\alpha}_{1},\ldots,\beta{_{1},\ldots}]$,
we have
\begin{align*}
&S^\pm(z) p\vert \varepsilon_{\rm A}+\rho'+(k+2)n', \rho'+2n',\sigma'\rangle\\
&\quad= z^{{\mp \rho'+\sigma' \over2}\mp n'}S_{osc}^\pm(z)p
  \vert \varepsilon_{\rm A}+\rho'+(k+2)n', \rho'+2n'\pm k,\sigma'+k+2\rangle,\\
  &\quad= z^{{\mp \rho+\sigma \over2}\mp n'}S_{osc}^\pm(z)p
  \vert \varepsilon_{\rm A}+\rho+(k+2)(n'\mp 1), \rho+2(n'\mp 1),\sigma\rangle,\\
&
\langle\varepsilon_{\rm A}+ \rho'+(k+2)n', \rho'+2n',\sigma'\vert p^*S^\pm(z)\\
&\quad =
 z^{-1+{\mp \rho'+\sigma' \over2}\mp n'} 
 \langle \varepsilon_{\rm A}+\rho'+(k+2)n', \rho'+2n'\mp k,\sigma'-k-2\vert p^* S_{osc}^{\pm}(z)\\
 &\quad =
 z^{-1+{\mp \rho+\sigma \over2}\mp n'} 
 \langle \varepsilon_{\rm A}+\rho+(k+2)(n'+1), \rho+2(n'+1),\sigma\vert p^* S_{osc}^{\pm}(z),
\end{align*}
where $S_{osc}^\pm(z) $  denotes the  oscillator part of $S^\pm(z) $.
Hence we have the expansion $S^\pm(z)=\sum_{n\in \mathbb{Z}} S^\pm_n z^{-n}$.
\end{proof}

\begin{prp}

Under the condition that we have the Fourier expansion $S^\pm(z)=\sum_{n\in \mathbb{Z}} S^\pm_n z^{-n}$,  
we have 
the fermionic screening charges
\begin{align*}
Q^\pm={1\over 2 \pi  \sqrt{-1}}\oint S^+(z)dz=S^\pm_1,
\end{align*}
satisfying
\begin{align*}
d Q^\pm=Q^\pm d,\quad
\mathbf{K}_mQ^\pm=Q^\pm\mathbf{K}_m  ,\quad \mathbf{T}_mQ^\pm=Q^\pm\mathbf{T}_m,\quad 
\mathbf{G}^\pm_mQ^\pm=-Q^\pm\mathbf{G}^\pm_m.
\end{align*}
Hence, if ${\mp \rho'+\sigma' \over2}={\mp \rho+\sigma \over2}\in \mathbb{Z}$, $Q^\pm$ 
is an intertwiner of the Fock representations of the algebra $\Svir$
\begin{align*}
& Q^\pm : \mathcal{F}_{\rm A}(q^{\rho'},q^{\sigma'})\rightarrow\mathcal{F}_{\rm A}(q^\rho,q^\sigma)=
\mathcal{F}_{\rm A}(q^{\rho'\pm(k+2)},q^{\sigma'+k+2}),\\
&  Q^\pm :  \mathcal{F}_{\rm A}^*(q^{\rho'},q^{\sigma'})\rightarrow\mathcal{F}_{\rm A}^*(q^\rho,q^\sigma)=
\mathcal{F}_{\rm A}^*(q^{\rho'\mp(k+2)},q^{\sigma'-k-2}),
\end{align*}
where ${\rm A}={\rm NS,R}$.

\end{prp}

Let us look at possible singular vectors obtained from the intertwiners $Q^\pm$ between Fock representations.
In the $\rm{NS}$ sector, if 
\begin{align*}
&d_0(\rho',\rho',\sigma')-d_0(\rho',\rho'\pm k,\sigma'+k+2)=
-1-{\mp \rho+\sigma\over 2}=m\in \mathbb{Z}_{\geq 0}, \\
& \rho'=\rho\mp (k+2), \qquad \sigma'=\sigma-(k+2),
\end{align*}
then we have the non-vanishing image of the highest weight vector $\vert \rho', \rho',\sigma'\rangle \in 
\mathcal{F}_{\rm NS}(q^{\rho'},q^{\sigma'})$;
\begin{align*}
0\neq Q^\pm \vert \rho', \rho',\sigma'\rangle \in 
\mathcal{F}(\rho \mp (k+2),\rho \mp 2 ,\sigma)\subset
\mathcal{F}_{\rm NS}(q^{\rho},q^{\sigma}),
\end{align*}
providing us with a singular vector in $\mathcal{F}_{\rm NS}(q^{\rho},q^{\sigma})$,
which has the $p$-degree $m+{1\over2}$ and the $x$-degree $\mp 1$.
On the other hand, for the dual Fock space, if 
\begin{align*}
&d_0(\rho',\rho',\sigma')-d_0(\rho',\rho'\mp k,\sigma'-k-2)=
+{\mp \rho+\sigma\over 2}=m\in \mathbb{Z}_{\geq 0},\\
&  \rho'=\rho\pm (k+2), \qquad \sigma'=\sigma+(k+2),
\end{align*}
then we have the non-vanishing image of the highest weight vector $\langle \rho', \rho',\sigma'\vert \in 
\mathcal{F}^*_{\rm NS}(q^{\rho'},q^{\sigma'})$;
\begin{align*}
0\neq \langle \rho', \rho',\sigma'\vert Q^\pm \in 
\mathcal{F}^*(\rho \pm (k+2),\rho \pm 2 ,\sigma)\subset
\mathcal{F}^*_{\rm NS}(q^{\rho},q^{\sigma}),
\end{align*}
providing us with a singular vector in $\mathcal{F}^*_{\rm NS}(q^{\rho},q^{\sigma})$,
which has the $p$-degree $m+{1\over2}$ and the $x$-degree $\pm 1$.
These  arguments show the vanishing line
\begin{align*}
g(\pm(2m+1);u,v)\propto
\left(q^{m +1} u^{\pm {1\over 2}}  v^{{1\over 2}} - q^{-m-1} u^{\mp{1\over 2}}  v^{-{1\over 2}} \right)
\left(q^{-m } u^{\mp {1\over 2}}  v^{{1\over 2}} - q^{m } u^{\pm{1\over 2}}  v^{-{1\over 2}} \right),
\end{align*}
in the Kac determinant  ${\rm det}^{\rm NS}_{m+{1\over 2},\pm 1}$.

Let us turn to the $\rm{R}$ sector. If 
\begin{align*}
&d_0(k/2+\rho',\rho',\sigma')-d_0(k/2+\rho',\rho'\pm k,\sigma'+k+2)=
-1-{\mp \rho+\sigma\over 2}=m\in \mathbb{Z}_{\geq 0}, \\
&\quad \rho'=\rho\mp (k+2), \qquad \sigma'=\sigma-(k+2),
\end{align*}
then we have the non-vanishing image of the highest weight vector $\vert k/2+\rho', \rho',\sigma'\rangle \in 
\mathcal{F}_{\rm R}(q^{\rho'},q^{\sigma'})$
\begin{align*}
0\neq Q^\pm \vert k/2+ \rho', \rho',\sigma'\rangle \in 
\mathcal{F}(q^{k/2+\rho \mp (k+2)},q^{\rho \mp 2} ,q^{\sigma})\subset
\mathcal{F}_{\rm R}(q^{\rho},q^{\sigma}),
\end{align*}
providing us with a singular vector in $\mathcal{F}_{\rm R}(q^{\rho},q^{\sigma})$,
which has the $p$-degree $m+(1\mp1)/2$ and the $x$-degree $\mp 1$.
Similarly for the dual Fock space, if 
\begin{align*}
&d_0(k/2+\rho',\rho',\sigma')-d_0(k/2+\rho',\rho'\mp k,\sigma'-k-2)=
+{\mp \rho+\sigma\over 2}=m\in \mathbb{Z}_{\geq 0}, \\
& \rho'=\rho\mp (k+2), \qquad \sigma'=\sigma-(k+2),
\end{align*}
then we have the non-vanishing image of the highest weight vector $\langle k/2+\rho', \rho',\sigma'\vert \in 
\mathcal{F}^*_{\rm R}(q^{\rho'},q^{\sigma'})$
\begin{align*}
0\neq \langle k/2+ \rho', \rho',\sigma'\vert Q^\pm \in 
\mathcal{F}^*(q^{k/2+\rho \pm (k+2)},q^{\rho \pm 2} ,q^{\sigma})\subset
\mathcal{F}^*_{\rm R}(q^{\rho},q^{\sigma}),
\end{align*}
providing us with a singular vector in $\mathcal{F}^*_{\rm R}(q^{\rho},q^{\sigma})$,
which has the $p$-degree $m+(1\pm 1)/2$ and the $x$-degree $\pm 1$.
Thus, we find the vanishing line
\begin{align*}
g(\pm(2m+1);u,v)\propto
\left(q^{m +1} u^{\pm {1\over 2}}  v^{{1\over 2}} - q^{-m-1} u^{\mp{1\over 2}}  v^{-{1\over 2}} \right)
\left(q^{-m } u^{\mp {1\over 2}}  v^{{1\over 2}} - q^{m } u^{\pm{1\over 2}}  v^{-{1\over 2}} \right),
\end{align*}
in the Kac determinant ${\rm det}^{\rm R}_{m+(1\pm 1)/2,\pm 1}$.

\subsection{Bosonic screening current $\mathbf{S}(z)$}

Write $p=q^{2(k+2)}$ for simplicity, and assume that $|p|<1$. 
We use the  modified Jacobi theta function $\theta(z;p)=\theta(z)$
with the argument $z$ and the nome $p$ 
\begin{align}
&\theta(z;p)=\sum_{n\in \mathbb{Z}} (-1)^n p^{n(n-1)/2} z^n=(z;p)_\infty (p/z;p)_\infty (p;p)_\infty, \\
&\theta(pz;p)=-z^{-1} \theta(z;p). \label{theta-pshift}
\end{align}
\begin{dfn}[\cite{Matsuo:1994nc}]
Set
\begin{align*}
&\mathbf{S}_\pm(z)= {~}: U(z) Z_\pm(q^{\mp{k+2\over 2}}z )^{-1} 
W_\pm(q^{\mp{k\over 2}}z )^{-1} :,\\
&U(z)=
\exp\left( -\sum_{m=1}^\infty {z^m q^{- {k+2\over 2}m}\over [(k+2)m]} \beta_{-m} \right)
\exp\left(\sum_{m=1}^\infty {z^{-m}q^{- {k+2\over 2}m} \over [(k+2)m]}\beta_{m} \right)
e^{-2Q_{\beta}}
z^{-{1\over k+2 }\beta_0},\\
&\mathbf{S}(z)={-1\over (q-q^{-1})z} (\mathbf{S}_+(z)-\mathbf{S}_-(z)).
\end{align*}
We call $\mathbf{S}(z)$ the bosonic screening current.
\end{dfn}

\begin{lem}\label{S-G,G-S}
We have
\begin{align*}
&
q^d\mathbf{S}_\pm(w)=\mathbf{S}_\pm(qw) q^d, \quad
q^d\mathbf{S}(w)=q\mathbf{S}(qw) q^d, \quad [J,\mathbf{S}(w)]=0,\\
&\mathbf{S}_{\epsilon_1}(w) \mathbf{G}^+_{\epsilon_2}(z)=
{q^{-\epsilon_1}w  -q^{\epsilon_2-\kappa}z \over w- q^{-\kappa}z} :\mathbf{S}_{\epsilon_1}(w) \mathbf{G}^+_{\epsilon_2}(z):,\\
& \mathbf{G}^+_{\epsilon_2}(z)\mathbf{S}_{\epsilon_1}(w)=
{q^{\epsilon_2-\kappa}z -q^{-\epsilon_1}w  \over q^{-\kappa}z-w} : \mathbf{G}^+_{\epsilon_2}(z)\mathbf{S}_{\epsilon_1}(w):,\\
&\mathbf{S}_{\epsilon_1}(w) \mathbf{G}^-_{\epsilon_2}(z)=
{q^{\epsilon_1}w  -q^{(k+1)\epsilon_2+\kappa}z \over w-q^{(k+2)\epsilon_2+\kappa}z} :\mathbf{S}_{\epsilon_1}(w) \mathbf{G}^-_{\epsilon_2}(z):,\\
& \mathbf{G}^-_{\epsilon_2}(z)\mathbf{S}_{\epsilon_1}(w)=
{q^{(k+1)\epsilon_2+\kappa}z-q^{\epsilon_1}w   \over q^{(k+2)\epsilon_2+\kappa}z-w}: \mathbf{G}^-_{\epsilon_2}(z)\mathbf{S}_{\epsilon_1}(w):,
\end{align*}
where $\kappa:= \frac{1}{2}(3k+4)$, and
\begin{align*}
\mathbf{S}_{\epsilon_1}(w_1) \mathbf{S}_{\epsilon_2}(w_2)=
w_1^{2\over k+2} 
{(q^{-2} w_2/w_1;q^{2(k+2)})_\infty \over (q^{2} w_2/w_1;q^{2(k+2)})_\infty}
{q^{-\epsilon_1}w_1-q^{-\epsilon_2}w_2\over w_1-q^{-2}w_2} 
:\mathbf{S}_{\epsilon_1}(w_1) \mathbf{S}_{\epsilon_2}(w_2):.
\end{align*}

\end{lem}

\begin{prp}
We have
\begin{align*}
\mathbf{S}(w_2) \mathbf{S}(w_1)=q^2 \left( w_2\over w_1\right)^{2\over k+2} 
 {\theta(q^{-2} w_1/w_2;p) \over \theta(q^{+2} w_1/w_2;p) }\mathbf{S}(w_1) \mathbf{S}(w_2) .
\end{align*}
In the {\rm A$=$NS,R} sector, 
we have
\begin{align*}
&[\mathbf{S}(w),\mathbf{G}^+(z)]=0,\quad [\mathbf{S}(w),\mathbf{K}^\pm(z)]=0,\\
&[\mathbf{S}(w),\mathbf{G}^-(z)]=
{1\over (q-q^{-1})w}
\Biggl(
\delta^{\rm A}\left({q^\kappa z\over q^{k+2} w} \right)
{\bf A}(q^{k+2} w)
-
\delta^{\rm A}\left(q^{k+2+\kappa}z\over w\right){\bf A}(q^{-k-2} w)
\Biggr),
\end{align*}
where 
\begin{align*}
{\bf A}(w)=& :\mathbf{S}_{-}(q^{-k-2}w) w^{1/2}\mathbf{G}^-_{-}(q^{-\kappa} w):{~}={~}
:\mathbf{S}_{+}(q^{k+2}w) w^{1/2}\mathbf{G}^-_{+}(q^{-\kappa} w): \\
=& w^{1/2} Y^{-}(w) 
\exp\left( -\sum_{m=1}^\infty {w^m q^{{k+2\over 2}m}\over [(k+2)m]} \beta_{-m} \right)
\exp\left(\sum_{m=1}^\infty {w^{-m}q^{{k+2\over 2}m} \over [(k+2)m]}\beta_{m} \right)
e^{-2Q_{\beta}}.
\end{align*}

\end{prp}

\begin{proof}
Recall that in the A$=$NS,R sector, we have
$\delta(z/w) \mathbf{G}^+(z)=\delta^{\rm A} (z/w) \mathbf{G}^+(w)$, 
$\delta(z/w) z^{1/2}\mathbf{G}^+_\pm(z)=\delta^{\rm A} (z/w) w^{1/2}\mathbf{G}^+_\pm (w)$.
From Lemma \ref{S-G,G-S}, we have
\begin{align*}
&[\mathbf{S}(w),\mathbf{G}^+(z)]=
{-1\over (q-q^{-1})w }{z^{1/2}\over (q-q^{-1})}
\sum_{\epsilon_1,\epsilon_2=\pm}[\mathbf{S}_{\epsilon_1}(w),\mathbf{G}^+_{\epsilon_2}(z)]\\
&=
{-1\over (q-q^{-1})w }{z^{1/2}\over (q-q^{-1})}
\sum_{\epsilon_1,\epsilon_2=\pm}
(q^{-\epsilon_1}-q^{\epsilon_2})\delta(q^{-\kappa} z/w)
:\mathbf{S}_{\epsilon_1}(w)\mathbf{G}^+_{\epsilon_2}(z):\\
&=
{1\over (q-q^{-1})w}
\delta^{\rm A}(q^{-\kappa} z/w) \left(
:\mathbf{S}_{+}(w)w^{1/2}\mathbf{G}^+_{+}(q^\kappa w):-
:\mathbf{S}_{-}(w)w^{1/2}\mathbf{G}^+_{-}(q^\kappa w): \right)=0,\\
&[\mathbf{S}(w),\mathbf{G}^-(z)]=
{-1\over (q-q^{-1})w }{z^{1/2}\over (q-q^{-1})}
\sum_{\epsilon_1,\epsilon_2=\pm}[\mathbf{S}_{\epsilon_1}(w),\mathbf{G}^-_{\epsilon_2}(z)]\\
&=
{-1\over (q-q^{-1})w }{z^{1/2}\over (q-q^{-1})}
\sum_{\epsilon_1,\epsilon_2=\pm}
(q^{\epsilon_1}-q^{-\epsilon_2})\delta(q^{(k+2)\epsilon_2+\kappa}z/w)
:\mathbf{S}_{\epsilon_1}(w)\mathbf{G}^-_{\epsilon_2}(z):\\
&=
{1\over (q-q^{-1})w}
\Biggl(
\delta^{\rm A}\left({q^{\kappa} z\over q^{k+2} w} \right)
:\mathbf{S}_{-}(w) (q^{k+2} w)^{1/2}\mathbf{G}^-_{-}(q^{k+2-\kappa} w):\\
&\qquad \qquad\qquad 
-
\delta^{\rm A}\left(q^{k+2+\kappa}z\over w\right)
:\mathbf{S}_{+}(w)(q^{-k-2} w)^{1/2}\mathbf{G}^-_{+}(q^{-k-2-\kappa} w):
\Biggr)\\
&=
{1\over (q-q^{-1})w}
\Biggl(
\delta^{\rm A}\left({q^\kappa z\over q^{k+2} w} \right)
{\bf A}(q^{k+2} w)
-
\delta^{\rm A}\left(q^{k+2+\kappa}z\over w\right){\bf A}(q^{-k-2} w)
\Biggr).
\end{align*}
\end{proof}

Introducing a $p$-shift invariant function
$w^{\beta_0\over k+2}{\theta(q^{2 \beta_0} w;p)/ \theta( w;p)}$,
we set
$$
{\bf X}(w)=\mathbf{S}(w) w^{\beta_0\over k+2}{\theta(q^{2 \beta_0} w;p)\over \theta( w;p)}.
$$
Since ${\bf X}(w)$ is single valued in $w$, we have the Laurent series expansion of ${\bf X}(w)$ 
defined on the annulus $|p|<w<1$, and the contour integral 
along the circle $C:|w|=|p^{1/2}|$
\begin{align*}
\oint_{C} {dw} {\bf X}(w)
\end{align*}
is well-defined. 

\begin{prp}[\cite{Jimbo:1996vu}]
For $r=1,2,3,\ldots$, 
set
\begin{align*}
&{\bf Q}_r=
\left( \oint_{C} {dw} {\bf X}(w)\right)^r.
\end{align*}
Then, after the symmetrization of the integrand, we have
\begin{align}\label{symmetrization}
&{\bf Q}_r=
{\prod_{i=1}^r \theta(q^{2i};p)\over  r! \cdot \theta(q^2;p)^r}
\oint_{C} {dw_1}\cdots \oint_{C} {dw_r}
\mathbf{S}(w_1)\cdots \mathbf{S}(w_r)
\prod_{i=1}^r  w_i^{\beta_0-2(r-i)\over k+2}\CR
&\cdot \prod_{1\leq i< j\leq r}{ \theta(w_i/w_j;p) \over  \theta(q^2 w_i/w_j;p)}
\cdot \prod_{i=1}^r {\theta(p^{ \beta_0-(r-1)\over k+2}w_i;p)  \over  \theta(w_i;p)}.
\end{align}
\end{prp}

Following the arguments in \cite{Jimbo:1996vu}, we can see that 
under the condition \eqref{integrality-1} or \eqref{integrality-2} below, the quasi-periodicity \eqref{theta-pshift}
implies a cancellation of the poles of the theta functions in the last factor of \eqref{symmetrization},
which allows us to shift the contour $C$ in the suitable direction. 
Taking the $p$-shift invariance into account, we obtain the following results;

\begin{cor}
Set
$$
\rho=\rho', \qquad \sigma=\sigma'-2 r.
$$
When 
\begin{equation}\label{integrality-1}
{\sigma'-r+1 \over k+2}={\sigma+r+1 \over k+2}=s\in \mathbb{Z},
\end{equation}
${\bf Q}_r$ is an intertwiner of the Fock representations of $\Svir$ 
\begin{align*}
{\bf Q}_r:\mathcal{F}_{\rm A}(q^{\rho'},q^{\sigma'})
\rightarrow \mathcal{F}_{\rm A}(q^{\rho},q^{\sigma}).
\end{align*}
Similarly, set
$$
\rho=\rho', \qquad \sigma=\sigma'+2 r.
$$
When 
\begin{equation}\label{integrality-2}
{\sigma'+r+1\over k+2}={\sigma-r+1\over k+2}=-s\in \mathbb{Z},
\end{equation}
${\bf Q}_r$ is an intertwiner of the dual Fock representations of $\Svir$ 
\begin{align*}
{\bf Q}_r:\mathcal{F}^*_{\rm A}(q^{\rho'},q^{\sigma'})
\rightarrow \mathcal{F}^*_{\rm A}(q^{\rho},q^{\sigma}).
\end{align*}

\end{cor}

As in the case of the fermionic screening operators, 
we can obtain the vanishing lines in the Kac determinant as follows;

Firstly, if 
\begin{align*}
d_0(\varepsilon^{\rm A}+\rho',\rho',\sigma') -
d_0(\varepsilon^{\rm A}+\rho',\rho',\sigma'-2r)=
{r(\sigma'-r+1)\over k+2}=rs>0,
\end{align*}
namely if 
\begin{align*}
v=q^\sigma=q^{-r-1+(k+2)s},\qquad s>0,
\end{align*}
there is a singular vector 
in $\mathcal{F}_{\rm A}(q^{\rho},q^{\sigma})$ with the $p$-degree $rs$ and $x$-degree $0$
given as the image of ${\bf Q}_r $ as
\begin{align*}
{\bf Q}_r \vert \varepsilon^{\rm A}+\rho',\rho',\sigma'\rangle 
\in \mathcal{F}( \varepsilon^{\rm A}+\rho,\rho,\sigma) \subset \mathcal{F}_{\rm A}(q^{\rho},q^{\sigma}).
\end{align*}

Secondly, on the dual side, if 
\begin{align*}
d_0(\varepsilon^{\rm A}+\rho',\rho',\sigma') -
d_0(\varepsilon^{\rm A}+\rho',\rho',\sigma'+2r)=
-{r(\sigma'+r+1)\over k+2}=rs>0,
\end{align*}
namely if 
\begin{align*}
v=q^\sigma=q^{r-1-(k+2)s},\qquad s>0,
\end{align*}
there is  a singular vector 
in $\mathcal{F}_{\rm A}^*(q^{\rho},q^{\sigma})$ with the $p$-degree $rs$ and $x$-degree $0$
given as the image of ${\bf Q}_r $ as
\begin{align*}
 \langle \varepsilon^{\rm A}+\rho',\rho',\sigma'\vert {\bf Q}_r
\in \mathcal{F}^*( \varepsilon^{\rm A}+\rho,\rho,\sigma) \subset \mathcal{F}_{\rm A}^*(q^{\rho},q^{\sigma}).
\end{align*}

In summary, we have the vanishing line
\begin{align*}
f(r,s;u,v) \propto
(q^{1-r +(k+2) s} v-q^{-1+r -(k+2) s} v^{-1})
(q^{-1-r +(k+2) s} v^{-1}-q^{1+r -(k+2) s} v).
\end{align*}
in the Kac determinant ${\rm det}^{\rm A}_{rs,0} $.



\end{document}